\documentclass[10pt,a4paper,leqno]{article}

\usepackage{a4wide}

\usepackage[utf8]{inputenc}
\usepackage{amsmath,amsfonts,amssymb,amsthm,hyperref}
\usepackage{cite}
\allowdisplaybreaks
\usepackage{algorithm}
\usepackage{algorithmic}
\usepackage{nicefrac}
\usepackage{graphicx}
\usepackage{subcaption}
\usepackage[format=hang,labelfont=sc,font=small]{caption}
\graphicspath{{figures/}}
\DeclareGraphicsExtensions{.pdf}

\numberwithin{equation}{section}
\numberwithin{figure}{section}

\usepackage{color}
\newcommand{\revision}[1]{\textcolor{black}{#1}}

\newcommand{\norm}[1]{\left\lVert#1\right\rVert}
\newcommand\abs[1]{\left|#1\right|}
\renewcommand{\vec}[1]{\mathbf{#1}}
\newcommand\vecsym[1]{\boldsymbol{#1}}
\renewcommand{\d}{\mathrm{d}}
\newcommand\N{\mathcal{N}}
\newcommand\R{\mathbb{R}}

\newcommand\Rdone{\mathbb{R}^{d_1}}
\newcommand\Rdtwo{\mathbb{R}^{d_2}}

\newcommand\Rnn{\mathbb{R}^{n\times n}}

\newcommand\Rdonedone{\mathbb{R}^{d_1\times d_1}}

\newtheorem{remark}{Remark}[section]

\newtheorem{lemma}[remark]{Lemma}
\newtheorem{corollary}[remark]{Corollary}

\newcommand{\keywords}{{\bf Keywords.}\ }
\newcommand{\subclass}{{\bf MSC2010.}\ }

\title{Mean field models for large data--clustering problems}
   
\author{M. Herty 
	\thanks{
		Institut f\"{u}r Geometrie und Praktische Mathematik -
		RWTH Aachen University --
		Templergraben 55, 52062 Aachen, Germany --
		{\sl herty@igpm.rwth-aachen.de}
	}
	\and 
	L. Pareschi 
	\thanks{
		Mathematics and Computer Science Department -
		University of Ferrara --
		Via Machiavelli 35,
		44121 Ferrara, Italy --
		{\sl lorenzo.pareschi@unife.it}
	}
	\and
	G. Visconti
	\thanks{
		Institut f\"{u}r Geometrie und Praktische Mathematik -
		RWTH Aachen University --
		Templergraben 55, 52062 Aachen, Germany --
		{\sl visconti@igpm.rwth-aachen.de}
	}
}

\date{\today}

\begin{document}

\maketitle    

\begin{abstract}
	We consider mean-field models for data--clustering problems starting from a generalization of the bounded confidence model for opinion dynamics. The microscopic model includes information on the position as well as on additional features of the particles in order to develop specific clustering effects. The corresponding mean--field limit is derived and properties of the model are investigated analytically. In particular, the mean--field formulation allows the use of a random subsets algorithm for efficient computations of the clusters. Applications to shape detection and image segmentation on standard test images are presented and discussed. 
\end{abstract}


\smallskip
\subclass{82C40, 94A08, 68U10}


\section{Introduction} \label{sec:introduction}
Particle and kinetic models for consensus and cluster formation appeared in recent literature for self--organized socio--economic dynamical systems as opinion formation, flocking of birds or fish, elections and referendums under influence of mass media, etc. See e.g.~\cite{AlbiPareschiZanella2017,2016BoudinSalvarani,CouzinKrauseFranksLevin2005,CuckerSmale2007a,DegondMotsch2011,DuringMarkowichPietschmannWolfram2009,HegselmannKrause2002,MotschTadmor2011,CarrilloFornasierToscaniVecil2010}, the review articles~\cite{MotschTadmor2014, AlbiPareschiToscaniZanella2017, AlbiBellomoPareschietal2019} and the book\cite{PareschiToscaniBOOK}.
A related research direction is based on using the consensus features of these models in an artificial way to solve problems of optimization or segmentation of data in large dimensions \cite{KE95,PTTM17, CCTT18, 2016Liuetal}.

In this paper, we aim at formulating suitable models on the microscopic (or particle) level as well as on the mean--field (or kinetic) level to describe the 
partition of a large set of data in clusters. This problem is also known as data clustering problem and it is widely studied in many applications like pattern recognition, shape detection and image segmentation problems. 
The proposed methods do not need to have fixed a--priori number of clusters and clusters are characterized by small in--group and large out--group distances.

Since we are interested in modes characterizing distance and qualitative features without additional physical or socio--economical modeling background the proposed model will generalize the Hegselmann--Krause (HK) opinion dynamics model~\cite{HegselmannKrause2002}. Originally, the HK model was proposed in a microscopic setting, in one--dimensional spatial and time discrete framework. Several extensions exist, in the sequel we will briefly review the basic model before discussing its extension towards the image clustering problems. 

To this aim, let us consider a group of $n$ particles with a (scalar) initial state $x_i(0)\in\R$, $i=1,\dots,n$, and the state of each particle varies depending on the state of the others. The key idea of the Hegselmann--Krause (HK) model~\cite{HegselmannKrause2002} is that particles with completely different opinions do not influence each other, and a sort of mediation occurs among agents whose opinions are within a bounded confidence interval described by a parameter $\epsilon\geq0.$  Let $\vec{x}(t) = [x_1(t),\dots,x_n(t)]^T$ be the state of the system  at time $t\geq 0$. Then  the dynamic of the $i$--th particle  is given by 
\begin{equation} \label{eq:hkModel1D}
	\frac{\d}{\d t} x_i(t) = \sum_{j=1}^n A_{ij}(t,\epsilon) \left( x_j(t) - x_i(t) \right), \quad i=1,\dots,n
\end{equation}
where $\vec{A}(t,\epsilon)\in\Rnn$ is the time--varying adjacency matrix whose entries are in the form
\revision{
\begin{equation} \label{eq:adjacencyMatrix}
	A_{ij}(t,\epsilon) := \begin{cases}
		\displaystyle{\frac{1}{\sigma_i}}, & \text{if $j\in\N_i(t,\epsilon)$}\\[2ex]
		0, & \text{otherwise}
	\end{cases}
\end{equation}
with
\begin{equation} \label{eq:interactionDomain}
	\N_i(t,\epsilon) := \left\{ j\in\{1,\dots,n\} : \abs{x_i(t)-x_j(t)}\leq\epsilon \right\}, \quad i=1,\dots,n
\end{equation}
defining the neighborhood of the $i$--th particle at time $t$, and
\begin{equation} \label{eq:sigma}
	\sigma_i := \begin{cases}
	\displaystyle{\abs{\N_i(t,\epsilon)}},\\[2ex]
	n
	\end{cases}
\end{equation}
the type of interactions. Precisely, when $\sigma_i=n$ the interactions are symmetric since $A_{ij} = A_{ji}$, $\forall\,i,j$, but $\vec{A}$ is not a stochastic matrix. Instead, when $\sigma_i=\abs{\N_i(t,\epsilon)}$, then $\vec{A}$ is a right stochastic matrix but interactions are no longer symmetric.
}


Several works have been proposed in the literature which analyze the properties of the HK model. For instance, for the analysis in the time discrete setting we refer to~\cite{JabinMotsch2014}. 
In~\cite{BlondelHendrickxTsitsiklis2009} it is proven that,
during the evolution of the system~\eqref{eq:hkModel1D}, the order of the states is preserved. Thanks to the definition of the \revision{interaction kernel~\eqref{eq:adjacencyMatrix}--\eqref{eq:sigma}},  if $\abs{x_i (t) - x_{i+1} (t)} > \epsilon$, at some time $t$, it  remains true for larger times. 
Therefore, the HK model tends to group the initial states in a finite number of clusters as proved in~\cite{MotschTadmor2014}. Following~\cite{MotschTadmor2014}, we define a cluster \revision{$\mathcal{C}(t)$ at time $t\geq 0$} a subset of particles separated from all the other particles
$$ A_{ij}(t,\epsilon) \neq 0 \ \text{ for all } \ i,j\in\mathcal{C}\revision{(t)}, \quad A_{ij}(t,\epsilon) = 0 \ \text{ whenever } \ i\in\mathcal{C}\revision{(t)}, j\notin\mathcal{C}\revision{(t)}.$$
In~\cite{BlondelHendrickxTsitsiklis2009,Dittmer2001,Lorenz2005} the stability of
the dynamical model is investigated. In particular, the fact that the system converges to a steady profile in finite time is proved in~\cite{BlondelHendrickxTsitsiklis2009}. For further results on the one--dimensional local and symmetric model we refer also to~\cite{BlondelHendrickxTsitsiklis2010}. We point out that also behavior of cluster formation in the transient is of interest in the mathematical literature~\cite{2016DietrichMartinJungers}.

In~\cite{NedicTouri2012,2010BMS}, the one-dimensional Hegselmann--Krause is generalized to the case of a multi--dimensional data--set. Subsequently, the multi--dimensional HK model has been used as a technique to cluster a big amount of data into a small number of subsets with some common features~\cite{Olivaetal2015} and to compare its performance with the $k$--means algorithm~\cite{MacQueen1967}. Recently, in~\cite{2016Liuetal,2014NovikovBenderskaya} new approaches to clustering problems and image segmentation have been proposed based on the Kuramoto model.

Here, we introduce a generalization of the multi--dimensional formulation of the HK model to solve data clustering problems. This amounts to  take into account clustering with respect to different features. We deal with data having both time dependent and static features. The latter describe intrinsic properties of a datum, such as the measure of a trustworthy information or the color intensity of pixels in images. We derive the corresponding mean--field limit and investigate analytically the properties of the model. In particular, following\cite{AlbiPareschi2013} the mean--field formulation allows the use of a random subset algorithm for efficient computations of the clusters.  

The rest of the manuscript is organized as follows. The microscopic model is introduced at the beginning of Section~\ref{sec:newModel} and briefly discussed in Section~\ref{sec:microProp}.
The proposed model is still a microscopic model and we describe the case of large data using a mean--field equation in Section~\ref{sec:mfModel}. 
Analytical properties of the kinetic equation, such as a--priori estimation on the evolution of the moments and characterization of the limit distribution, are discussed in Section~\ref{sec:kineticProp}. Numerical evidence of the theoretical results is provided in Section~\ref{sec:numericalProp}.
Further, we propose applications on detection and compression of data, such as shape detection, in Section~\ref{sec:shapeDetect}, and image segmentation, in Section~\ref{sec:imageSeg}.
We finally conclude with some remarks and future research directions in Section~\ref{sec:conclusion}.

\section{Microscopic models for data--clustering} \label{sec:newModel}

Each particle $i=1,\dots,N$ is endowed with a time--dependent state vector $\vec{x}_i=\vec{x}_i(t)$ as well as
features $\vec{c}_i = [c_{i,1},\dots,c_{i,d_2}] \in \Rdtwo$  representing  static characteristics of the system, i.e., $\vec{c}_i$ is independent of time. As a motivation example consider an image segmentation problem where $\vec{x}_i$ are the center point of  a pixel or voxels of the image and $\vec{c}_i$ the color coding at the center point. 
 
As in the HK model we define the neighborhood of the particles by 
\begin{equation} \label{eq:interactionDomainGeneral}
\N_i(t,\epsilon_1,\epsilon_2) := \left\{ j\in\{1,\dots,n\} : \norm{\vec{x}_i(t)-\vec{x}_j(t)}_{\Rdone}\leq\epsilon_1, \norm{\vec{c}_i-\vec{c}_j}_{\Rdtwo}\leq\epsilon_2 \right\}, \quad i=1,\dots,n
\end{equation}
and $A_{ij}(t,\epsilon_1,\epsilon_2)$ are entries of the time--varying matrix $\vec{A}(t,\epsilon_1,\epsilon_2)\in\Rnn$ defined as
\revision{
\begin{equation} \label{eq:generalMatrix}
	A_{ij}(t,\epsilon_1,\epsilon_2) := \begin{cases}
	\displaystyle{\frac{1}{\sigma_i}}, & \text{if $j\in\N_i(t,\epsilon_1,\epsilon_2)$}\\[2ex]
	0, & \text{otherwise}
	\end{cases}
\end{equation}
with $\sigma_i$ defined analogously to~\eqref{eq:sigma}.
}
Here, $\epsilon_1\geq 0$ and $\epsilon_2\geq 0$ are two bounded confidence levels. The two metrics $\norm{\cdot}_{\Rdone}$ and $\norm{\cdot}_{\Rdtwo}$ need to be properly defined according to the specific context of the problem. Then, the mathematical model \revision{for any $t\geq 0$} is given by 
\revision{
\begin{align} \label{eq:microStaticEq}
	c_{i,k}(t) &= c_{i,k}(0), \quad i=1,\dots,n, \quad k=1,\dots,d_2, \\ 
 \label{eq:generalModel}
	\frac{\d}{\d t} x_{i,k}(t) &= \sum_{j=1}^n A_{ij}(t,\epsilon_1,\epsilon_2) \left( x_{j,k}(t) - x_{i,k}(t) \right), \quad i=1,\dots,n, \quad k=1,\dots,d_1
\end{align}
and initial condition $c_{i,k}(0)=c_{i,k}^0$} and $x_{i,k}(0)=x_{i,k}^0.$ Notice that $\epsilon_2 > \max_{i,j=1,\dots,n} \norm{\vec{c}_i-\vec{c}_j}_{\Rdtwo}$ is a sufficient condition to reduce model~\eqref{eq:generalModel} to the multi--dimensional version of the HK model~\eqref{eq:hkModel1D}. 

\subsection{Properties of the microscopic model} \label{sec:microProp}
Existence and convergence of solutions to system~\eqref{eq:generalModel} can be established by using same techniques as in~\cite{BlondelHendrickxTsitsiklis2010} where the original HK in the one--dimensional case was analyzed. In fact, system~\eqref{eq:microStaticEq}-~\eqref{eq:generalModel} belongs to the same class of state-switched systems.

In the case of symmetric interactions we can recover from the previous model similar results on the moment behavior as presented for the HK model in \cite{BlondelHendrickxTsitsiklis2010,Hendrickx2008}. We only record the results here since the proofs are slight variations of existing results (see for example \cite{BlondelHendrickxTsitsiklis2010}). 

Define the moments $\vec{m}_1(t) \in \Rdone$ and $\vec{m}_2(t) \in \Rdonedone$ with respect \revision{to} the time dependent feature as
\begin{equation} \label{eq:discreteMom}
\vec{m}_1(t) := \sum_{i=1}^n \vec{x}_i(t), \quad \quad \vec{m}_2(t) := \sum_{i=1}^n \vec{x}_i(t) \otimes \vec{x}_i(t),
\end{equation}
then the following result holds true.
\begin{lemma} \label{th:discreteMomSymmetric}
	Let \revision{$(\vec{x}_i(t))_{1\leq i \leq n}$} be the solution of the dynamical system~\eqref{eq:generalModel} with symmetric interactions, i.e.~$\sigma_i=n$ in~\eqref{eq:generalMatrix}. Then, we obtain 
	$$
	\vec{m}_1(t) = \vec{m}_1(0), \quad \frac{\d}{\d t} \left( \vec{m}_2(t) \right)_{kk} \leq 0, \; k=1,\dots,d_1.
	$$
\end{lemma}
\begin{corollary}
	Let \revision{$(\vec{x}_i(t))_{1\leq i \leq n}$} be the solution of the dynamical system~\eqref{eq:generalModel}. Assume that $\epsilon_1$ and $\epsilon_2$ are sufficiently large so that interactions are global. Then the first moment is conserved and the following decay estimate \revision{holds:}
	\begin{align*}
	\frac{\mathrm{d}}{\mathrm{d}t} \left( \vec{m}_2(t) \right)_{k\ell}& = \frac{2}{n} \left( \vec{m}_1(0)\right )_k \left( \vec{m}_1(0)\right )_\ell - 2 \left( \vec{m}_2(t) \right)_{k\ell}, \\
	\lim\limits_{t \to \infty}	\left( \vec{m}_2(t) \right)_{k\ell} &= \frac{\left( \vec{m}_1(0)\right )_k \left( \vec{m}_1(0)\right )_\ell}{n}.
	\end{align*} 
\end{corollary}
Later, we will show that similar results hold for the continuous model.
Some remarks on further properties concerning the particle model~\eqref{eq:generalModel} are in order.
\begin{remark}~ 
\begin{itemize}
\item Formation of clusters in the large time behavior is extensively investigated in the review article~\cite{MotschTadmor2014} for general systems of the form~\eqref{eq:generalModel}. However, number of clusters cannot be a--priori predicted starting from a given initial configuration.

\item In the \revision{non--symmetric} case, conservation of the first moment does not hold true. Also, in general, it is not possible to show the decay of the second moment although clustering still appears, see~\cite{MotschTadmor2014}.
\item Extensions of the model to take into account non static features are obtained by including an interaction term on the right hand side of \eqref{eq:microStaticEq}. The determination of such interaction term, however, would be rather problem dependent and in this paper we will not explore further this direction.
\end{itemize}
\end{remark}


\section{Mean--field description} \label{sec:mfModel}

In the case of many particles we derive the formal mean--field equation. Let $\Omega_1 \subseteq \Rdone$, $\Omega_2 \subseteq \Rdtwo$ \revision{be compact domains} and $\Omega = \Omega_1 \times \Omega_2$. \revision{For $n\geq 1$} we denote by  \revision{$f_n:\mathbb{R}^+\times \Omega \to \mathbb{R}$} the empirical distribution on $\Omega \subset \mathbb{R}^{d_1\times d_2}$ given by  
$$
	f_n(t,\vec{x},\vec{c}) := \frac{1}{n} \sum_{i=1}^n \delta(\vec{x}-\vec{x}_i(t)) \delta(\vec{c} - \vec{c}_i(t)). 
$$

Let us consider
a test function $\varphi(\vec{x},\vec{c}) \in C^1_0(\Omega)$\revision{, i.e.~the space of continuous and compactly supported functions on $\Omega$ with continuous derivative.} Denote by \revision{$\langle \cdot,\cdot \rangle$ the integration of $f_n$ against the test function $\varphi$ on $\Omega$.} We have
\begin{align*}
\frac{\d}{\d t} \langle f_n(t),\varphi \rangle & =  \frac{1}{n} 
\sum_{i=1}^n \frac{d}{dt } \varphi(\vec{x}_i(t),\vec{c}_i(t)) = \frac{1}{n}
\sum_{i=1}^n \frac{1}{\sigma_i} \sum_{j\in\N_i(t,\epsilon_1,\epsilon_2)}  \nabla_x \varphi(\vec{x}_i(t),\vec{c}_i(t))\cdot
\left( \vec{x}_j(t) - \vec{x}_i(t) \right)\\
&= 
\left\langle f_n(t), \frac{1}{n\,\sigma(t,\vec{x},\vec{c})}  \sum_{j=1}^n \chi_{\epsilon_1}\left( \norm{\vec{x}_j(t)-\vec{x}} \right) \chi_{\epsilon_2}\left( \norm{\vec{c}_j(t)-\vec{c}} \right) (\vec{x}_j(t)-\vec{x}) \cdot \nabla_x \varphi \right\rangle
\end{align*}
where \revision{we defined $$\chi_\epsilon(x)=\begin{cases} 1, \quad x\leq \epsilon\\0, \quad \text{else}\end{cases}$$ and} we used the fact that equation~\eqref{eq:generalModel} can be re-written as
\[
\frac{\d}{\d t} \vec{x}_i(t) = \frac{1}{n\,\sigma(t,\vec{x}_i(t),\vec{c}_i(t))} \sum_{j=1}^n \chi_{\epsilon_1}\left( \norm{\vec{x}_j(t)-\vec{x}_i(t)} \right) \chi_{\epsilon_2}\left( \norm{\vec{c}_j(t)-\vec{c}_i(t)} \right)\left( \vec{x}_j(t) - \vec{x}_i(t) \right),
\]
with $\sigma_i= n \sigma(t,\vec{x}_i(t),\vec{c}_i(t)) $ and
\[
\sigma(t,\vec{x}_i(t),\vec{c}_i(t)) = \frac1{n} \sum_{j=1}^n 
\chi_{\epsilon_1}\left( \norm{\vec{x}_j(t)-\vec{x}_i(t)} \right) \chi_{\epsilon_2}\left( \norm{\vec{c}_j(t)-\vec{c}_i(t)} \right).
\]
We can easily compute
\begin{align*}
\sigma(t,\vec{x},\vec{c})&= \frac1{n}\sum_{j=1}^n 
\chi_{\epsilon_1}\left( \norm{\vec{x}_j(t)-\vec{x}} \right) \chi_{\epsilon_2}\left( \norm{\vec{c}_j(t)-\vec{c}} \right)\\
&= 
\left\langle \chi_{\epsilon_1}\left( \norm{\vec{y}-\vec{x}} \right) \chi_{\epsilon_2}\left( \norm{\vec{z}-\vec{c}} \right),\delta (\vec{y}- \vec{x}_j(t))\delta (\vec{z}- \vec{c}_j(t)) \right\rangle\\
&\int_{\Omega}\chi_{\epsilon_1}\left( \norm{\vec{y}-\vec{x}} \right) \chi_{\epsilon_2}\left( \norm{\vec{z}-\vec{c}} \right)f_n(t,\vec{y},\vec{z})\, \d \vec{z}\,\d \vec{y}, 
\end{align*}
and similarly
\begin{align*}
&\frac1{n\,\sigma(t,\vec{x},\vec{c})}\sum_{j=1}^n \chi_{\epsilon_1}\left( \norm{\vec{x}_j(t)-\vec{x}} \right) \chi_{\epsilon_2}\left( \norm{\vec{c}_j(t)-\vec{c}} \right) (\vec{x}_j(t)-\vec{x})\\
&= \frac1{\sigma(t,\vec{x},\vec{c})} \int_{\Omega} \chi_{\epsilon_1}\left( \norm{\vec{y}-\vec{x}} \right) \chi_{\epsilon_2}\left( \norm{\vec{z}-\vec{c}} \right) (\vec{y}-\vec{x}) f_n(t,\vec{y},\vec{z}\,)\, \d \vec{z}\,\d \vec{y}.
\end{align*}
Collecting these formal computations,  
after integration by part in $\vec{x}$, we obtain the weak form of the mean--field equation 
\begin{align*}
	\frac{\d}{\d t} \langle f_n,\varphi \rangle + \left\langle \nabla_\vec{x} \cdot \left(f_n(t,\vec{x},\vec{c})\int_{\Omega} \frac{\chi_{\epsilon_1}\left( \norm{\vec{y}-\vec{x}} \right) \chi_{\epsilon_2}\left( \norm{\vec{z}-\vec{c}} \right)}{\sigma(t,\vec{x},\vec{c})} (\vec{y}-\vec{x}) f_n(t,\vec{y},\vec{z})\, \d \vec{z}\,\d \vec{y} \right),\varphi \right\rangle = 0. 
\end{align*}

If we now define a \emph{kernel} $\mathcal{A}$ as continuous extension of the adjacency matrix~\eqref{eq:generalMatrix} 
\begin{equation} \label{eq:kernelH}
	\mathcal{A}_{\epsilon_1,\epsilon_2}(t,\vec{x},\vec{c},\vec{y},\vec{z}) = \frac{\chi_{\epsilon_1}(\norm{\vec{y}-\vec{x}}) \chi_{\epsilon_2}\left( \norm{\vec{z}-\vec{c}} \right)}{\sigma(t,\vec{x},\vec{c})},
\end{equation}
and $\vec{V}$ given by 
\begin{equation} \label{eq:vField}
	\vec{V}_{\epsilon_1,\epsilon_2}(t,\vec{x},\vec{c}) = \int_{\Omega} \mathcal{A}_{\epsilon_1,\epsilon_2}(t,\vec{x},\vec{c},\vec{y},\vec{z}) (\vec{y}-\vec{x}) f(t,\vec{y},\vec{z})\, \d\vec{z}\,\d\vec{y}
\end{equation}
in the limit $n\to\infty$, assuming that the empirical measure $f_n(t,\vec{x},\vec{c})$ converge to $f(t,\vec{x},\vec{c})$, we formally obtain the strong form of the kinetic equation as 
\begin{equation} \label{eq:kineticEq}
	\partial_t f(t,\vec{x},\vec{c}) + \nabla_\vec{x} \cdot \left( \vec{V}_{\epsilon_1,\epsilon_2}(t,\vec{x},\vec{c}) f(t,\vec{x},\vec{c}) \right) = 0. 
\end{equation}
%

Rigorous analytical results on convergence in the case of $\epsilon_2$ very large and symmetric interactions have been already discussed, for instance in~\cite{CanutoFagnaniTilli2008,CanutoFagnaniTilli12}. These results guarantee convergence of the distribution $f$ in~\eqref{eq:kineticEq} to a probability limit distribution $f^\infty$, provided \revision{the initial distribution $f_0$ at time $t=0$} has finite second moment and $\mathcal{A}$ is a non--negative, bounded, measurable and symmetric kernel. Observe that these assumptions are verified also in our framework. For a more detailed discussion on analytical results on convergence in mean--field models, we refer e.g.~to~\cite{CanizoCarrilloRosado2011,CarrilloFornasierToscaniVecil2010}

We briefly discuss the relation to similar kinetic models.

In~\cite{BorraLorenzi2013} the analysis of a homogeneous kinetic model for opinion dynamics under bounded confidence is studied. Therein, the model is derived by using a Boltzmann--like approach with instantaneous binary interactions describing compromise. An analogous derivation of the kinetic equation~\eqref{eq:kineticEq} is also possible using a binary interaction model based on an explicit Euler discretization of the underlying particle dynamics~\eqref{eq:generalModel} with $n=2$ and performing a grazing collision limit. We omit this computation.

In~\cite{BorraLorenzi2013} the authors prove the weak convergence of the solution to a convex combination of Dirac delta functions. In~\cite{CanutoFagnaniTilli2008} a similar asymptotic distribution is found for the mean--field limit of the classical Hegselmann--Krause model.

A further symmetric clustering model with weighted interactions with respect a fixed number of closest neighbors and corresponding mean--field limit has been introduced in~\cite{AlbiPareschi2013AppMathLett}. Moments and long--time behavior could also be studied therein. Finally, models for other applications are also able to cluster information e.g.~in traffic flow modeling~\cite{PgSmTaVg2} where the physical acceleration has the role of the bounded confidence.  
 
\subsection{Properties of the mean--field model} \label{sec:kineticProp}

As preliminary remark we observe that the marginal distribution $\tilde{f}^c(t,\vec{c}):=\int_{\Omega_1} f(t,\vec{x},\vec{c}) \d\vec{x}$ is preserved in time by the kinetic equation~\eqref{eq:kineticEq}. Instead, it is easy to check that the marginal distribution $\tilde{f}^x(t,\vec{x}):=\int_{\Omega_2} f(t,\vec{x},\vec{c}) \d\vec{c}$ is not stationary and its behavior in time is still influenced by a $\vec{c}$ dependent kernel. These considerations are direct consequence of the microscopic model~\eqref{eq:generalModel} and~\eqref{eq:microStaticEq}. For this reason, in the following results and if not otherwise stated, we mainly focus on the analysis of moments with respect to the variable $\vec{x}$.

As in the discrete case we define the $p$--th moment of the kinetic distribution with respect to $\vec{x}$ as
\begin{equation*}
	\langle \vec{x}^{\vecsym{\alpha}} \rangle(t) = \int_{\Omega} \vec{x}^{\vecsym{\alpha}} f(t,\vec{x},\vec{c}) \d\vec{x} \d\vec{c}, \quad \abs{\vecsym{\alpha}} = p\in\mathbb{N}, \ \alpha_i\in\mathbb{N}.
\end{equation*}
\revision{Here, we used the following multi--index notation $\vec{x}^{\vecsym{\alpha}} = x_1^{\alpha_1} \cdots x_d^{\alpha_{d_1}}$ and $\abs{\vecsym{\alpha}}=\sum_{i=1}^{d_1} \alpha_i$.}
In particular, for each $k,j=1,\dots,d_1$ we will denote
\begin{align*}
	u_k(t) &= \langle \vec{x}^{\vecsym{\alpha}} \rangle(t), \quad \abs{\vecsym{\alpha}} = 1, \ \alpha_i=\delta_{ik}, \ i=1,\dots,d_1\\
	E_{kj}(t) &= \langle \vec{x}^{\vecsym{\alpha}} \rangle(t), \quad \abs{\vecsym{\alpha}} = 2, \ \alpha_i=\delta_{ik}+\delta_{ij}, \ i=1,\dots,d_1.
\end{align*}
the first moment and the second moment, respectively. 
Then, for a symmetric kernel $\mathcal{A}$ we obtain by the integration of the kinetic equation

\begin{lemma} \label{th:kineticMomSymm}
	Let $f(t,\vec{x},\vec{c})$ be the solution of the model~\eqref{eq:kineticEq} with symmetric kernel $\mathcal{A}$, i.e., \revision{$\sigma(t,\vec{x},\vec{c}) = 1$}. Then, the following a--priori estimates hold true:
	$$
		\frac{\d}{\d t} u_k(t) = 0, \quad \frac{\d}{\d t} E_{kk}(t) \leq 0, \quad \abs{E_{kj}(t)} \leq \frac12 \left( E_{kk} + E_{jj} \right)(0)
	$$
	for each $k,j=1,\dots,d_1$.
\end{lemma}
\begin{proof}
	For each fixed $k=1,\dots,d_1$ we have
	\begin{align*}
		\frac{\d}{\d t} u_k(t) &= -\int_\Omega x_k \nabla_\vec{x} \cdot \left( \vec{V}_{\epsilon_1,\epsilon_2}(t,\vec{x},\vec{c}) f(t,\vec{x},\vec{c}) \right) \d\vec{x} \d\vec{c} \\&= \int_{\Omega} \int_{\Omega} \chi_{\epsilon_1}(\norm{\vec{y}-\vec{x}}) \chi_{\epsilon_2}(\norm{\vec{z}-\vec{c}}) (y_k-x_k) f(t,\vec{y},\vec{z}) f(t,\vec{x},\vec{c}) \d\vec{y} \d\vec{z} \d\vec{x} \d\vec{c} = 0.
	\end{align*}
	The last term vanishes due to anti--symmetry of the integrand by the interchange of variables $\vec{x} \leftrightarrow \vec{y}$. The second statement follows by observing that for each $k=1,\dots,d_1$ we have
	\begin{align*}
		\frac{\d}{\d t} E_k(t) &= -\int_\Omega x_k^2 \nabla_\vec{x} \cdot \left( \vec{V}_{\epsilon_1,\epsilon_2}(t,\vec{x},\vec{c}) f(t,\vec{x},\vec{c}) \right) \d\vec{x} \d\vec{c}\\
		&= - \int_{\Omega} \int_{\Omega} \chi_{\epsilon_1}(\norm{\vec{y}-\vec{x}}) \chi_{\epsilon_2}(\norm{\vec{z}-\vec{c}}) (y_k-x_k)^2 f(t,\vec{y},\vec{z}) f(t,\vec{x},\vec{c}) \d\vec{y} \d\vec{z} \d\vec{x} \d\vec{c}\leq 0.
	\end{align*}
Hence, we obtain 
	$$
		0 \leq \abs{E_{kj}(t)} \leq \int_{\Omega} \abs{x_k x_j} f(t,\vec{x},\vec{c}) \d\vec{x} \d\vec{c} \leq \frac12 \left( E_{kk} + E_{jj} \right)(t) \leq \frac12 \left( E_{kk} + E_{jj} \right)(0).\qedhere
	$$
\end{proof}

Instead, in the case of global interactions and for a general kernel $\mathcal{A}$ and we have
\begin{corollary} \label{cor:kineticGlobal}
	Assume that $\epsilon_1$ and $\epsilon_2$ are sufficiently large so that $\mathcal{A}(t,\vec{x},\vec{c},\vec{y},\vec{z};\epsilon_1,\epsilon_2)=1$, for all $\vec{x},\vec{y}\in\Omega_1$, $\vec{c},\vec{z}\in\Omega_2$ and $t\geq 0$. Then, for all $k,j=1,\dots,d_1$, the following relations hold true:
	$$
		u_k(t) = u_k(0), \ \forall\,t\geq 0 \quad \mbox{and} \quad E_{kj}(t) \xrightarrow{ t \to \infty } u_k(0)u_j(0).
	$$
\end{corollary}
\begin{proof}
	Multiplying~\eqref{eq:kineticEq} by $\vec{x}^{\vecsym{\alpha}}$ with $\abs{\vecsym{\alpha}}=p\in\mathbb{N}$, and integrating over $\Omega$, we compute
	$$
		\frac{\d}{\d t} \langle \vec{x}^{\vecsym{\alpha}} \rangle(t) = -p \langle \vec{x}^{\vecsym{\alpha}} \rangle(t) + \sum_{\ell=1}^{d_1} \alpha_\ell u_\ell(t) \langle \vec{x}^{\vecsym{\alpha}^{(\ell)}} \rangle(t) 
	$$
	where $\vecsym{\alpha}^{(\ell)} = \vecsym{\alpha} - \vec{e}_\ell$ with $\vec{e}_\ell$ being the $\ell$--th element of the standard basis vector of $\Rdone$ and $\abs{\vecsym{\alpha}^{(\ell)}} = p-1$. Then, for $p=1$ we still have have conservation of the first moment since we obtain
	$$
		\frac{\d}{\d t} u_k(t) = -u_k(t) + u_k(t) = 0, \quad k=1,\dots,d_1.
	$$
	For $p=2$ we have
	$$
		\frac{\d}{\d t} E_{kj}(t) = -2E_{kj}(t) + 2 u_k(t)u_j(t) = -2E_{kj}(t) + 2 u_k(0)u_j(0), \quad k,j=1,\dots,d_1
	$$
	and therefore
	\revision{
	$$
		E_{kj}(t) = E_{kj}(0) e^{-2t} + u_k(0)u_j(0) \left( 1 - e^{-2t} \right) \xrightarrow{ t \to \infty } u_k(0)u_j(0), \quad k,j=1,\dots,d_1.\qedhere
	$$
	}
\end{proof}

Observe that, compared to the local interaction case analyzed in Theorem~\ref{th:kineticMomSymm} where in general the energy decay property of all second order moments is not guaranteed, in global interactions also the mixed second order moments decay in time. In other words, while in the local interaction model only variances go to zero, in the global interaction model both variances and covariances go to zero in the large time behavior.


Lemma~\ref{th:kineticMomSymm} and Corollary~\ref{cor:kineticGlobal} suggest that for $t\to \infty$ any initial distribution will tend to a stationary distribution that is concentrated on a finite number of points in $\Omega_1$ and represents therefore clusters. Hence, in the following we investigate the asymptotic distribution of the model~\eqref{eq:kineticEq}.

\begin{lemma} \label{th:steadystateEpsSmall}
	Let $\epsilon_1$ and $\epsilon_2$ be arbitrary positive bounded confidence levels. Then, the distribution
	\begin{equation} \label{eq:finfty}
	f^\infty(\vec{x},\vec{c}) = \sum_{k=1}^{n_1} f_k \delta(\vec{x}-\vec{x}_k) \sum_{\ell=1}^{n_2(k)} \delta(\vec{c}-\vec{c}_{k_\ell})
	\end{equation}
	with $f_k>0$ and $\sum_{k=1}^{n_1} \sum_{\ell=1}^{n_2(k)} f_k=1$ is a weak stationary solution of the model~\eqref{eq:kineticEq} if and only if either $\norm{\vec{x}_i-\vec{x}_k}_{\Rdone}>\epsilon_1$ for all $i\neq k$ or $\norm{\vec{c}_{i_j}-\vec{c}_{k_\ell}}_{\Rdtwo}>\epsilon_2$ for all $i\neq k$, for all $j,\ell$, or both hold true .
\end{lemma}
\begin{proof}
	Let us assume that at least one of $\norm{\vec{x}_i-\vec{x}_k}_{\Rdone}>\epsilon_1$ for all $i\neq k$ and $\norm{\vec{c}_{i_j}-\vec{c}_{k_\ell}}_{\Rdtwo}>\epsilon_2$ for all $i\neq k$, for all $j,\ell$, is verified and prove that $f^\infty$ is the asymptotic distribution of the equation~\eqref{eq:kineticEq}, i.e.~given a test function $\varphi\in C_0^1(\Omega)$
	$$
		\int_{\Omega} \varphi(\vec{x},\vec{c}) \nabla_\vec{x} \cdot \left(\vec{V}_{\epsilon_1,\epsilon_2}(t,\vec{x},\vec{c}) f^\infty(\vec{x},\vec{c})\right) \d\vec{x} \d\vec{c} = 0.
	$$
	We obtain 
	\begin{align*}
	\int_{\Omega} \varphi(\vec{x},\vec{c}) \nabla_\vec{x} \cdot \left(\vec{V}_{\epsilon_1,\epsilon_2}(t,\vec{x},\vec{c}) f^\infty(\vec{x},\vec{c})\right) \d\vec{x} \d\vec{c} &= -\sum_{k=1}^{n_1} f_k \sum_{\ell=1}^{n_2(k)} \vec{V}(t,\vec{x}_k,\vec{c}_{k_\ell};\epsilon_1,\epsilon_2) \cdot \nabla_\vec{x} \varphi(\vec{x},\vec{c})|_{\vec{x}=\vec{x}_k} \\
	&= -\sum_{k=1}^{n_1} f_k^2 \sum_{\ell=1}^{n_2(k)} \frac{(\vec{x}_k-\vec{x}_k)}{\sigma(\vec{x}_k,\vec{c}_{k_\ell})} \cdot \nabla_\vec{x} \varphi(\vec{x})|_{\vec{x}=\vec{x}_k} = 0.
	\end{align*}
	Conversely,  assume that $f^\infty$ as in~\eqref{eq:finfty} is the steady state of~\eqref{eq:kineticEq}. Assume by contradiction that there exist $k,i\in\{1\dots,n_1\}$ such that $\norm{\vec{x}_k-\vec{x}_i}_{\Rdone}\leq\epsilon_1$ and $\ell\in\{1\dots,n_2(k)\}$, $j\in\{1\dots,n_2(i)\}$ such that $\norm{\vec{c}_{i_j}-\vec{c}_{k_\ell}}_{\Rdtwo}\leq\epsilon_2$. Then, using similar computations as before we obtain 
	\begin{align*}
	\int_{\Omega} \varphi(\vec{x},\vec{c}) \nabla_\vec{x} \cdot \left(\vec{V}_{\epsilon_1,\epsilon_2}(t,\vec{x},\vec{c}) f^\infty(\vec{x},\vec{c})\right) \d\vec{x} \d\vec{c} =& - f_k f_i \frac{(\vec{x}_k-\vec{x}_i)}{\sigma(\vec{x}_k,\vec{c}_{k_\ell})} \cdot \nabla_\vec{x} \varphi(\vec{x},\vec{c})|_{\vec{x}=\vec{x}_k} \\ &- f_i f_k \frac{(\vec{x}_i-\vec{x}_k)}{\sigma(\vec{x}_i,\vec{c}_{i_j})} \cdot \nabla_\vec{x} \varphi(\vec{x},\vec{c})|_{\vec{x}=\vec{x}_i} \neq 0
	\end{align*}
	which contradicts the hypothesis. 
\end{proof}
\begin{lemma} \label{th:steadystateEpsLarge}
	Assume that $\epsilon_1$ and $\epsilon_2$ are sufficiently large so that $\mathcal{A}(t,\vec{x},\vec{c},\vec{y},\vec{z};\epsilon_1,\epsilon_2)=1$, for all $\vec{x},\vec{y}\in\Omega_1$, $\vec{c},\vec{z}\in\Omega_2$ and $t\geq 0$. 
	Then, the distribution
	$$
		f^\infty(\vec{x},\vec{c}) = \delta(\vec{x}-\vec{\bar{x}}) \tilde{f}^c(0,\vec{c}) = \prod_{k=1}^{d_1} \delta(x_k-\bar{x}_k) \tilde{f}^c(0,\vec{c})
	$$
	with $\tilde{f}^c(0,\vec{c}) = \int_{\Omega_1} f_0(\vec{x},\vec{c}) \d\vec{x}$ is a weak stationary solution of the model~\eqref{eq:kineticEq} if and only if $\bar{x}_k = u_k(0)$.
\end{lemma}
\begin{proof}
		Substituting the expression of $f^\infty$ in the weak form of the kinetic equation, integrating by parts and using the fact that for $\epsilon_1$ and $\epsilon_2$ large enough so that
	$$
		\vec{V}_{\epsilon_1,\epsilon_2}(t,\vec{x},\vec{c}) = \int_{\Omega} \vec{y} f(t,\vec{y},\vec{z}) \d\vec{z} \d\vec{y} - \vec{x} = \int_{\Omega} \vec{y} f_0(\vec{y},\vec{z}) \d\vec{z} \d\vec{y} - \vec{x}
	$$
	we obtain
	\begin{align*}
		\int_{\Omega} \varphi(\vec{x},\vec{c}) \nabla_\vec{x} \cdot \left(\vec{V}_{\epsilon_1,\epsilon_2}(t,\vec{x},\vec{c}) f^\infty(\vec{x},\vec{c})\right) \d\vec{x} \d\vec{c} = \int_{\Omega_2} \tilde{f}(\vec{c}) \left( \vec{\bar{x}} - \int_{\Omega} \vec{y} f_0(\vec{y},\vec{z}) \d\vec{z} \d\vec{y} \right) \cdot \nabla_\vec{x} \varphi(\vec{x},\vec{c})|_{\vec{x}=\vec{\bar{x}}} \d\vec{c}
	\end{align*}
	which is zero if  $\vec{\bar{x}} = \int_{\Omega} \vec{y} f_0(\vec{y},\vec{z}) \d\vec{z} \d\vec{y}$. 
\end{proof}

\begin{remark} \label{rem:hkConsistency}
	As direct consequence of Theorem~\ref{th:steadystateEpsSmall} and Theorem~\ref{th:steadystateEpsLarge} we have that taking $\epsilon_2$ sufficiently large or an initial distribution $f_0(\vec{x},\vec{c})$ being atomic with respect the second variable, the same quantized steady--state of the kinetic formulation of the microscopic model~\eqref{eq:hkModel1D} are preserved, cf.~\cite{CanutoFagnaniTilli2008}.
\end{remark}

\begin{remark}
	In Theorem~\ref{th:steadystateEpsSmall} the number of clusters $n_1$, as well the values $f_j$ and the positions $\vec{x}_j$ of the clusters, are functions of the initial distribution $f_0(\vec{x},\vec{c})$ and of the bounded confidence levels $\epsilon_1$ and $\epsilon_2$. As pointed--out also in~\cite{CanutoFagnaniTilli12}, in general it is not possible to predict the number of Dirac deltas in the asymptotic configuration from a given initial distribution. However, for $\epsilon_2$ sufficiently large and assuming $\Omega = [0,1]^d$, the number of clusters is $1\leq \tilde{n} \leq \lfloor \frac{1}{\epsilon_1^d} \rfloor$. This consideration is also observed at the Boltzmann level, see~\cite{BorraLorenzi2013}.
\end{remark}

\section{Numerical experiments and applications} \label{sec:examples1}

The theoretical results on the asymptotic behavior of the mean--field equation~\eqref{eq:kineticEq} introduced in the previous section are here also numerically investigated. Moreover, we show the efficiency of the model as technique to solve realistic data--clustering problems and to this end we focus on applications to the field of shape detection and image segmentation.

In order to efficiently solve the kinetic model~\eqref{eq:kineticEq} we employ the Mean Field Interaction Algorithm introduced in~\cite{AlbiPareschi2013} which is based on random subset evaluation of the kernel term $\mathcal{A}_{\epsilon_1,\epsilon_2}(t,\vec{x},\vec{c},\vec{y},\vec{z})$ given in~\eqref{eq:kernelH}. The algorithm used in the numerical experiments is summarized by the steps in Algorithm~\ref{alg:meanfield} for a time interval $[0,T]$ discretized in $k_{tot}$ subintervals of size $\Delta t$. The computational cost of the algorithm is $O=(MN)$, where $M$ is the size of the subset of interacting particles, and for $M=N$ we obtain the explicit Euler scheme for the original $N$ particle system~\eqref{eq:generalModel} whose cost is $O(N^2)$. For further details on the Mean Field Interaction Algorithm we refer to~\cite{AlbiPareschi2013}. 

\begin{algorithm}[t!]
	\begin{algorithmic}[1]
		\STATE Given $N$ sample pairs $(x_i^0,c_i^0)$, with $i=1,\dots,N$ computed from the initial distribution $f^0(x,c)$ and $M\leq N$;
		\FOR{$k=0$ \TO $k_{tot}-1$}
		\FOR{$i=1$ \TO $N$}
		\STATE sample $M$ data $j_1,\dots,j_M$ uniformly without repetition among all data;
		\STATE compute
		$$
		\revision{\bar{A}(x_i^k,c_i^0) = \sum_{\ell = 1}^M \mathcal{A}_{\epsilon_1,\epsilon_2}(x_i^k,c_i^0,x_{j_\ell}^k,c_{j_\ell}^0), \quad \bar{x}_i^k = \sum_{\ell=1}^M \frac{\mathcal{A}_{\epsilon_1,\epsilon_2}(x_i^k,c_i^0,x_{j_\ell}^k,c_{j_\ell}^0)}{\bar{A}(x_i^k,c_i^0)}x_{j_\ell}}
		$$
		\STATE compute the data change
		$$
		x_i^{k+1} = x_i^k \left(1-\Delta t \bar{A}(x_i^k,c_i^0)\right) + \Delta t \bar{A}(x_i^k,c_i^0) \bar{x}_i^k
		$$
		\ENDFOR
		\ENDFOR
	\end{algorithmic}
	\caption{Mean Field Interaction Algorithm for the kinetic equation~\eqref{eq:kineticEq}.}
	\label{alg:meanfield}
\end{algorithm}

\subsection{Numerical steady--states and moment evolution} \label{sec:numericalProp}

We use the Mean Field Interaction Algorithm to numerically investigate the properties of the kinetic model proved in the previous section. We analyze two typical situations which lead to the two applications we show later in this section.

\subsubsection{Constant static feature.}
First, we consider an initial distribution being the uniform distribution with respect to $\vec{x}$ and a constant distribution along the static feature variable $\vec{c}$, i.e.~$f_0(\vec{x},\vec{c}) = \chi_{[0,1]^{d_1}}(\vec{\vec{x}})$. This choice is obviously also consistent with considering the bounded confidence level $\epsilon_2$ very large. We show that, as discussed in Remark~\ref{rem:hkConsistency}, equation~\eqref{eq:kineticEq} provides the same steady--states of the Hegselmann--Krause model.

\paragraph{One--dimension.} The numerical steady--states of the mean--field equation~\eqref{eq:kineticEq} are first computed for the one--dimensional case and compared to the steady--states of the microscopic model~\eqref{eq:generalModel}. Observe that, under the assumption of an initial constant distribution along $\vec{c}$, the asymptotic behavior of~\eqref{eq:generalModel} is the classical one prescribed by the one--dimensional Hegselmann--Krause model~\eqref{eq:hkModel1D}.
The evolution in time of the first and second moment is also provided.

\begin{figure}[t!]
	\centering
		\includegraphics[width=0.49\textwidth]{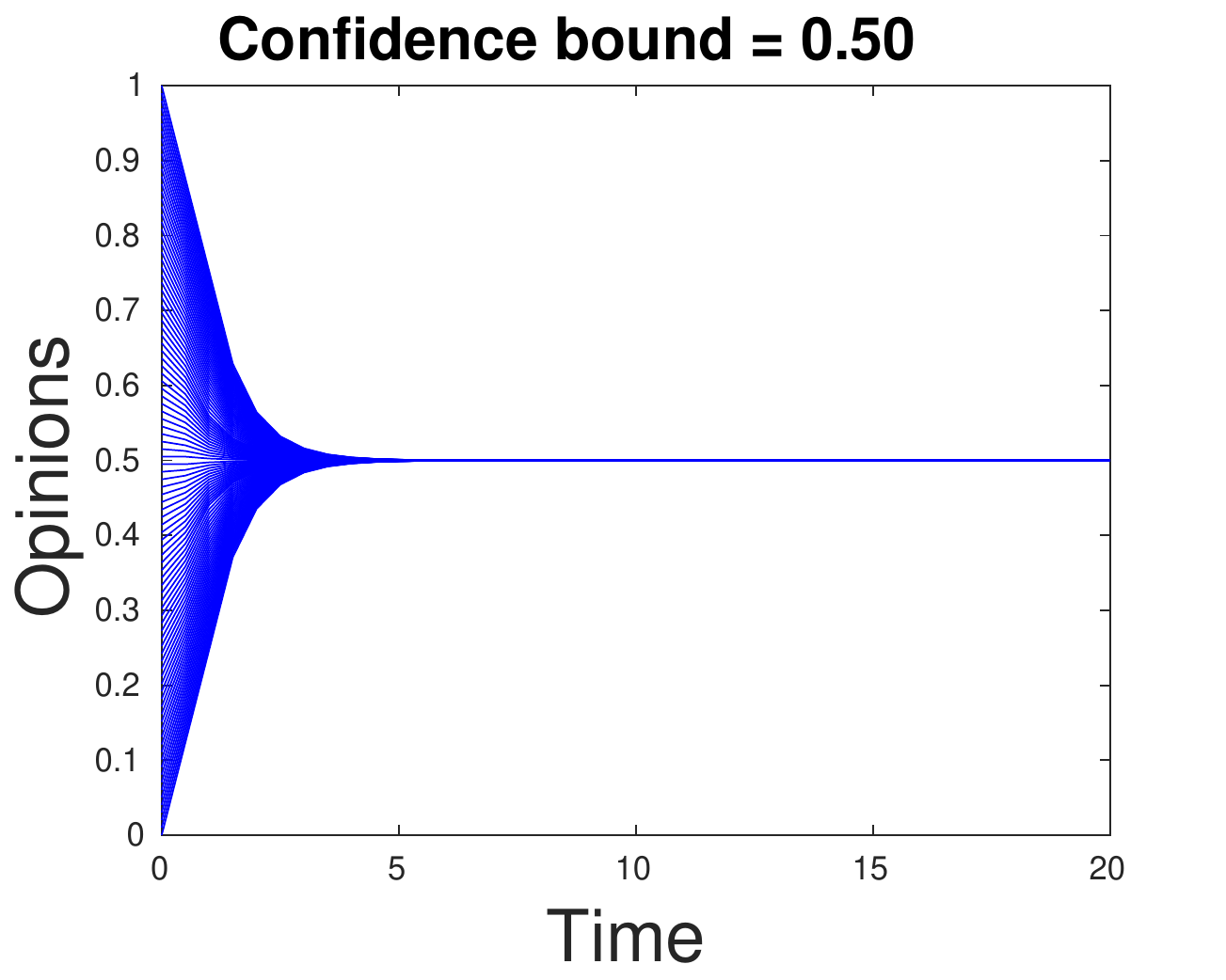}
	\includegraphics[width=0.49\textwidth]{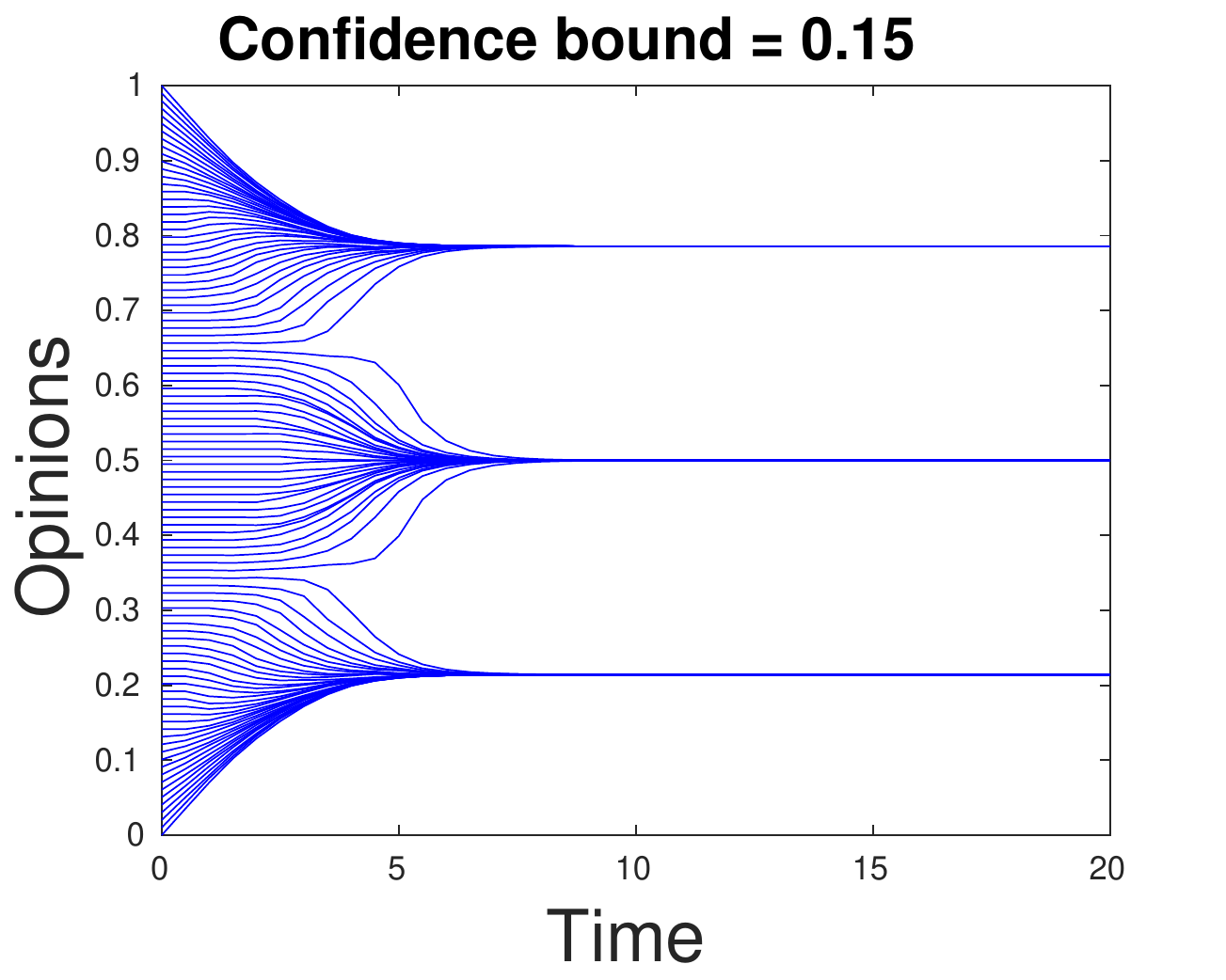}
	\includegraphics[width=0.49\textwidth]{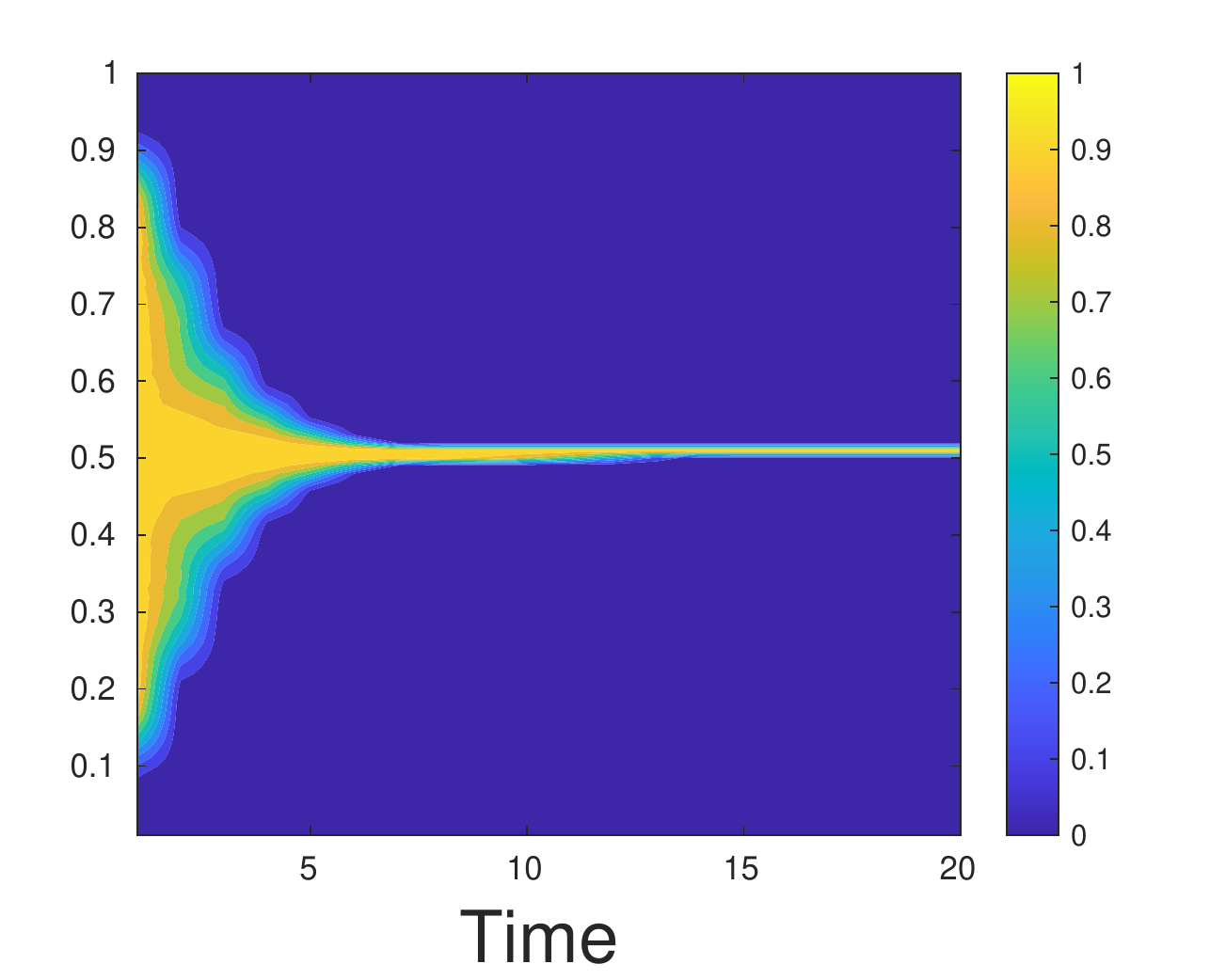}
	\includegraphics[width=0.49\textwidth]{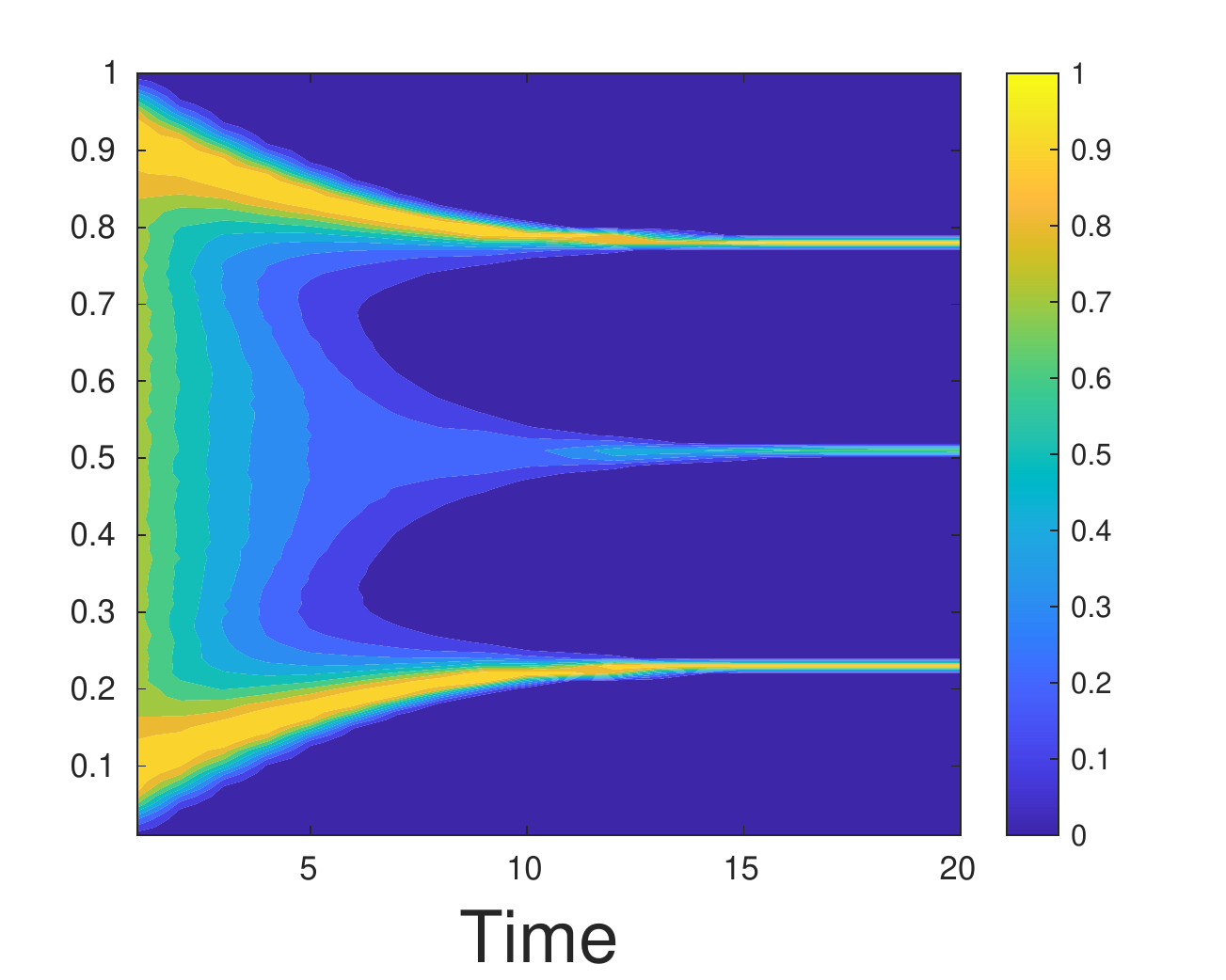}
	\caption{Trend to the steady--state of the one--dimensional Hegselmann--Krause model~\eqref{eq:hkModel1D} with $n=100$ agents equally spaced at initial time and non--symmetric interactions (top row) and of the mean--field model~\eqref{eq:kineticEq} computed with Algorithm~\ref{alg:meanfield} (bottom row) up to final time \revision{$T=20$}. Left panels show the case for \revision{$\epsilon_1=0.5$}, the right panels show the case for $\epsilon_1=0.15$.\label{fig:onedimUniformSS}}
\end{figure}

In Figure~\ref{fig:onedimUniformSS} we show the steady--state provided by the mean--field kinetic model~\eqref{eq:kineticEq} for the one--dimensional initial uniform distribution on $[0,1]$. The final time is \revision{$T=20$} and the time step is $\Delta t=0.5$. We consider $N=5\times 10^5$ particles in the Monte Carlo method so that we reduce error due to the sampling. The number of interacting particles is taken as $M=10$. We show the results for two values of the confidence bound, \revision{$\epsilon_1=0.5$} in the left panel and $\epsilon_1=0.15$ in the right panel. The evolution of the distributions up to final time is showed by normalizing with respect to the maximum value at each fixed time. For \revision{$\epsilon_1=0.5$} we observe the formation of a consensus state in large time behavior. In fact, the final distribution is a Dirac delta centered in the initial value of the first moment. For $\epsilon_1=0.15$, instead, three clusters arise at equilibrium, 
similarly the classical Hegselmann--Krause model for the same confidence bounds and equally distributed initial data.

\begin{figure}[t!]
	\centering
	\includegraphics[width=0.49\textwidth]{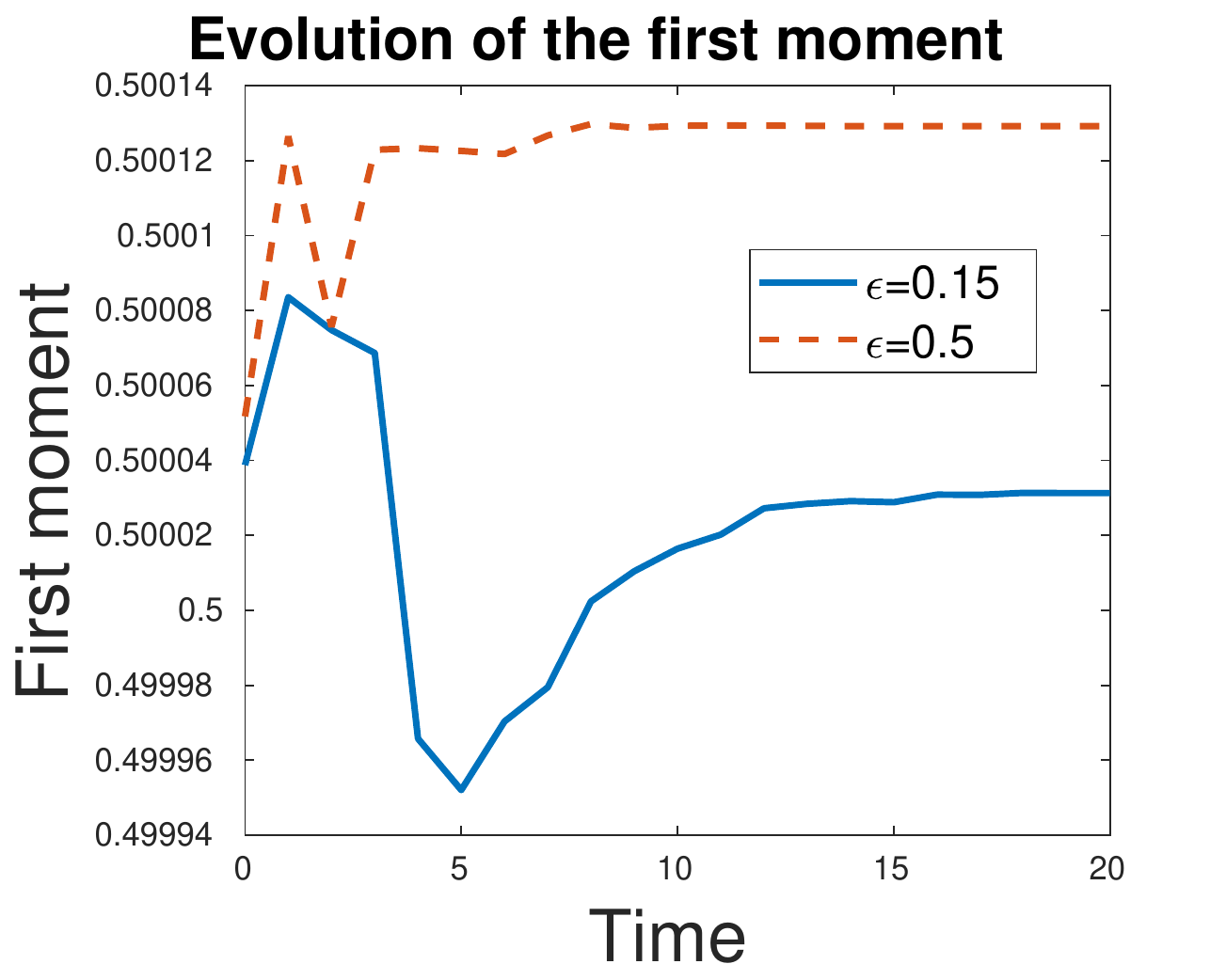}
	\includegraphics[width=0.49\textwidth]{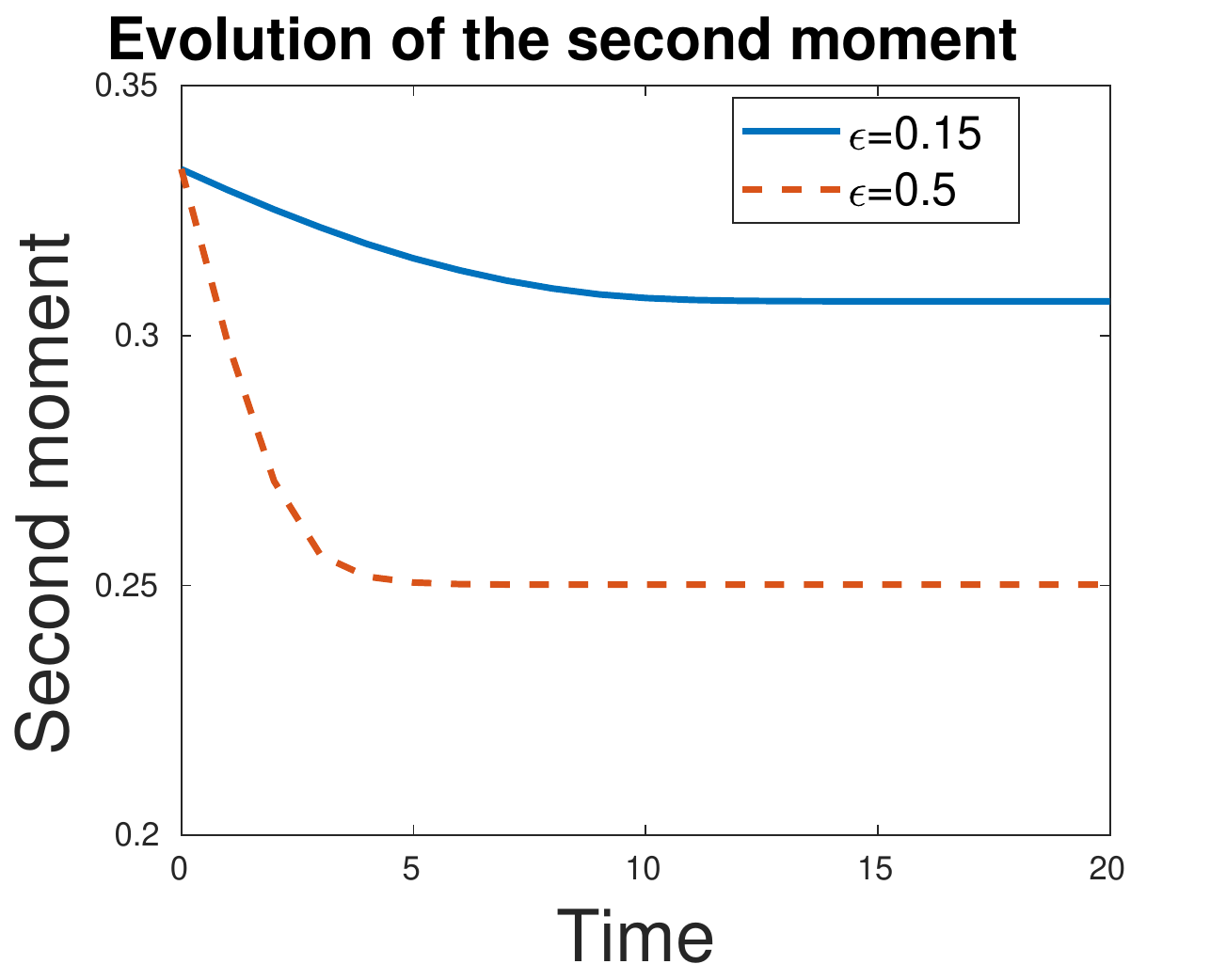}
	\caption{Evolution in time of the first moment (left) and of the second moment (right) for the two values of the bounded confidence level \revision{$\epsilon_1=0.5$} (dashed lines) and $\epsilon_1=0.15$ (solid lines).\label{fig:onedimUniformMom}}
\end{figure}

As observed in Figure~\ref{fig:onedimUniformMom}, for both values of the bounded confidence level, we have conservation of the first moment and energy dissipation in time. Although we are using a non--symmetric model, the first moment is preserved during the time evolution, with deviation of order \revision{$10^{-4}$} from the initial value, since the initial distribution is symmetric. The fluctuations in the first moment are due to the stochastic method.

\begin{figure}[t!]
	\centering
	\includegraphics[width=0.49\textwidth]{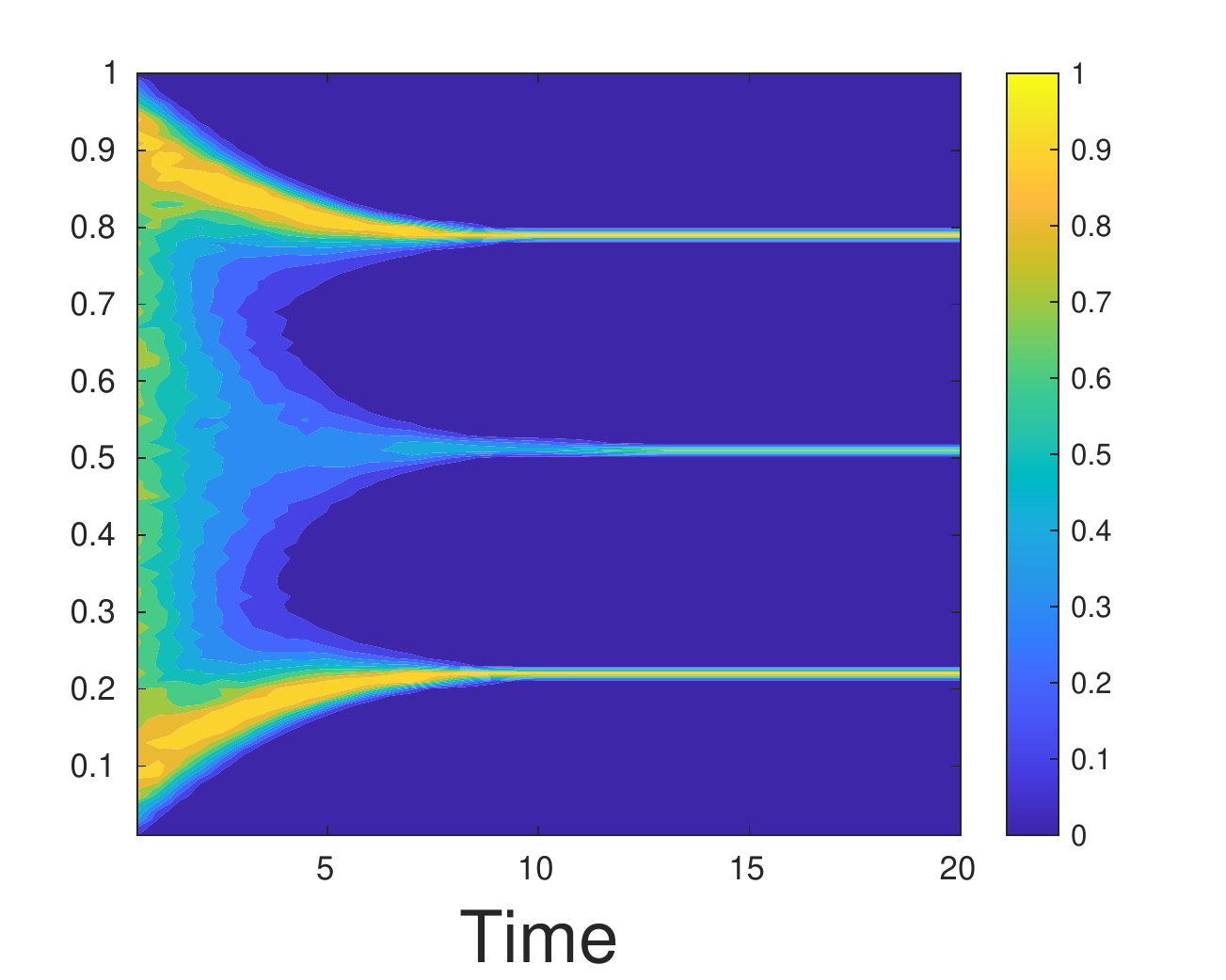}
	\includegraphics[width=0.49\textwidth]{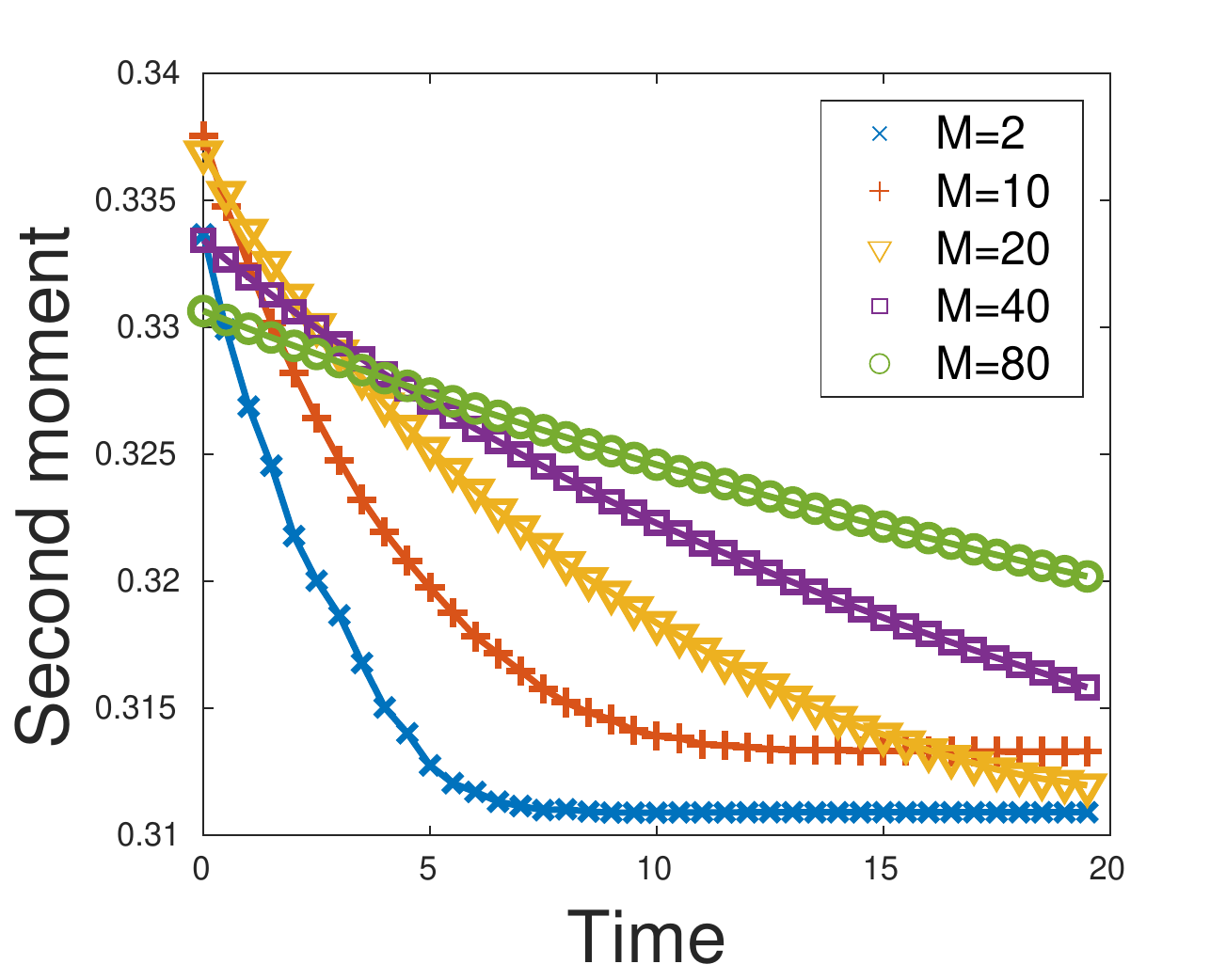}
	\caption{Left: trend to the steady--state of the mean--field model~\eqref{eq:kineticEq} computed with Algorithm~\ref{alg:meanfield} with \revision{$N=2\times10^4$}, $M=2$, $\epsilon_1=0.15$ and up to final time \revision{$T=20$}. Right: energy decay of the mean--field model~\eqref{eq:kineticEq} for several values of interacting particles $M$.\label{fig:analysisM}}
\end{figure}

In Figure~\ref{fig:analysisM} we show some results of a simple analysis of Algorithm~\ref{alg:meanfield} with respect to the values of $M$, i.e.~the size of the random subset of interacting particles. Left panel shows the trend to the steady--state of the kinetic model computed with Algorithm~\ref{alg:meanfield} and same parameters as in the right plot of Figure~\ref{fig:onedimUniformSS}. However, here we consider \revision{$N=2\times10^4$ and} $M=2$. The same equilibrium state is reached but with a faster transient. The dependence of the velocity to the formation of cluster on the size $M$ of the random subset is showed in the right panel of Figure~\ref{fig:analysisM}. The velocity to equilibrium decreases as $M$ increases. In particular, the $M=2$ case, where each particle is enforced to align to the velocity of the other particle instead of their ``mean'', exhibits the fastest convergence. This analysis suggests that, although model~\eqref{eq:kineticEq} has a convolution structure and in principle can be efficiently solved with a Fast Fourier Transform at a $O(N\log N)$ cost, Algorithm~\ref{alg:meanfield} may be preferable since it has a comparable (or even lower) computational cost and it can be applied to more general kernels where the convolution structure is lost.

\paragraph{Two--dimensions.} We also show one example of the numerical steady--states of the mean--field equation~\eqref{eq:kineticEq} in the two--dimensional case. Again we employ Algorithm~\ref{alg:meanfield}. Figure~\ref{fig:twodimUniformSS} shows the \revision{steady--state} for the particle density and the kinetic density at time $t=4$ and final time $T=50$. We \revision{use} $N=10000$ particles and the time step is $\Delta t=0.5$. Using the Mean Field Interaction Algorithm~\ref{alg:meanfield}, the number of interacting particles is taken as $M=10$. We show the results for one value of the confidence bound, $\epsilon_1=0.15$ which leads to the formation of $8$ clusters at equilibrium. The evolution in time of the two--dimensional moments is also provided in Figure~\ref{fig:twodimUniformMom}. We still observe conservation of the first moment, due to the initial symmetric distribution, and decay in time for the two non--mixed second moments. 

\begin{figure}[t!]
	\centering
	\includegraphics[width=0.49\textwidth]{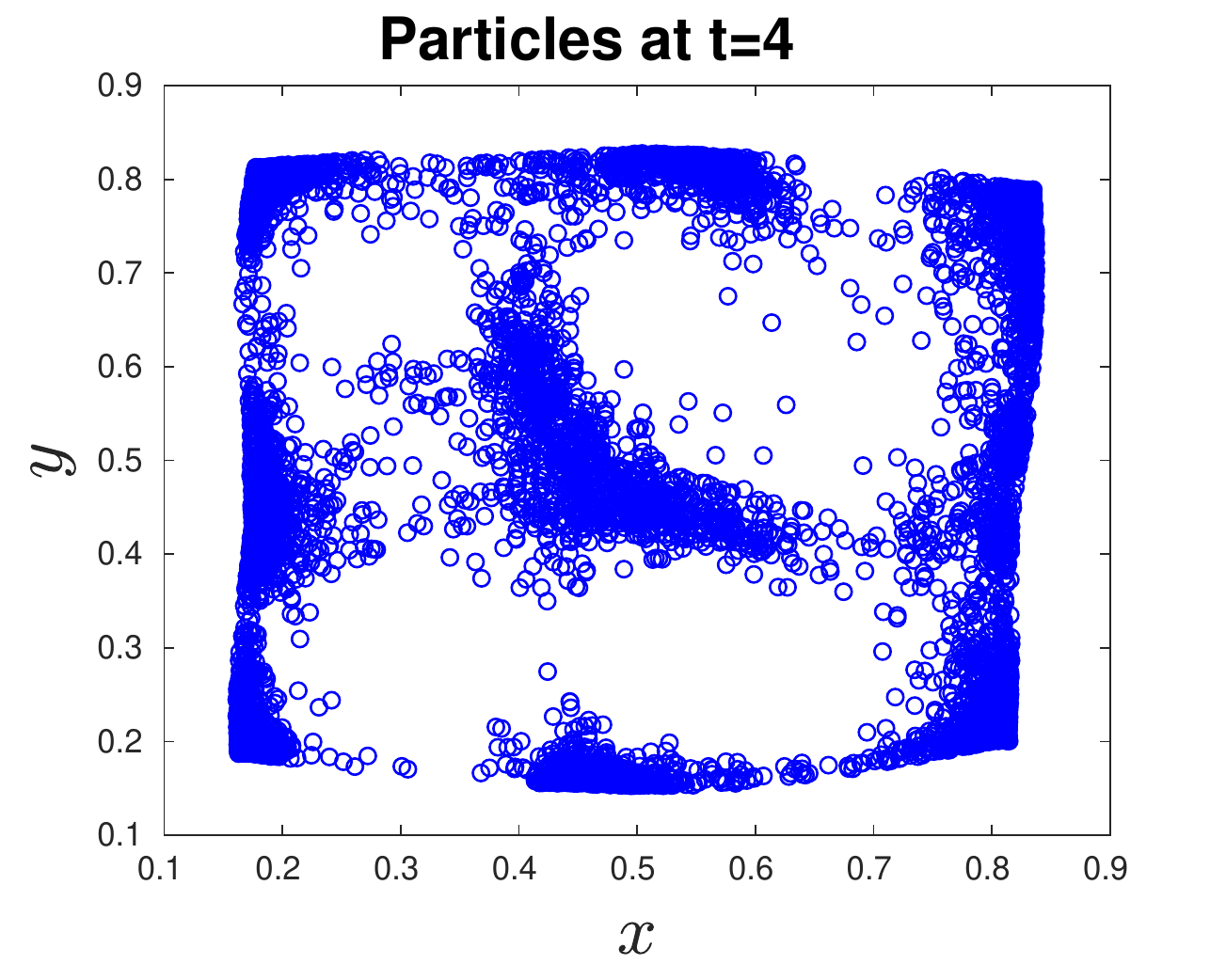}
	\includegraphics[width=0.49\textwidth]{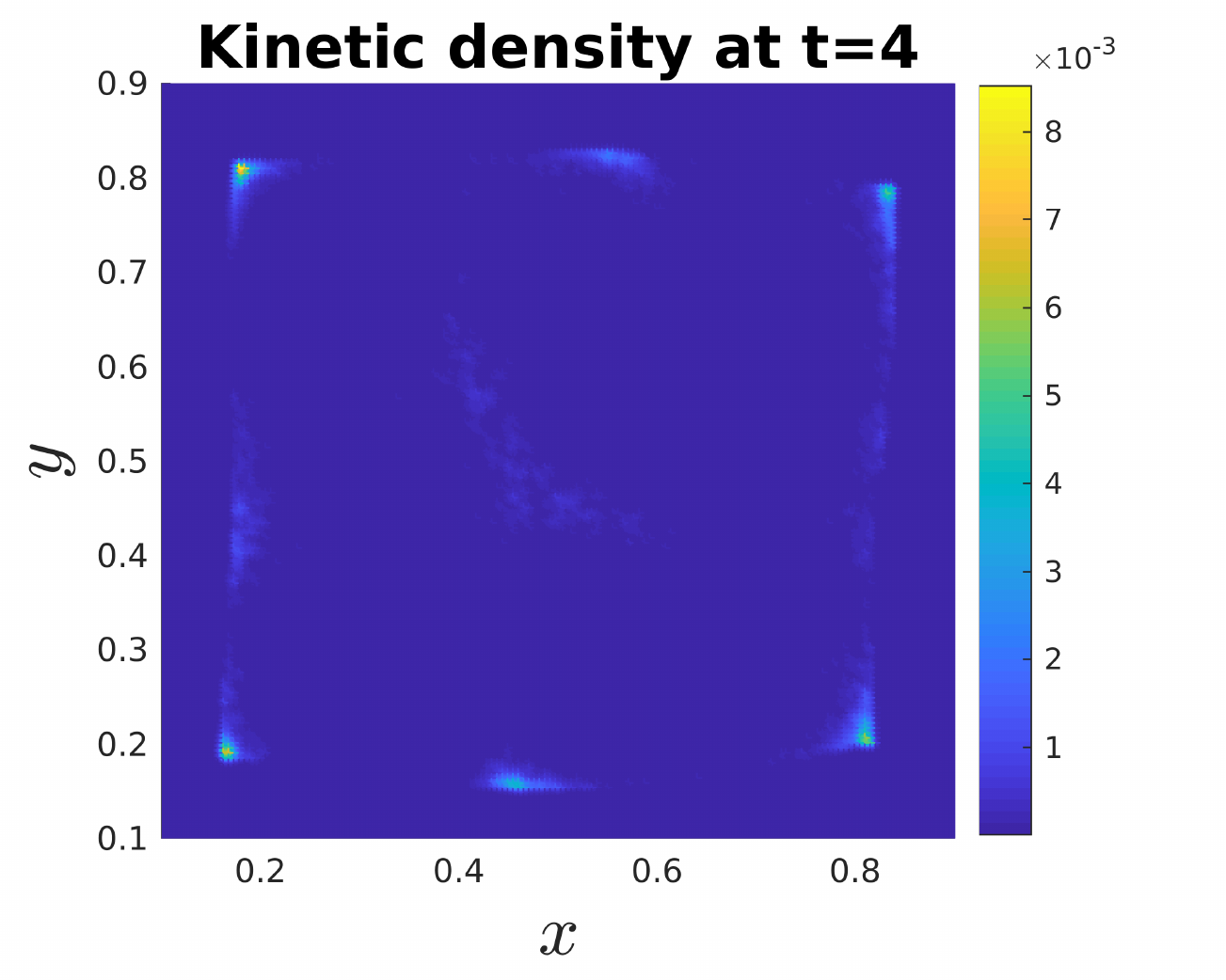}
	\includegraphics[width=0.49\textwidth]{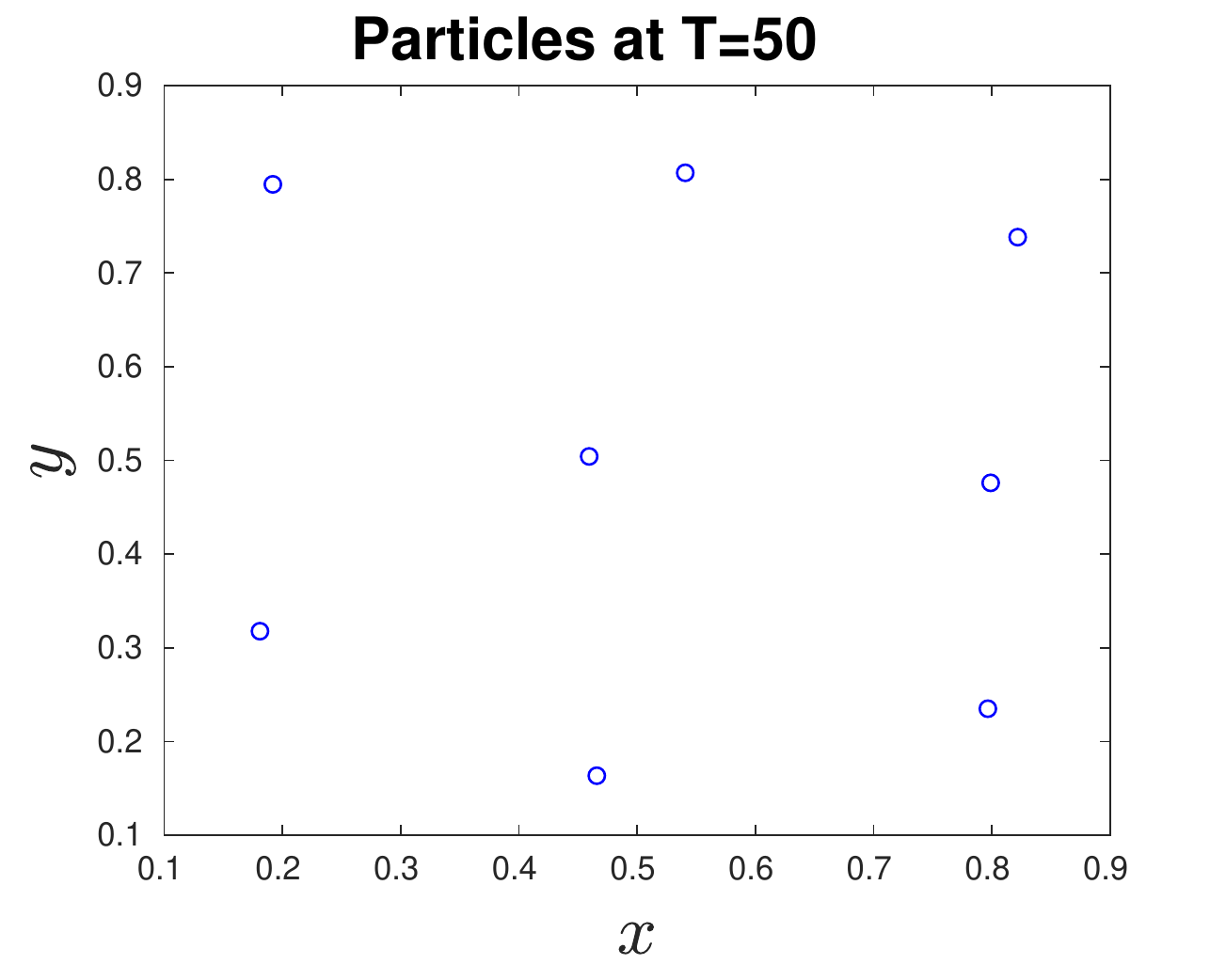}
	\includegraphics[width=0.49\textwidth]{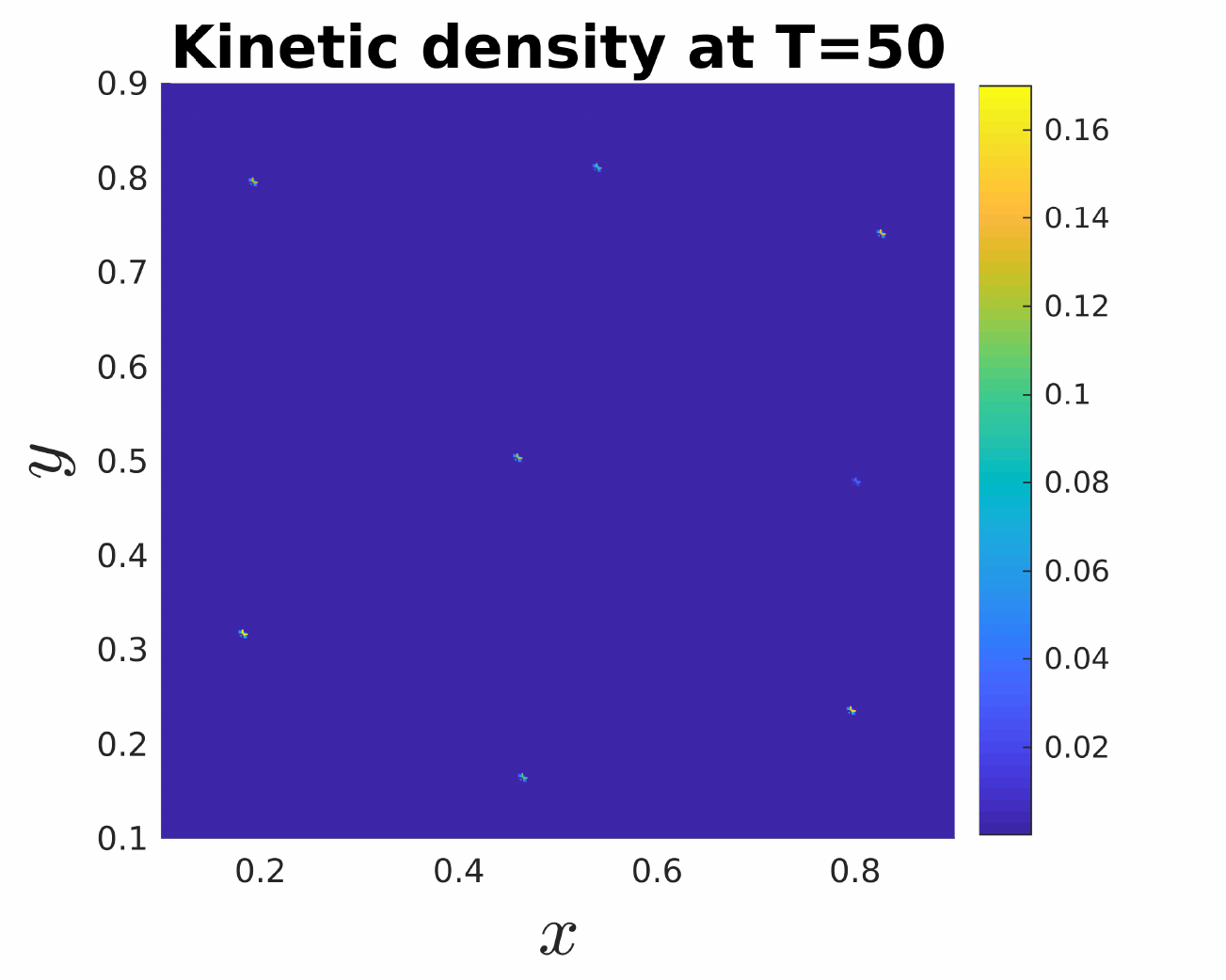}
	\caption{Particle solution (left plots) with $N=10000$ and kinetic density (right plots). Results are provided at time $t=4$ (top row) and final time $T=50$ (bottom row). The bounded confidence level is $\epsilon_1=0.15$.\label{fig:twodimUniformSS}}
\end{figure}

\begin{figure}[t!]
	\centering
	\includegraphics[width=0.49\textwidth]{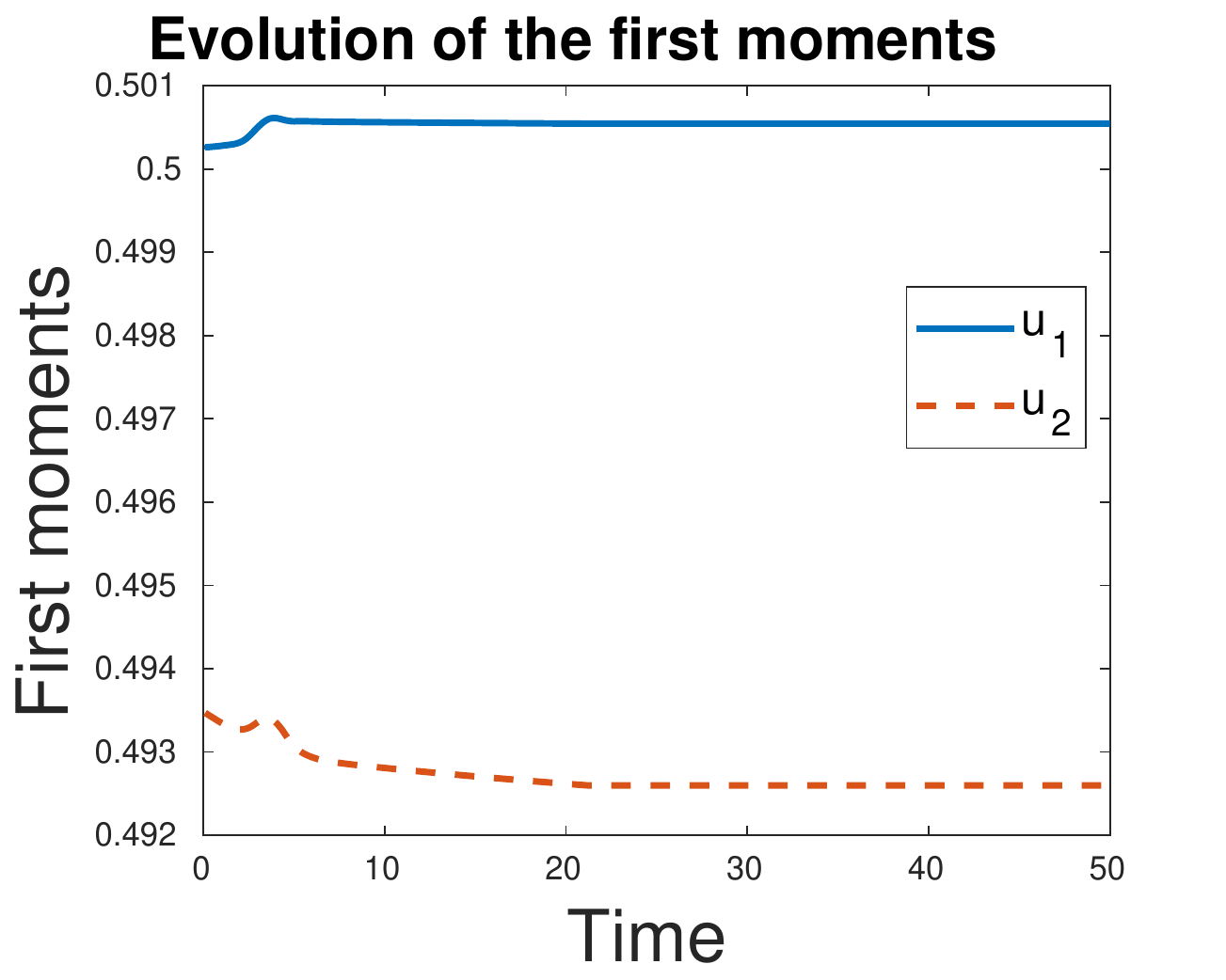}
	\includegraphics[width=0.49\textwidth]{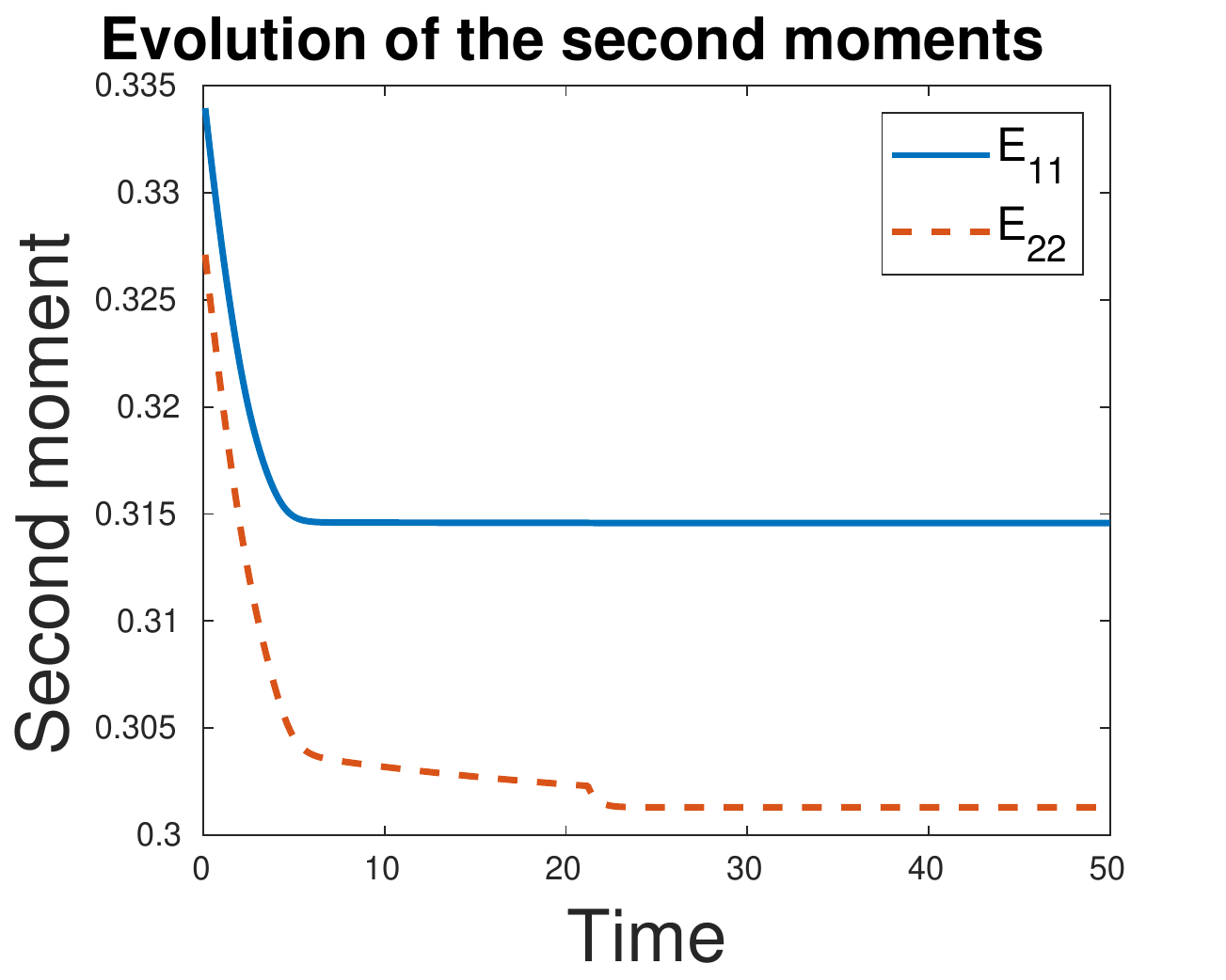}
	\caption{Evolution in time of the two--dimensional first moments (left) and second moments (right) for the bounded confidence level $\epsilon_1=0.15$.\label{fig:twodimUniformMom}}
\end{figure}

\subsubsection{Non--constant static feature.}

Next, we consider a non--constant initial distribution along the variable $\vec{c}$. We consider a one--dimensional setting both along $\vec{x}$ and $\vec{c}$ since the goal of this experiment is to provide evidence of the influence of the static feature on the clustering process. The initial kinetic distribution it given as a tensor product between a uniform and Gaussian distribution so that $$f_0(x,c) = \chi_{[0,1]}(x) \frac{1}{\sqrt{2\pi\sigma^2}}\exp\left(\displaystyle{-\frac{(c-\mu)^2}{2\sigma^2}}\right),$$
with mean $\mu=0.5$, and variance $\sigma^2=0.3$. See the top left plot in Figure~\ref{fig:staticfeature}. All the simulations are performed by using $N=5000$ samples up to equilibrium.

\begin{figure}[t!]
	\centering
	\includegraphics[width=0.49\textwidth]{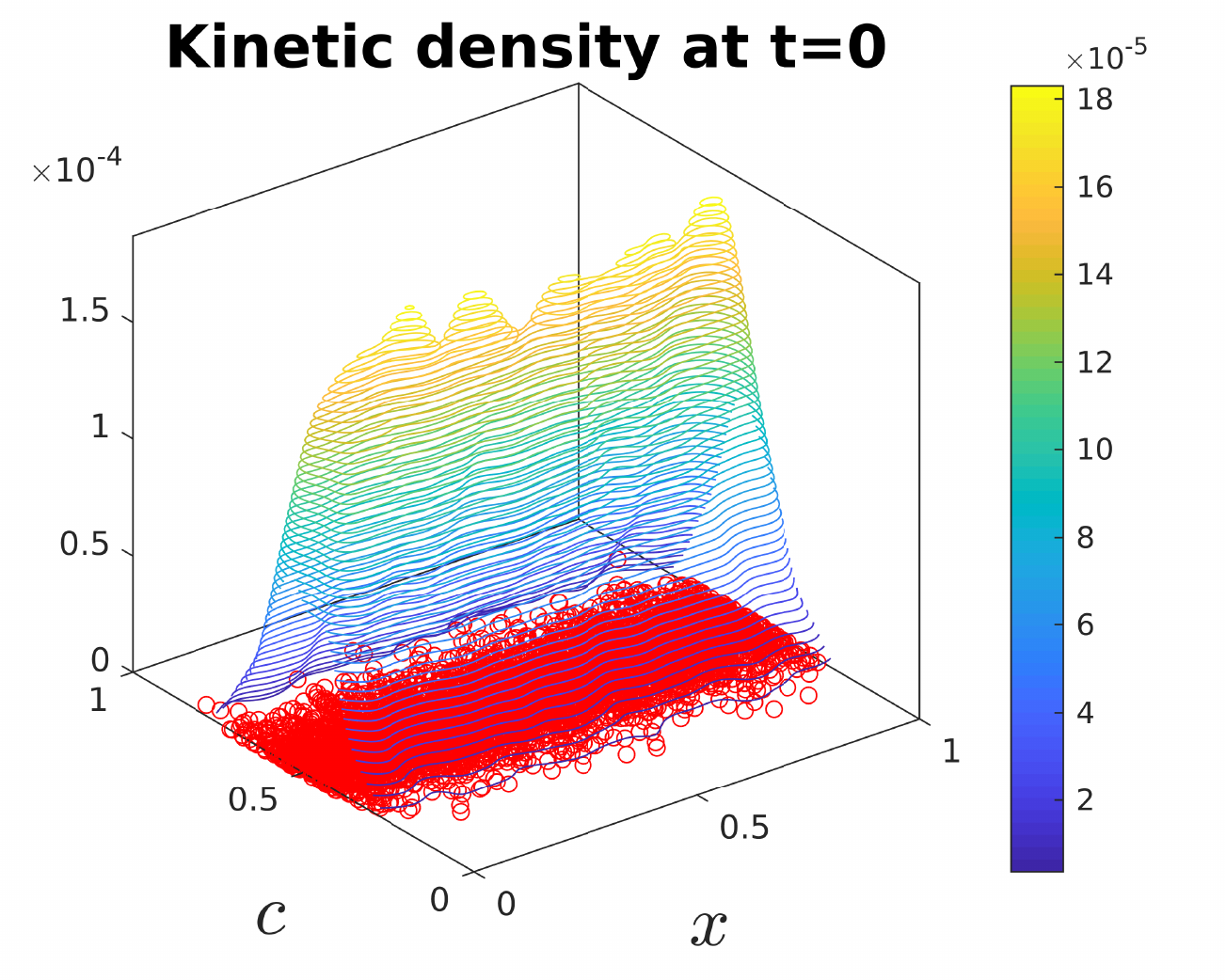}
	\includegraphics[width=0.49\textwidth]{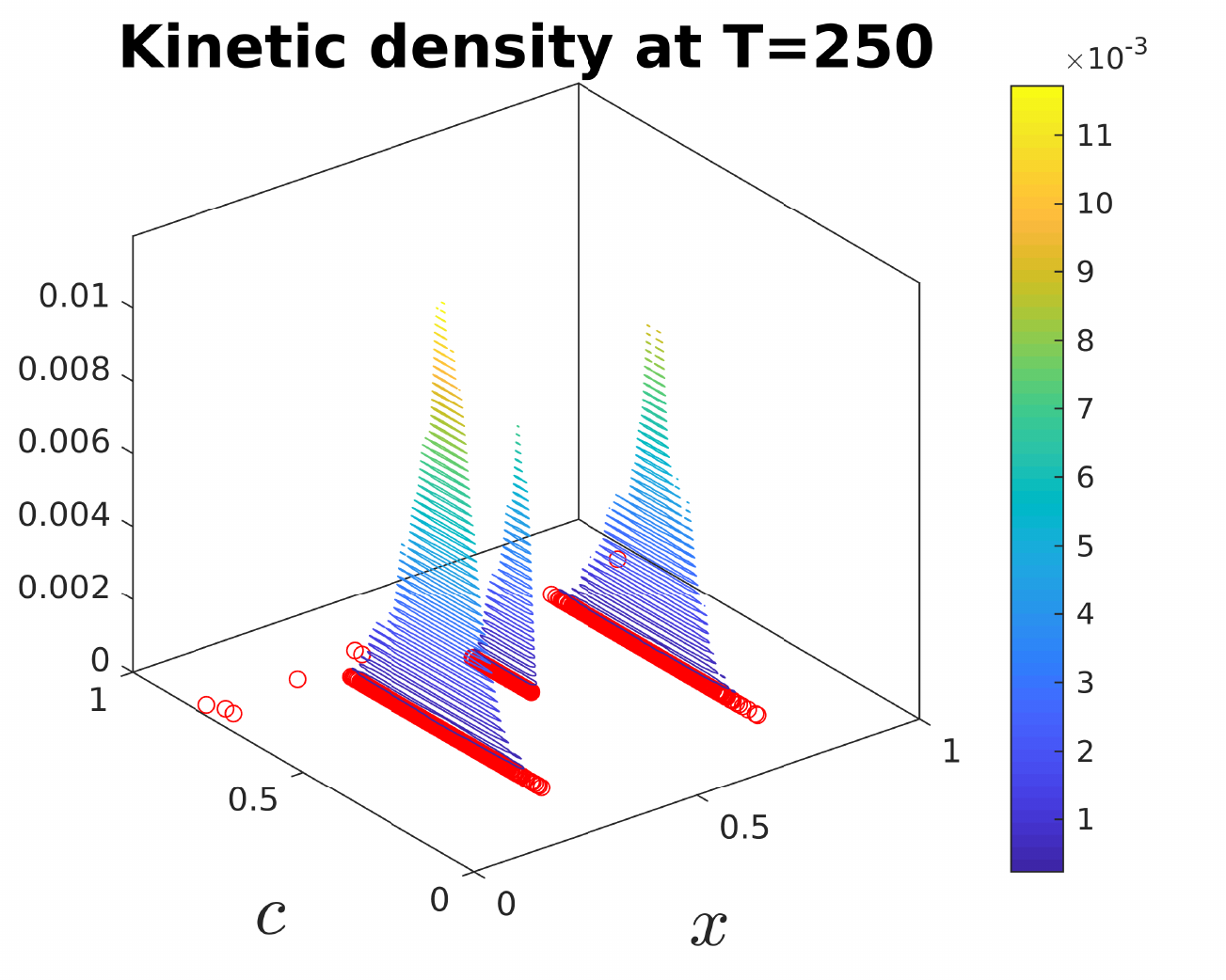}\\
	\includegraphics[width=0.49\textwidth]{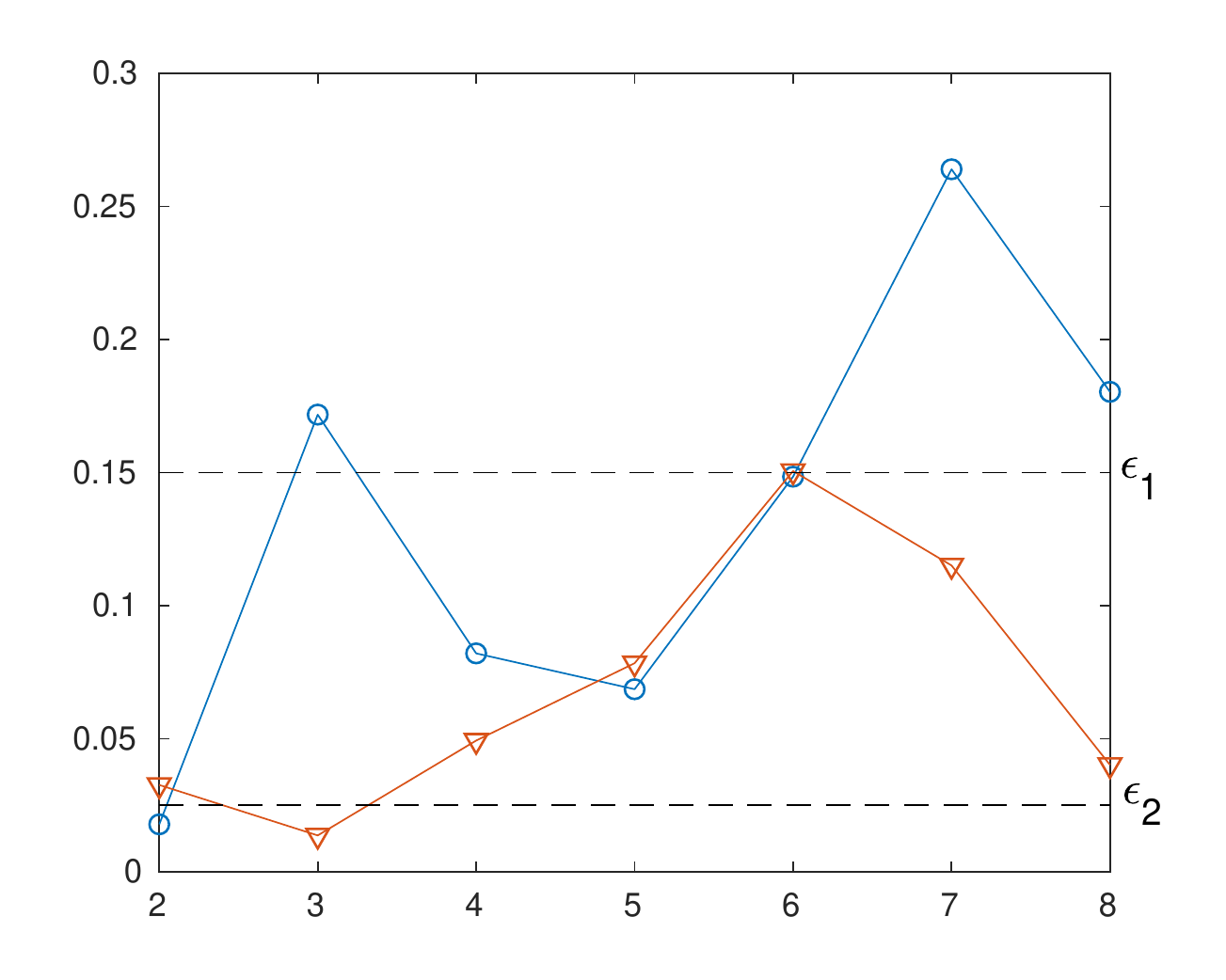}
	\includegraphics[width=0.49\textwidth]{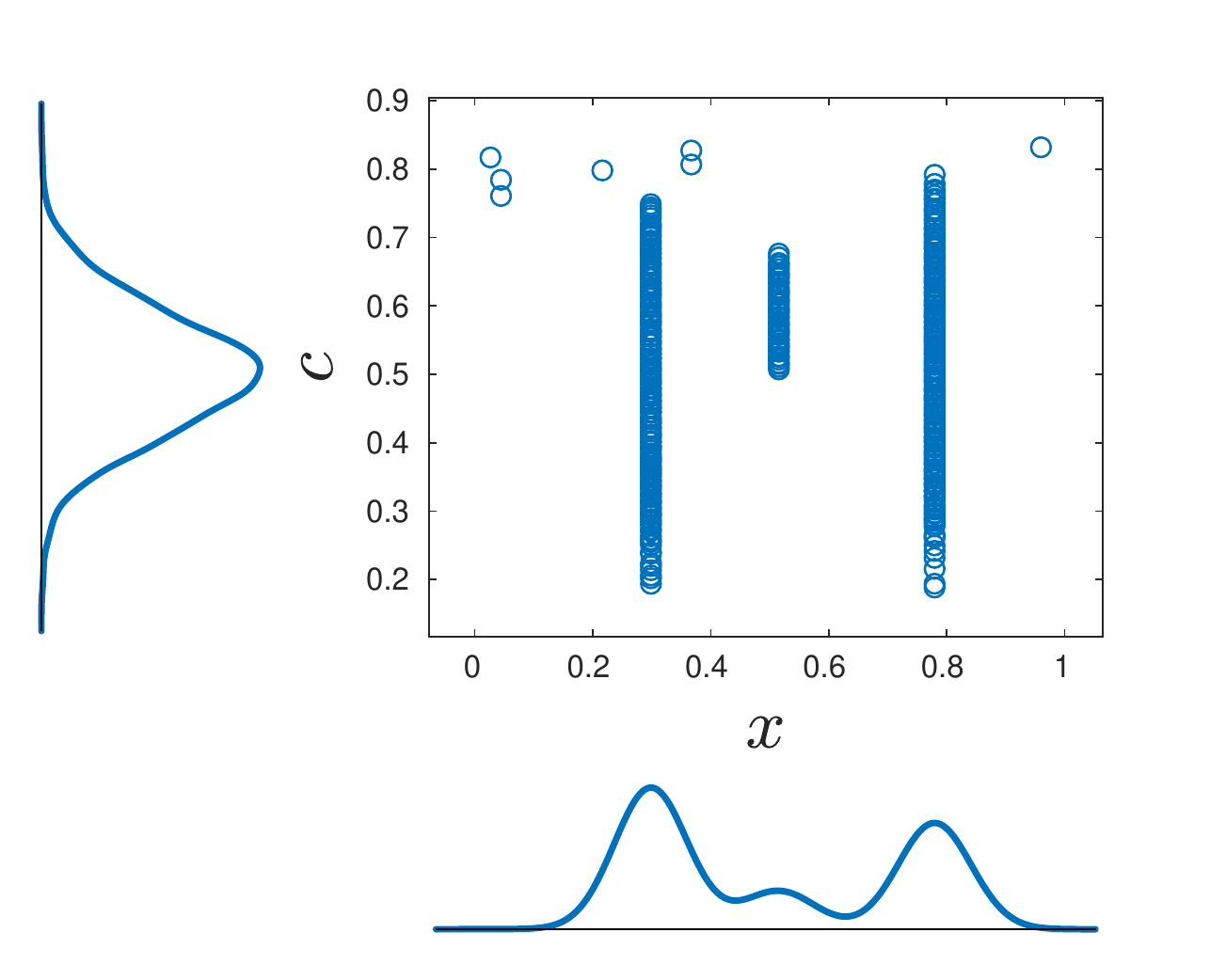}
	\caption{Top row: particles and kinetic density at initial time (left plot) and at equilibrium (right plot). Bottom row: at left, analysis of the distances between clusters in $x$ (blue line with circle markers) and $c$ direction (red line with triangle markers); at right, plot of the marginals. Confidence levels are $\epsilon_1=0.15$ and $\epsilon_2=0.025$.\label{fig:staticfeature}}
\end{figure}

In this experiment, we expect that the behavior at equilibrium is influenced also by the interactions with respect to the static feature variable due to the corresponding characteristic function in the kernel~\eqref{eq:kernelH}. In Figure~\ref{fig:staticfeature} we show the results provided by using $\epsilon_1=0.15$ and $\epsilon_2=0.025$. The top left plot shows particles and kinetic density at equilibrium. Observe that, compared to the case with constant distribution along the static feature, the low value of $\epsilon_2$ allows for the formation of more clusters, precisely $8$. Clusters arise because either their distance in $x$ direction is larger than the confidence level $\epsilon_1$ or the minimum distance in $c$ direction is larger than $\epsilon_2$. This consideration is analyzed in the bottom left plot of Figure~\ref{fig:staticfeature} where the blue line with circle markers shows the distance in $x$ between two adjacent clusters and the red line with triangle markers the minimum distance between the static features of particles being in two adjacent clusters. Observe that when the distance in $x$ is lower than the value $\epsilon_1$, the corresponding distance between the static features is larger than $\epsilon_2$, and therefore cluster is no longer possible. Moreover, we notice that time to reach equilibrium is larger due to two levels of clustering. Finally, the bottom right plot of Figure~\ref{fig:staticfeature} shows the two marginal distributions.

\begin{figure}[t!]
	\centering
	\includegraphics[width=0.49\textwidth]{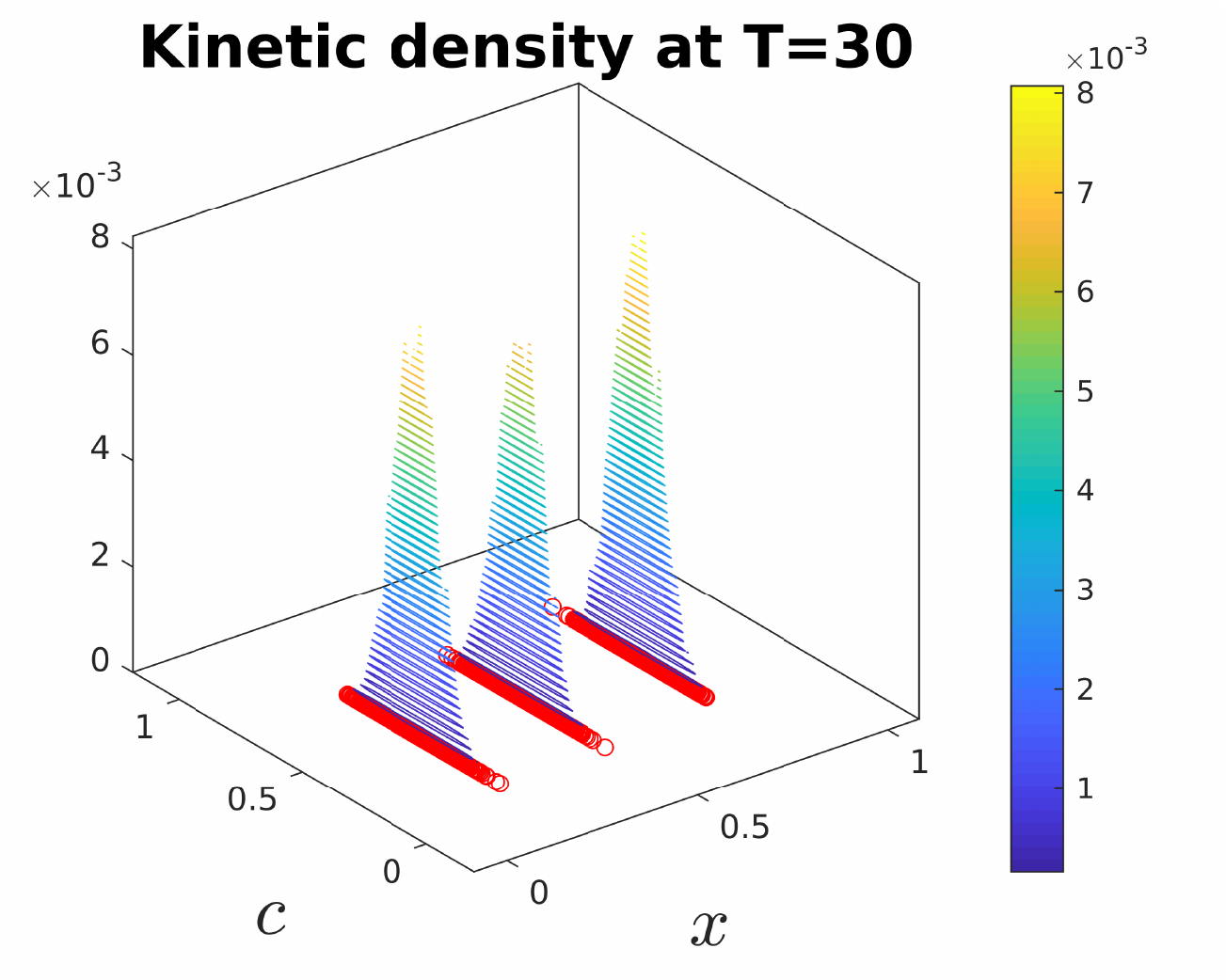}
	\includegraphics[width=0.49\textwidth]{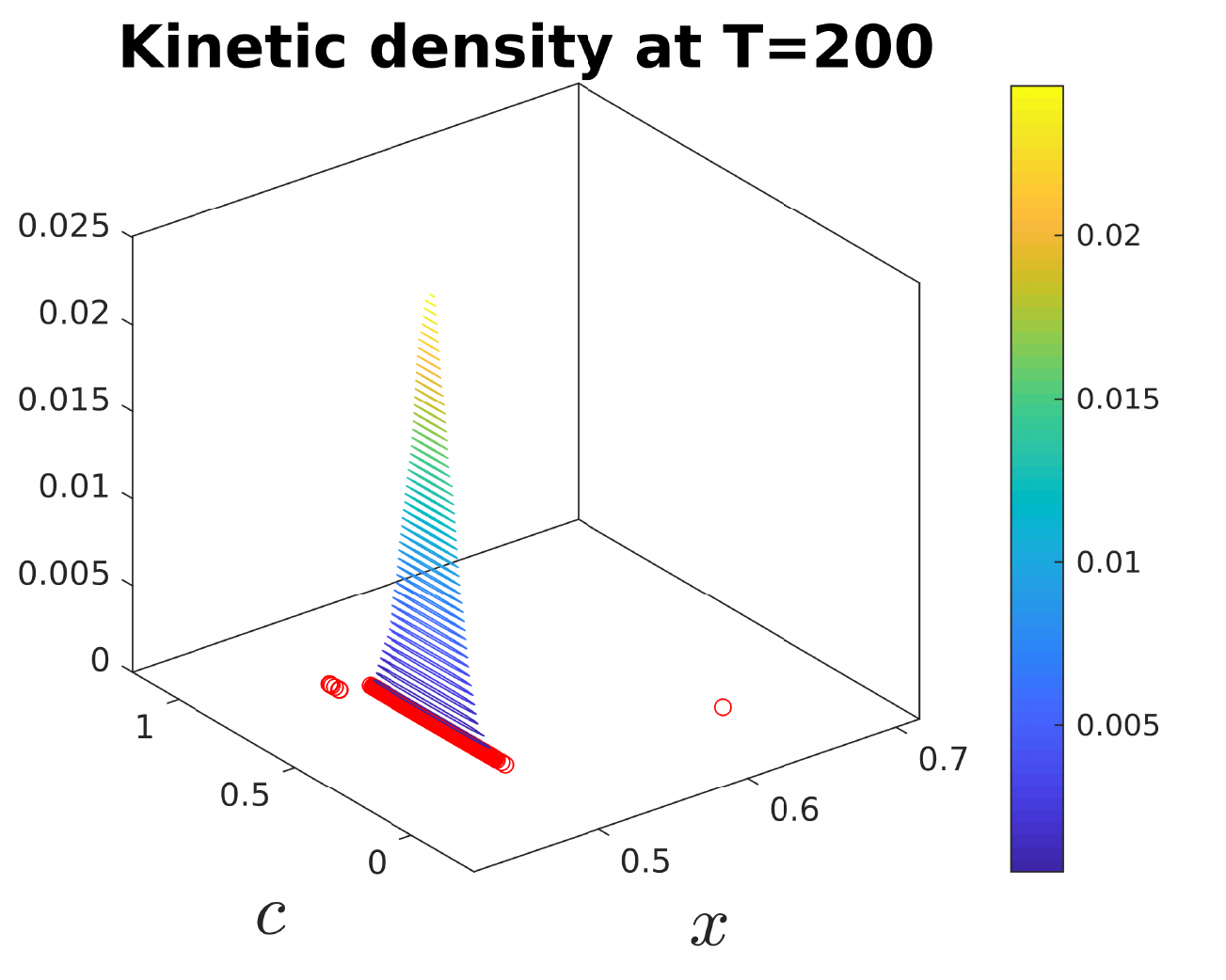}
	\caption{Particle and kinetic density at equilibrium with confidence levels $\epsilon_1=0.15$, $\epsilon_2=0.1$ (left) and $\epsilon_1=1$, $\epsilon_2=0.025$ (right).\label{fig:staticfeature2}}
\end{figure}

In Figure~\ref{fig:staticfeature2} we show the behavior at equilibrium by considering two different pairs of confidence levels. In the left plot, $\epsilon_1=0.15$ and $\epsilon_2=0.1$. Here $\epsilon_2$ is large enough to not influence the clustering process and in fact we recover $3$ clusters, exactly the same provided in Figure~\ref{fig:onedimUniformSS} when the initial kinetic distribution was constant in $c$. In the right plot, $\epsilon_1=1$ and $\epsilon_2=0.025$. In this case, $\epsilon_1$ is very large and we show the formation of clusters at equilibrium when only clustering due to the static feature is allowed.

\subsection{Clustering and shape detection} \label{sec:shapeDetect}

The property of the equation~\eqref{eq:kineticEq}, analytically studied in Section~\ref{sec:kineticProp} and numerically investigated in Section~\ref{sec:numericalProp}, of having quantized steady--states makes the model suitable for solving data clustering problems. In particular, in this section we use the kinetic model as technique for shape detection.

Shape detection can be considered as a branch of pattern recognition problems that focuses on the discovering of previously unknown patterns in a big set of data. Machine learning and many clustering algorithms have already been applied to this class of problems, such as the $k$--means algorithm which however suffers from the necessity of defining a--priori the number of clusters. Instead, as already pointed out, opinion dynamics based models automatically find the number of clusters as functions of the confidence value.

More specifically, in this section we apply the models introduced in Section~\ref{sec:newModel} to the example of a character recognition. We consider a letter ``A'' which is composed by $n$ two-dimensional points $\mathcal{L} = \{\vec{x}_i\}_{i=1}^n$ defining three segments in $\Omega=[0,1]^2$. To each of these points we apply an additive noise distributed according to a uniform or a Gaussian distribution. We obtain a new set of $n$ points
$$
\tilde{\vec{x}}_i = \vec{x}_i + \alpha \boldsymbol{\theta}_i, \quad \alpha>0, \quad \boldsymbol{\theta}_i \sim \mathcal{U}(-1,1) \ \mbox{ or } \ \boldsymbol{\theta}_i \sim \mathcal{N}(0,1)
$$
which define the noisy pattern. Here, the value $\alpha$ represents the percentage of noise and it is also chosen in such a way that $\tilde{\vec{x}}_i\in\Omega$, $i=1,\dots,n$. The points $\{\tilde{\vec{x}}_i\}_{i=1}^n$ are used as initial condition of the model. They are seen as samples from a distribution defining the noisy pattern and the Mean Field Interaction Algorithm~\ref{alg:meanfield} is employed in order to solve the kinetic model.

The goal of this example is to cluster the noisy data into a set of points which give information on the shape of the exact pattern to be detected. This can be considered also as a dimensionality reduction problem which could help other algorithms to recognize the unknown pattern efficiently by using information on the position of the clusters. Here the efficiency of the results is studied by means of the following measure:
$$
	\mathcal{E}(\epsilon_1,\epsilon_2,\alpha) = \frac{1}{\tilde{n}} \sum_{k=1}^{\tilde{n}} \min_{\vec{x}\in\mathcal{L}} \norm{\vec{C}_k - \vec{x}}_2
$$
where $\{ \vec{C}_k \}_{k=1}^{\tilde{n}}$ is the position of clusters in $\Omega$, $\tilde{n}$ the number of clusters at equilibrium. The quantity $\mathcal{E}$ measures for each cluster the minimum $2$--norm distance to the exact pattern and then computes the normalized $1$--norm of these quantities. \revision{The definition of the measure $\mathcal{E}$ is inspired by the idea of computing an average of the minimum distances between the clusters and the shape $\mathcal{L}$, when formation of multiple clusters distributed close the shape $\mathcal{L}$ arises, as in the cases showed here. Certainly, $\mathcal{E}$ cannot provide a general way to measure the quality of the results, since e.g.~in case of one steady-state cluster positioned exactly on a point of $\mathcal{L}$, we would have $\mathcal{E}=0$ providing absurdly the optimal choice.}

In the following, all the results are given for an initial set of $n=5000$ particles and a large enough final time to reach a steady--state $T=50$. Moreover, we use the Euclidean norm and all particles are initialized with a constant static feature. For an extension of this application to clustering also with respect \revision{to} the static feature see Remark~\ref{rm:staticFeatureShapeDetect} below.

\begin{table}[t!]
	\begin{center}
		\caption{Number of clusters and errors depending on the bounded confidence value and the percentage of the noise when uniformly distributed.}
		\label{tab:shapeDetectUniform}
		\small
		\setlength{\tabcolsep}{5pt}
		\begin{tabular}{|ccc|ccc|ccc|}
			\hline
			\multicolumn{3}{|c|}{$\alpha = 5\%$} & \multicolumn{3}{c|}{$\alpha = 7.5\%$} & \multicolumn{3}{c|}{$\alpha = 10\%$}
			\\
			$\epsilon_1$ & $\mathcal{E}$ & $\tilde{n}$ & $\epsilon_1$ & $\mathcal{E}$ & $\tilde{n}$ & $\epsilon_1$ & $\mathcal{E}$ & $\tilde{n}$
			\\
			\hline
			0.03 & 1.25e-02 & 30 & 0.03 & 3.47e-02 & 51 & 0.06 & 2.64e-02 & 11 \\
			0.04 & 4.10e-03 & 16 & 0.05 & 1.21e-02 & 14 & 0.07 & 1.48e-02 & 8 \\
			\textbf{0.05} & \textbf{4.00e-03} & \textbf{12} & \textbf{0.07} & \textbf{7.70e-03} & \textbf{8} & \textbf{0.08} & \textbf{1.12e-02} & \textbf{8} \\
			0.06 & 4.60e-03 & 9 & 0.09 & 7.90e-03 & 8 & 0.09 & 1.63e-02 & 5 \\
			0.07 & 5.40e-03 & 8 & 0.11 & 1.66e-02 & 3 & 0.10 & 1.63e-02 & 5 \\
			\hline
		\end{tabular}
	\end{center}
\end{table}

\begin{figure}[t!]
	\centering
	\includegraphics[width=0.49\textwidth]{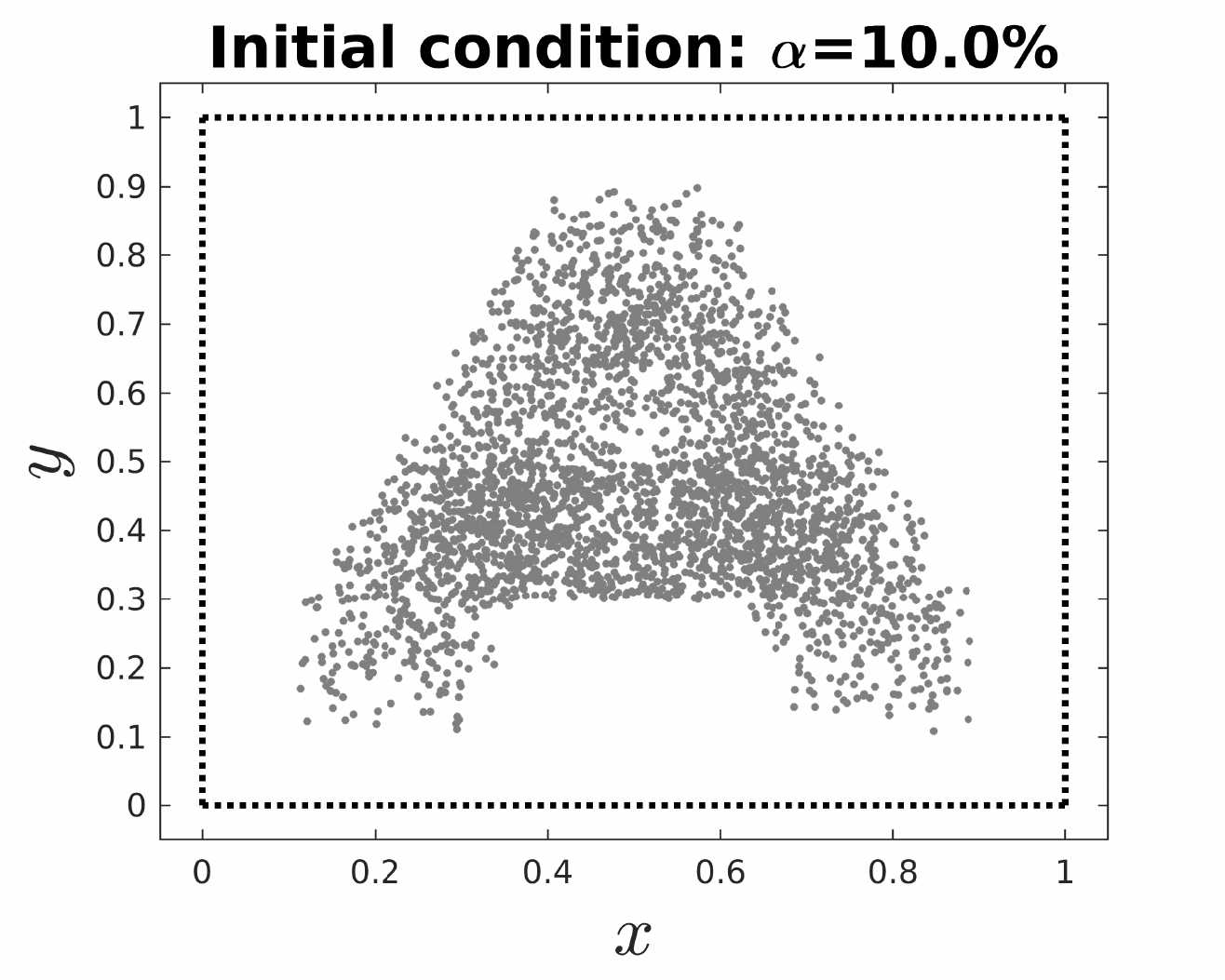}
	\includegraphics[width=0.49\textwidth]{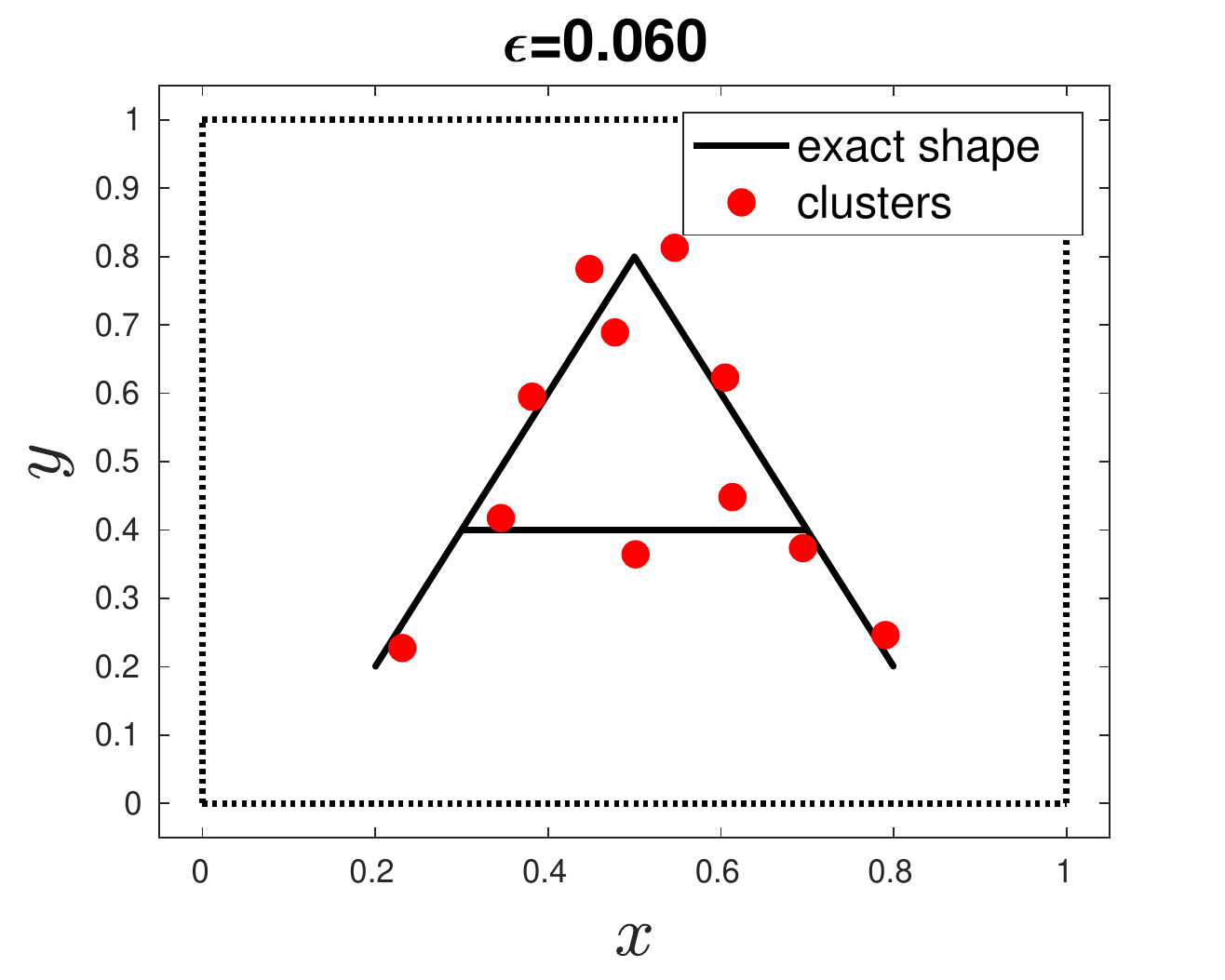}
	\includegraphics[width=0.49\textwidth]{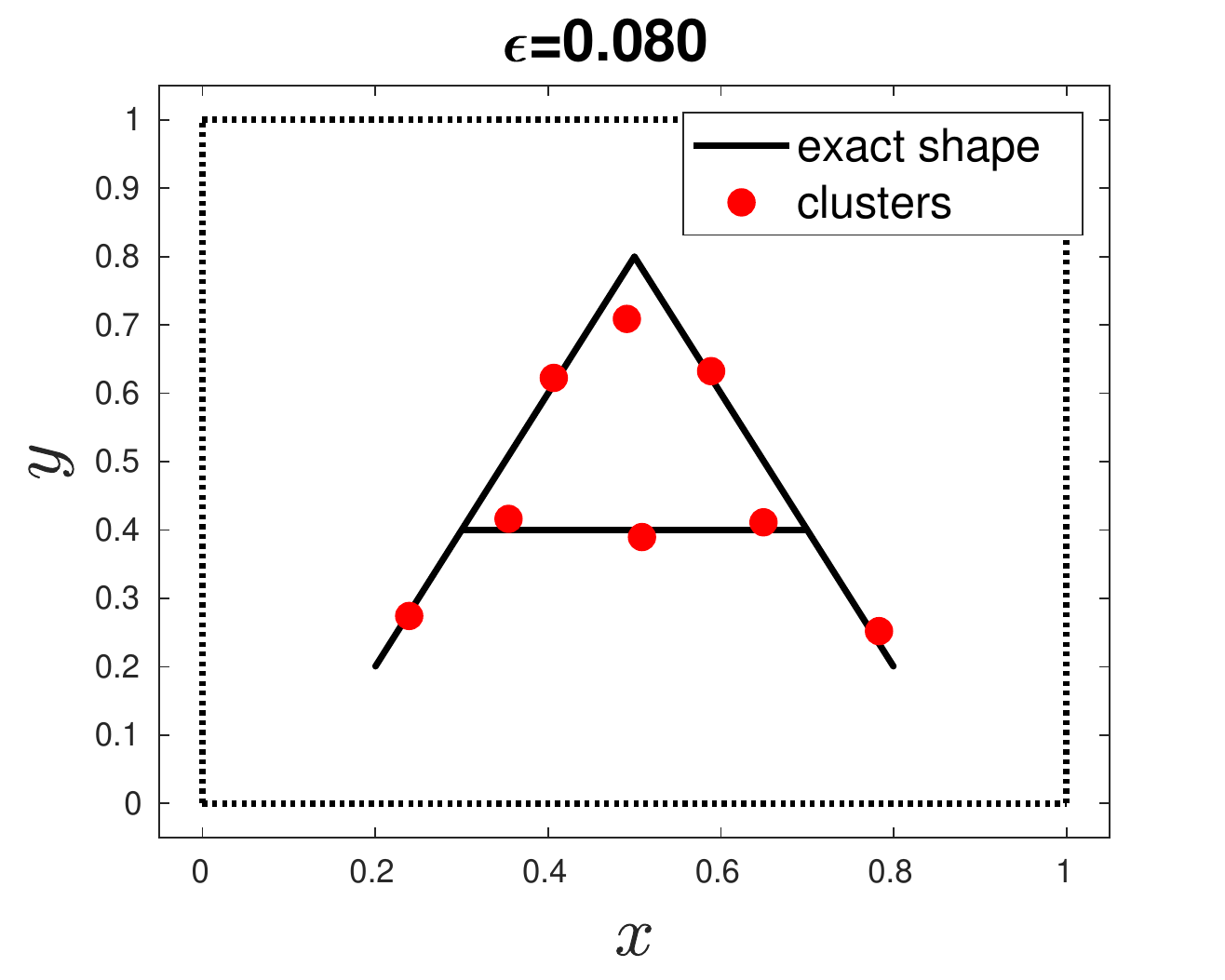}
	\includegraphics[width=0.49\textwidth]{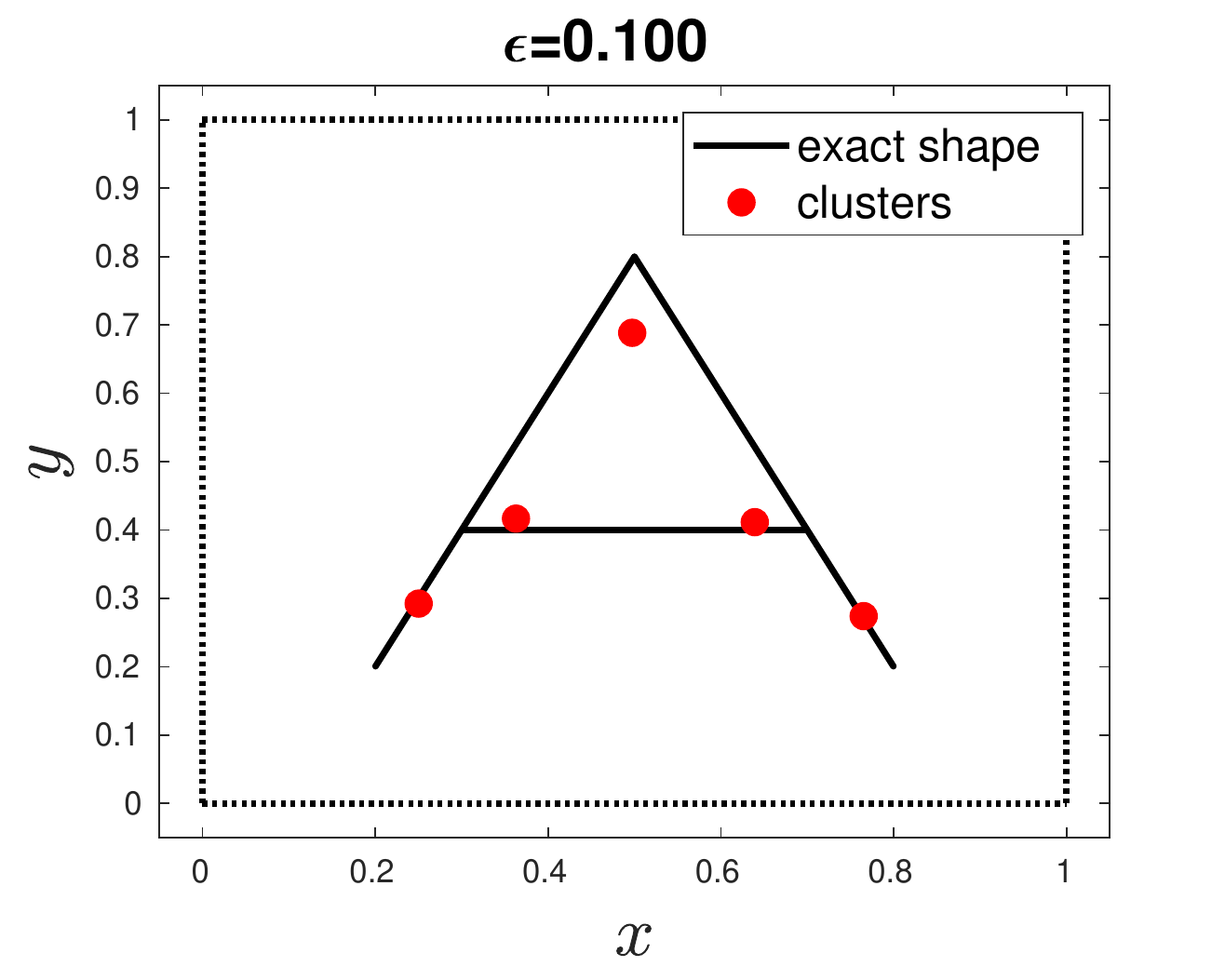}
	\caption{Shape detection of the letter ``A'' initialized with $10\%$ of a uniform additive noise. Top left panel shows the initial condition. We show clusters obtained with bounded confidence values $\epsilon_1=0.06$ (top right), $\epsilon_1=0.08$ (bottom left) and $\epsilon_1=0.1$ (bottom right). \label{fig:shapeDetectUniform}}
\end{figure}

\begin{table}[t!]
	\begin{center}
		\caption{Number of clusters and errors depending on the bounded confidence value and the percentage of the noise when normally distributed.}
		\label{tab:shapeDetectGaussian}
		\small
		\setlength{\tabcolsep}{5pt}
		\begin{tabular}{|ccc|ccc|ccc|}
			\hline
			\multicolumn{3}{|c|}{$\alpha = 5\%$} & \multicolumn{3}{c|}{$\alpha = 5.5\%$} & \multicolumn{3}{c|}{$\alpha = 6\%$}
			\\
			$\epsilon_1$ & $\mathcal{E}$ & $\tilde{n}$ & $\epsilon_1$ & $\mathcal{E}$ & $\tilde{n}$ & $\epsilon_1$ & $\mathcal{E}$ & $\tilde{n}$
			\\
			\hline
			0.05 & 4.44e-02 & 23 & 0.05 & 4.73e-02 & 24 & 0.05 & 6.37e-02 & 30 \\
			0.06 & 1.36e-02 & 11 & 0.06 & 2.62e-02 & 13 & 0.06 & 4.16e-02 & 16 \\
			\textbf{0.065} & \textbf{6.40e-03} & \textbf{9} & 0.065 & 1.63e-02 & 11 & 0.07 & 2.12e-02 & 9 \\
			0.07 & 6.70e-03 & 7 & \textbf{0.0675} & \textbf{7.40e-03} & \textbf{10} & 0.075 & 9.70e-03 & 7 \\
			0.08 & 8.50e-03 & 7 & 0.07 & 8.00e-03 & 9 & \textbf{0.08} & \textbf{9.20e-03} & \textbf{7} \\
			0.09 & 1.00e-02 & 6 & 0.08 & 9.80e-03 & 7 & 0.085 & 1.10e-02 & 5 \\
			\hline
		\end{tabular}
	\end{center}
\end{table}

\begin{figure}[t!]
	\centering
	\includegraphics[width=0.49\textwidth]{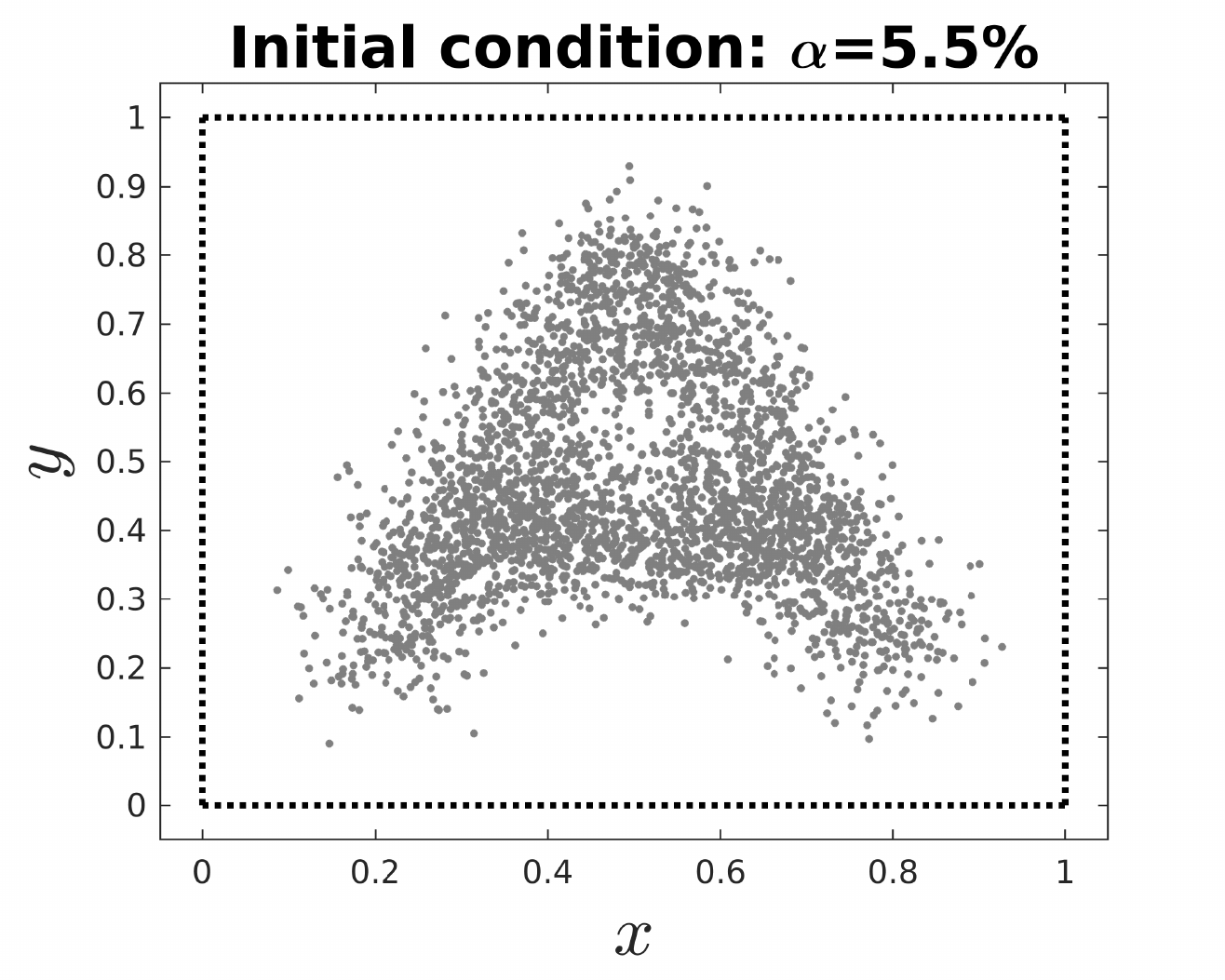}
	\includegraphics[width=0.49\textwidth]{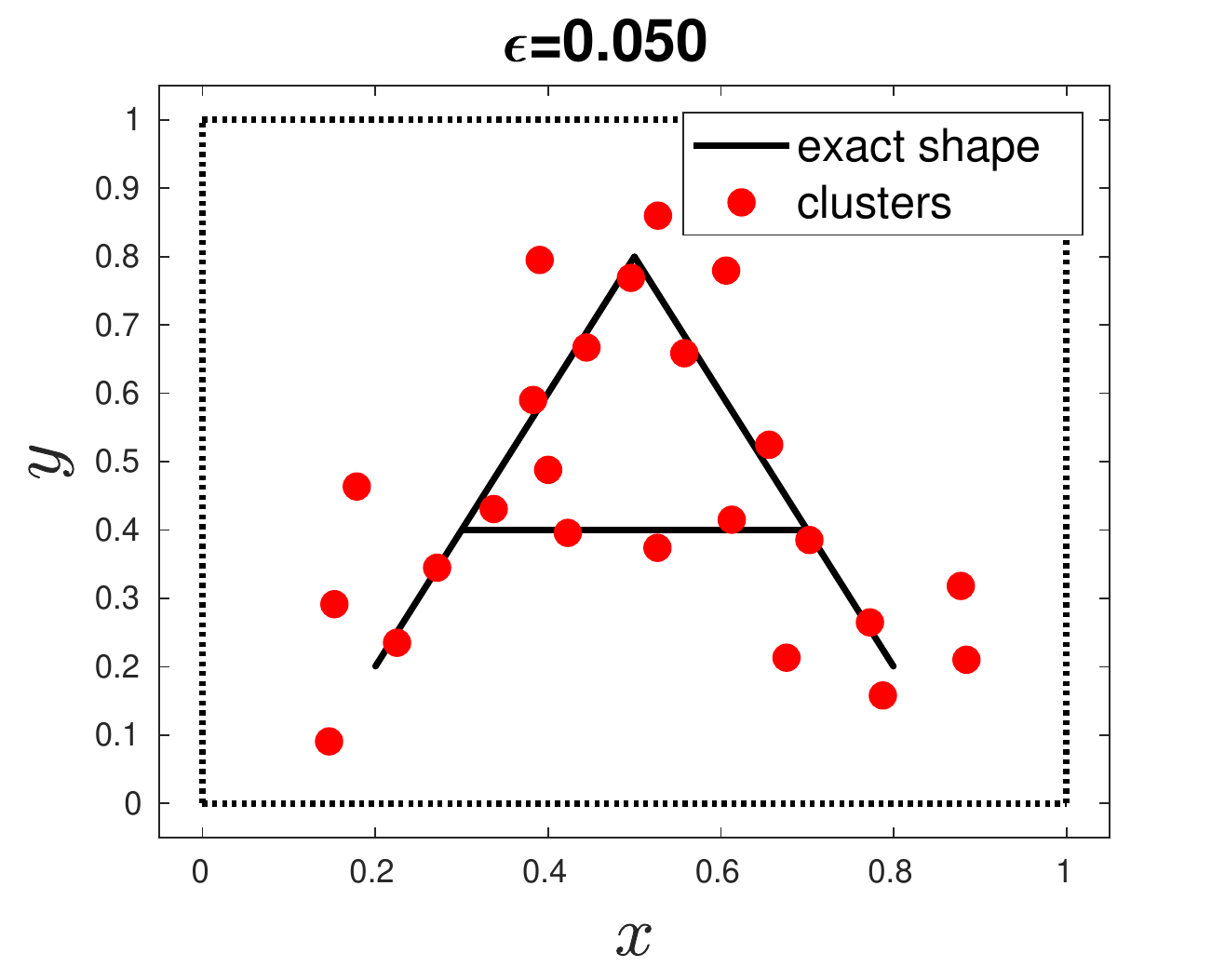}
	\includegraphics[width=0.49\textwidth]{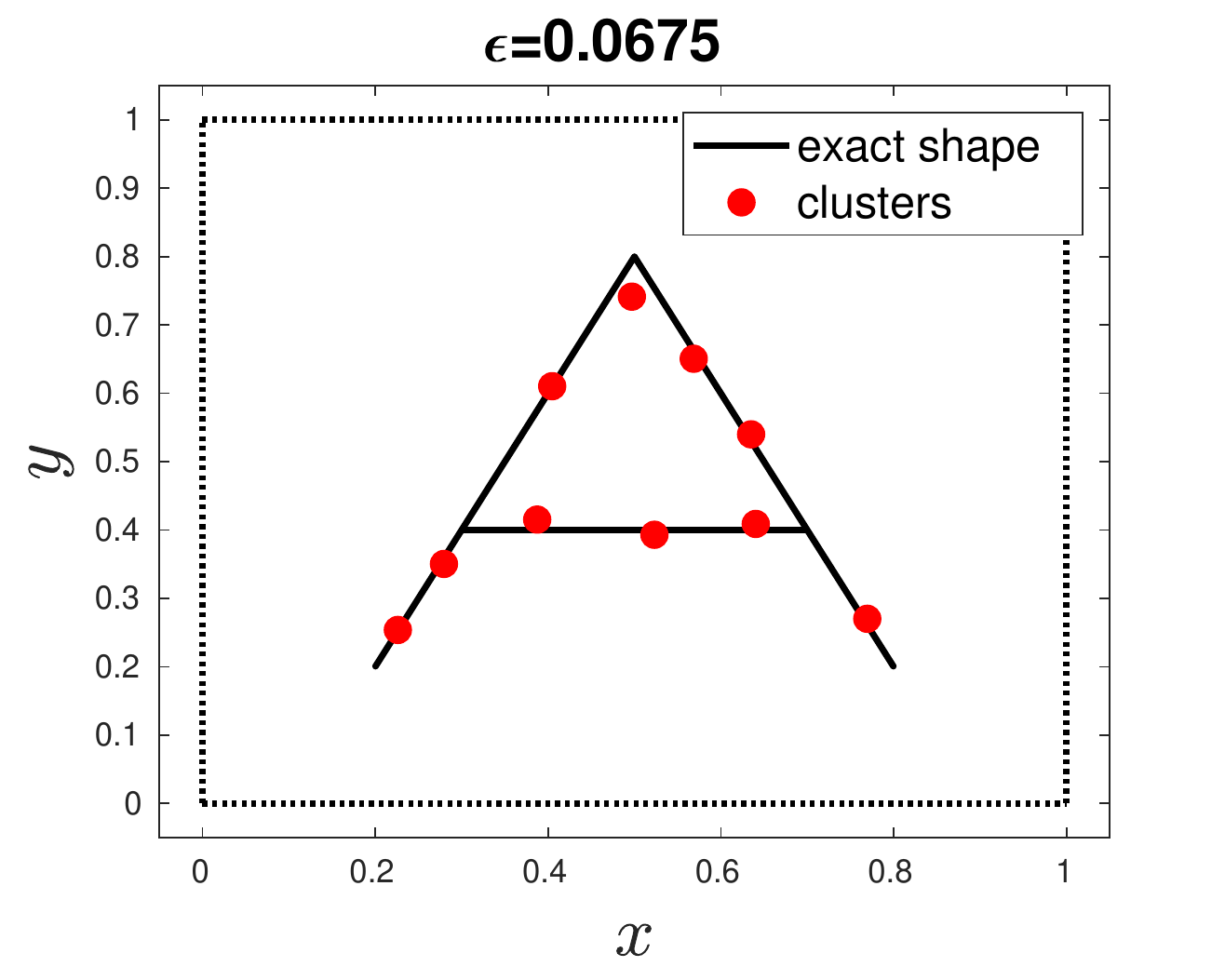}
	\includegraphics[width=0.49\textwidth]{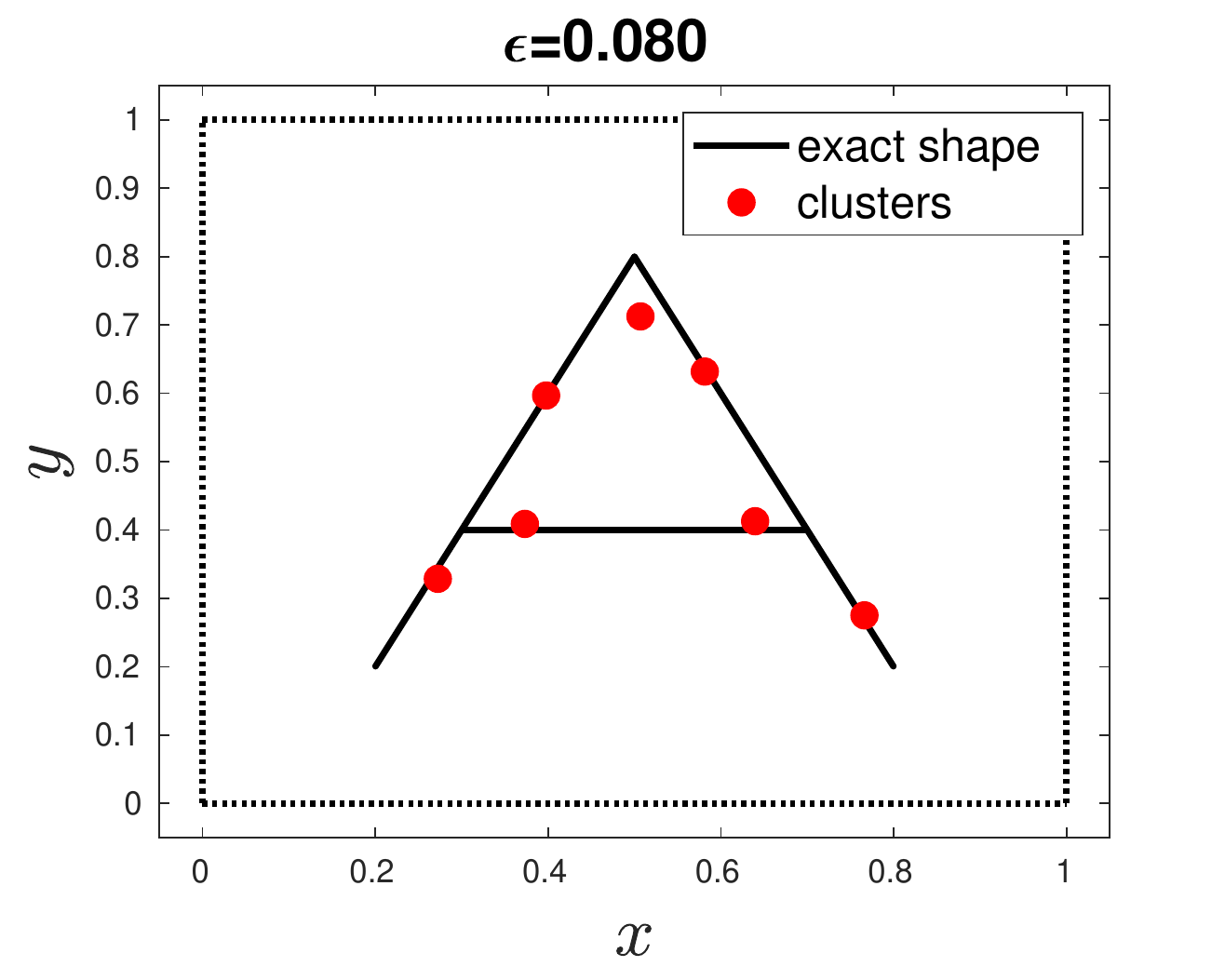}
	\caption{Shape detection of the letter ``A'' initialized with $5.5\%$ of a Gaussian additive noise. Top left panel shows the initial condition. We show clusters obtained with bounded confidence values $\epsilon_1=0.05$ (top right), $\epsilon_1=0.0675$ (bottom left) and $\epsilon_1=0.08$ (bottom right)..\label{fig:shapeDetectGaussian}}
\end{figure}

In Table~\ref{tab:shapeDetectUniform} and Table~\ref{tab:shapeDetectGaussian} we show a sensitivity study of the error $\mathcal{E}$ as function of the percentage of noise $\alpha$ and of the bounded confidence value $\epsilon_1$, for initial noisy data uniformly and normally distributed, respectively. The number of clusters at equilibrium is also provided. Results with minimal error value are highlighted with bold font. In particular, we point out as the increasing percentage of noise highly affects the results.

In Figure~\ref{fig:shapeDetectUniform} and Figure~\ref{fig:shapeDetectGaussian} some of this cases are shown. Precisely, for the uniformly distributed initial noisy data, Figure~\ref{fig:shapeDetectUniform}, we consider the largest value of the percentage of noise $\alpha=10\%$ and three values of the bounded confidence level $\epsilon_1$: \revision{the smallest $\epsilon_1=0.06$ in Table~\ref{tab:shapeDetectUniform}, the optimal $\epsilon_1=0.08$ (giving the smallest value of the error $\mathcal{E}$) and the largest one $\epsilon_1=0.1$.} We observe that the result is highly influenced by the choice of the confidence bound which determines the number of final clusters. Smaller and larger values of $\epsilon_1$ result in larger values of the error measure $\mathcal{E}$. In fact, in one case the model provides too many clusters which are very spread out, with several outliers. In the other case too few clusters so that all the information on the exact shape is lost. For the optimal value of $\epsilon_1=$ the model provides a good agreement with the exact pattern.

Similar considerations hold for Figure~\ref{fig:shapeDetectGaussian} where we show the noisy initial data obtained with $\alpha=10\%$ percentage of noise and three values of the bounded confidence level $\epsilon_1=0.05$, $\epsilon_1=0.0675$ (giving the smallest value of the error $\mathcal{E}$) and $\epsilon_1=0.08$.

We point out that the goal of this application is to provide a preliminary step for shape feature extraction. This technique should be coupled with a further algorithm which is able to select the correct letter from a given alphabet. The dimensionality reduction provided by the data clustering approach allows to efficiently employ the subsequent recognition analysis.


\begin{remark}[Shape detection with non--constant static feature] \label{rm:staticFeatureShapeDetect}
	In the previous examples the static features is taken constant for all the data. In the non--constant case, one could think of this feature as an additional initial given input which measures the quality of the information of each single point and that can be modeled by the distance to its exact value in $\mathcal{L}$. \end{remark}

\subsection{Color image segmentation} \label{sec:imageSeg}

We now turn to the second application of the data clustering model~\ref{eq:kineticEq} which is the segmentation of gray scale images. Image segmentation is widely used in medical and astronomical image processing, face recognition, etc. In fact, this technique allows to partition \revision{an} image in sets of significant regions of pixels sharing same characteristics such as closeness and similar intensity of color. As in shape detection, the aim of such approach is to modify the description of an image into a structure which easier to be analyzed for subsequent computer vision algorithms.

Several mathematical techniques have been applied to image segmentation. For instance we mention level set methods~\cite{Osher1988,Sethian1999} methods based on the Kuramoto model~\cite{2016Liuetal,2014NovikovBenderskaya} and supervised convolution neural networks~\cite{Chenetal2018,Zhangetal2014}. A more complicated procedure to segmentation of images is the clustering approach \revision{and} we refer e.g.~to the $k$--means method~\cite{Shan2018,Zheng2018}, $c$--means method~\cite{Memon2019} and hierarchical clustering method~\cite{Tian2016}.

Here, we study the efficiency of the kinetic model~\eqref{eq:kineticEq} to solve image segmentation problems and therefore we still relies on the clustering technique. Thanks to the model introduced in this paper, segmentation can be performed by using two levels of clustering in order to determine regions of pixels with similar characteristics. More precisely, we cluster based on the \revision{Euclidean} distance and the distance of the gray intensity between pixels. Taking into account spatial coordinates is needed in order to avoid the problem of selecting homogeneous regions of pixels which are however distinct in the original image. The clustering with respect the intensity of the gray colors is performed by using the static feature variable $\vec{c}$.

The application of~\eqref{eq:kineticEq} works as described in the following. Number of pixels is the number of particles which are assumed to be equally spaced samples. The intensity of the gray color of each pixel is computed at initial time and defines the one--dimensional static feature $c$ of each particle. The Mean Field Interaction Algorithm~\ref{alg:meanfield} is then applied to find clusters representing local regions of the original image with homogeneous intensity of color. At equilibrium, we compute the mean of the color intensity of pixels belonging to the same cluster and then they are mapped back to the original positions in the initial image to get the segmentation. Another possible approach, which is more suitable for images with very sharp color intensities, such as in astronomical images, is to apply thresholding technique. This method is based on replacing the color intensity of each cluster by black if the mean is below a certain threshold or by white otherwise. This technique results therefore in binary images.

\begin{figure}[t!]
	\centering
	\includegraphics[width=0.32\textwidth]{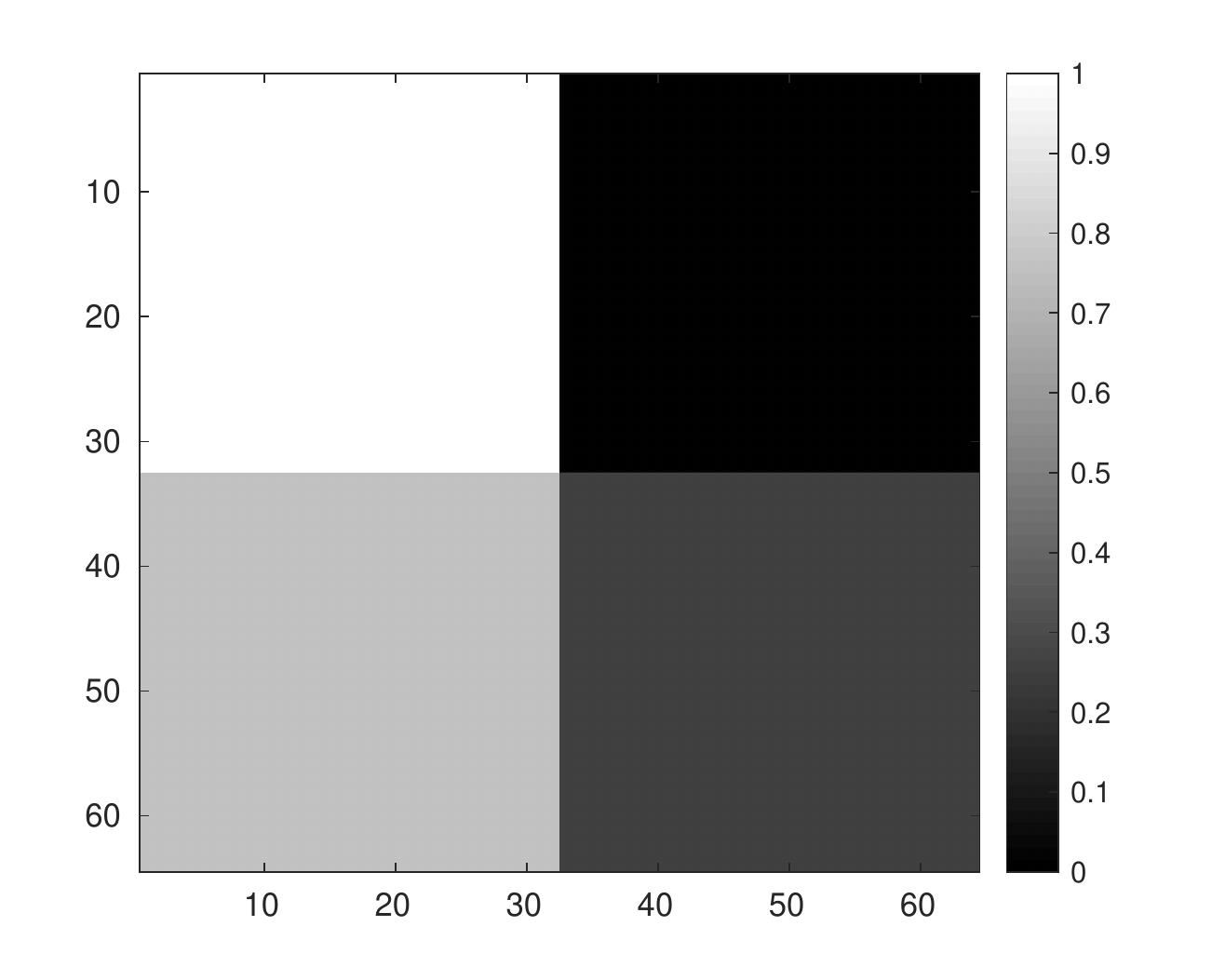}
	\includegraphics[width=0.32\textwidth]{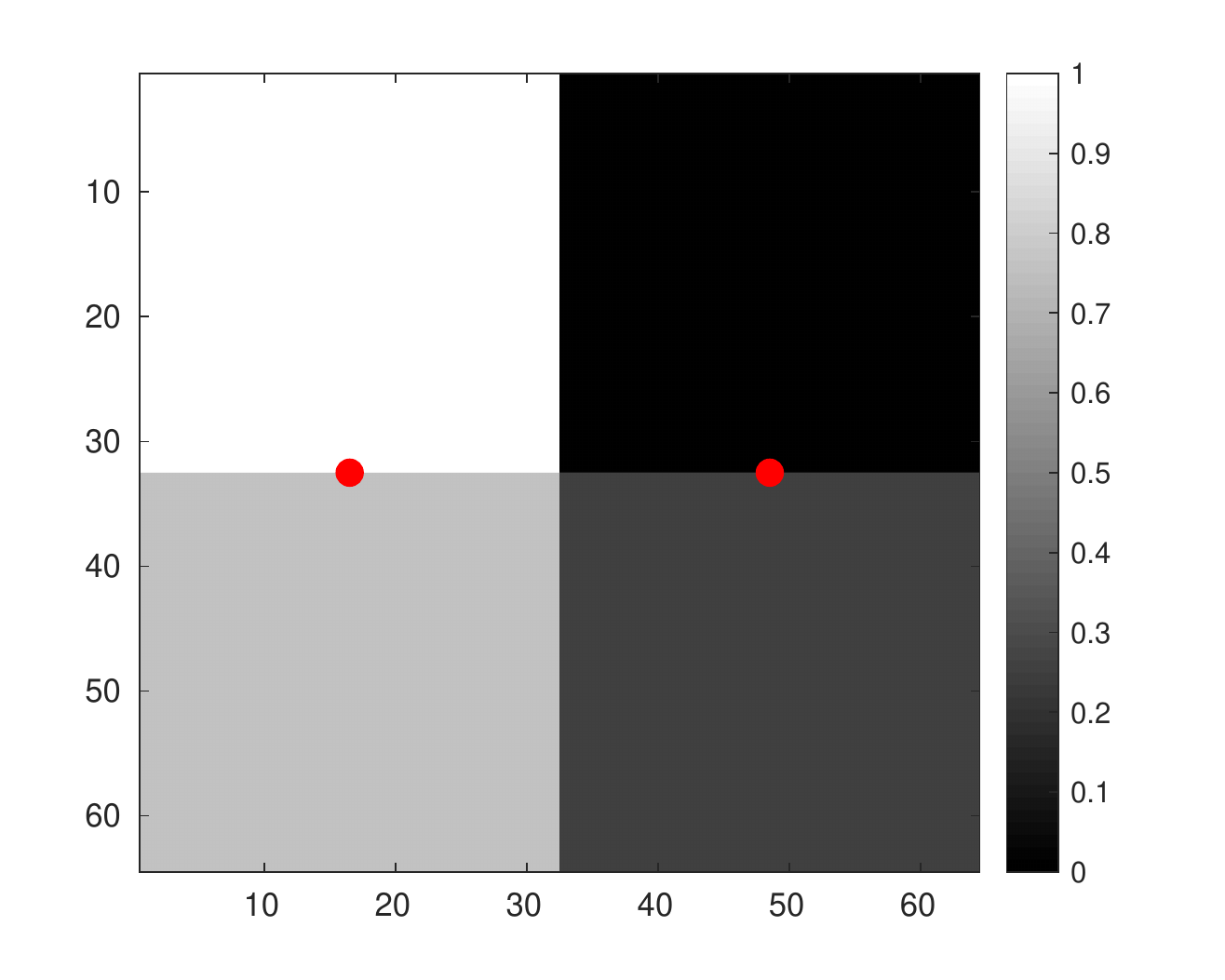}
	\includegraphics[width=0.32\textwidth]{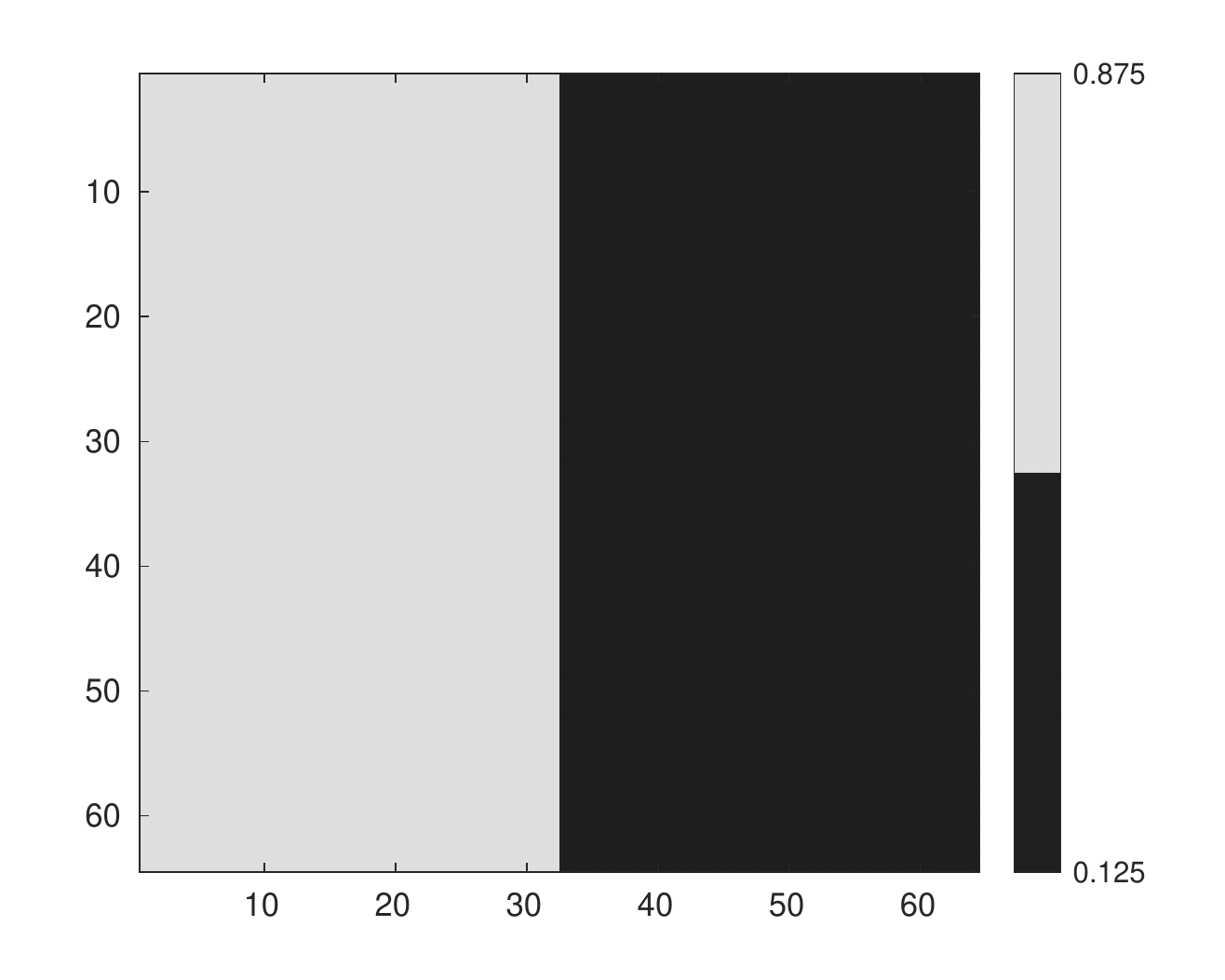}
	\caption{Left panel: initial image of $4096$ pixels with four regions with different gray intensity. Middle panel: red dots show the positions of the clusters at equilibrium. Right panel: segmentation of the initial image in two regions.\label{fig:imSegBenchmark}}
\end{figure}

In Figure~\ref{fig:imSegBenchmark} we first apply the segmentation to a benchmark test in order to show how the model works on a simple test. We initialize an gray scale image with $4096$ pixels with four regions of different intensity colors. The top and the bottom left corners have intensity $1$ and $0.75$, respectively. The top and the bottom right corners have intensity $0$ and $0.25$, respectively. See the left panel of Figure~\ref{fig:imSegBenchmark}. The positions of the pixels are rescaled on $\Omega=[0,1]^2$. The clustering algorithm is applied with the two bounded confidence values $\epsilon_1=0.5$ and $\epsilon_2=0.3$. At equilibrium the two red clusters showed in the middle panel of Figure~\ref{fig:imSegBenchmark} are computed by the kinetic model. The mean of the intensity colors of the pixels in each cluster is computed and mapped back to the initial position obtaining the segmentation in the right panel of Figure~\ref{fig:imSegBenchmark}. The two values of the gray intensity are $0.125$ and $0.875$.

\begin{figure}[t!]
	\centering
	\begin{subfigure}[b]{0.31\textwidth}
		\includegraphics[width=\textwidth]{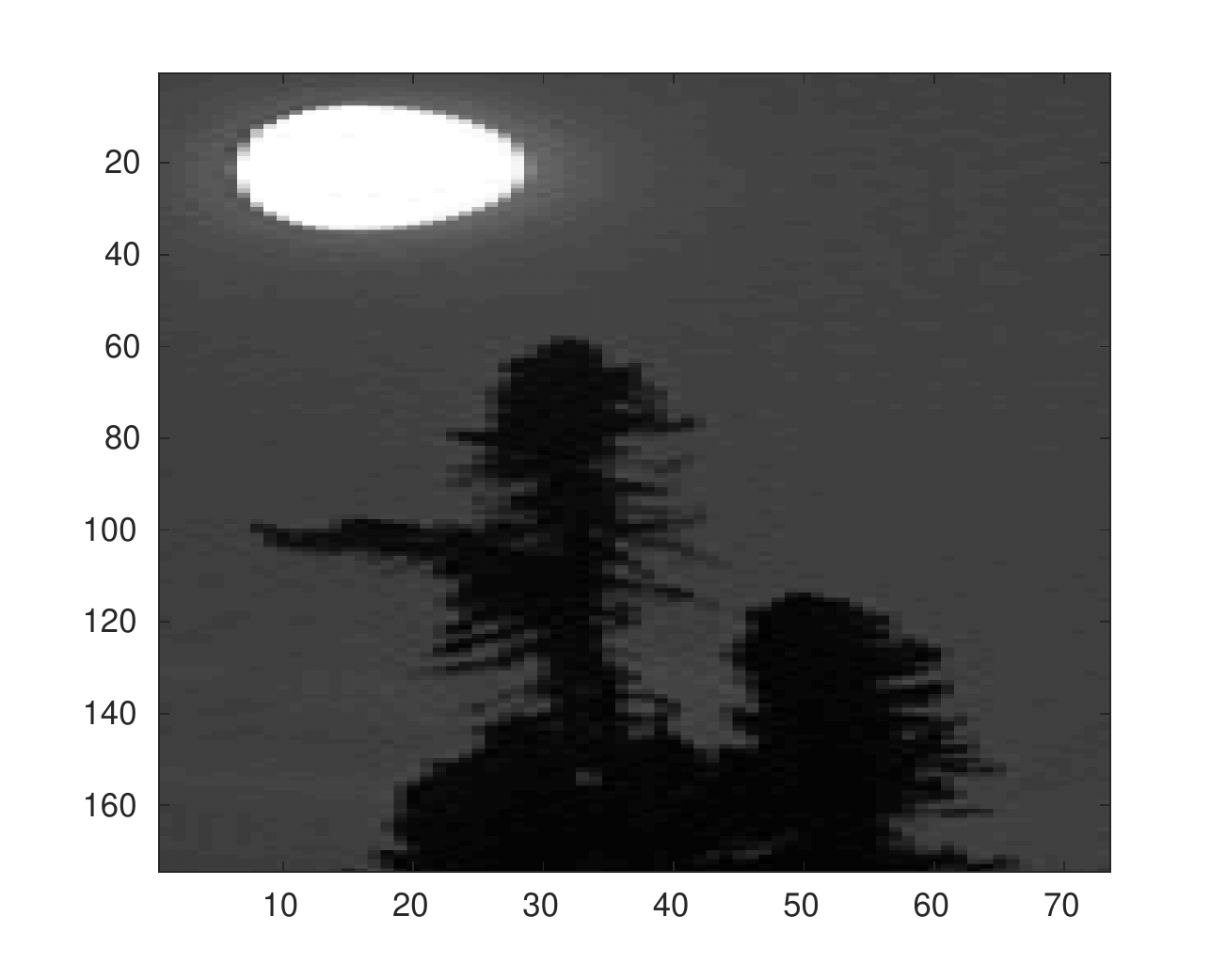}
		\caption{Initial image:\\$12702$ pixels}
	\end{subfigure}
	~ 
	\begin{subfigure}[b]{0.31\textwidth}
		\includegraphics[width=\textwidth]{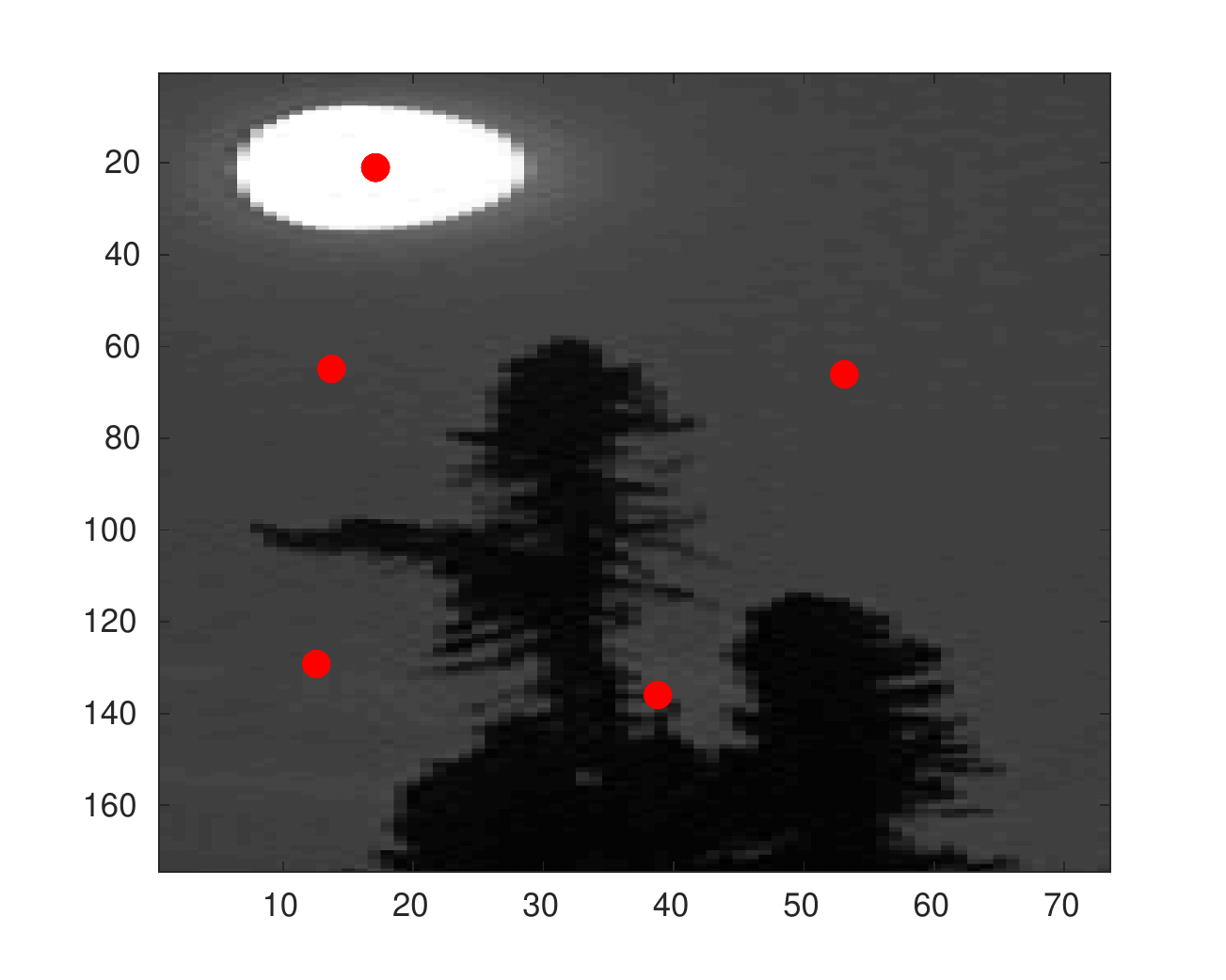}
		\caption{Clusters:\\ $\epsilon_1=0.2$, $\epsilon_2=0.1$}
	\end{subfigure}
	~ 
	\begin{subfigure}[b]{0.31\textwidth}
		\includegraphics[width=\textwidth]{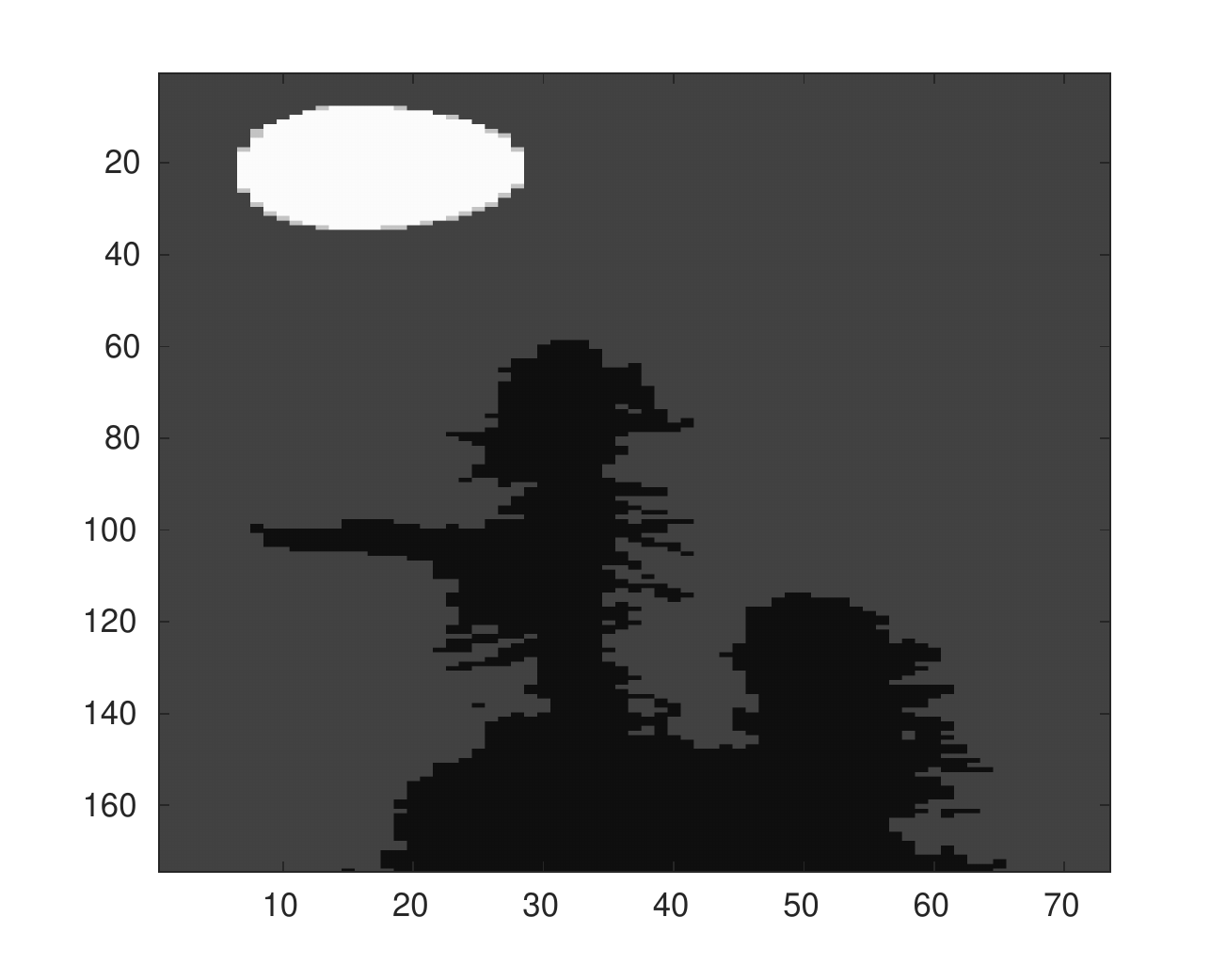}
		\caption{Segmentation:\\ $\epsilon_1=0.2$, $\epsilon_2=0.1$}
	\end{subfigure}
	\caption{Image segmentation of $174\times73$ gray scale image taken by the data--set~\cite{amfm_pami2011}.\label{fig:imSegNight}}
\end{figure}

We now apply the segmentation process to real images taken from different data--sets. In Figure~\ref{fig:imSegNight} we consider an image of $174\times73$ pixels showing a night sky, trees and the moon. On this example, the method is able to provide good results with a wide range of confidence values since the variation of the gray scale is sharp between the objects which are also characterized by regions having homogeneous colors. In Figure~\ref{fig:imSegNight} we report the result obtained with $\epsilon_1=0.2$ and $\epsilon_2=0.1$ which results in the formation of $5$ clusters.

\begin{figure}[t!]
	\centering
	\begin{subfigure}[b]{0.31\textwidth}
		\includegraphics[width=\textwidth]{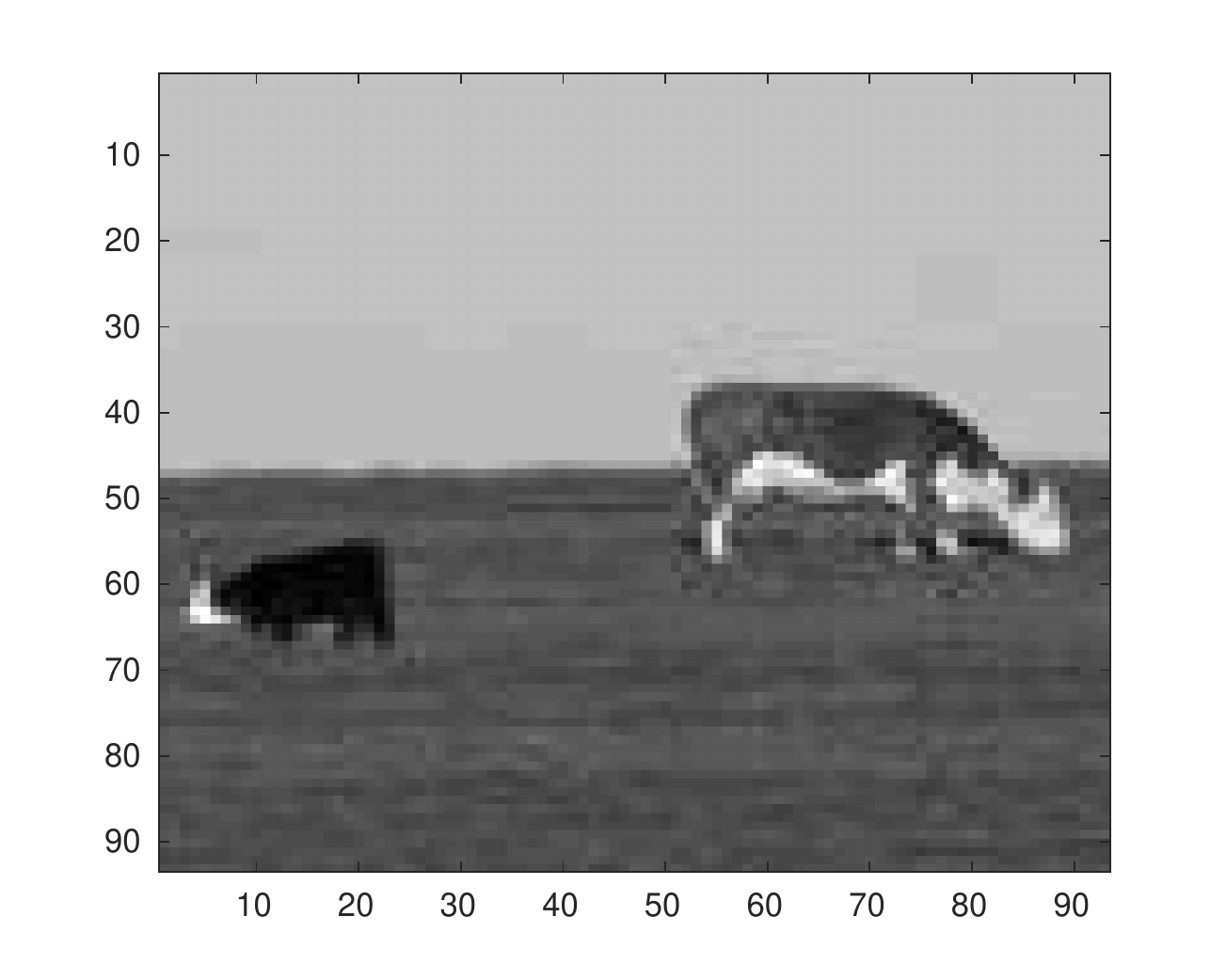}
		\caption{Initial image:\\$8649$ pixels}
	\end{subfigure}
	~ 
	\begin{subfigure}[b]{0.31\textwidth}
		\includegraphics[width=\textwidth]{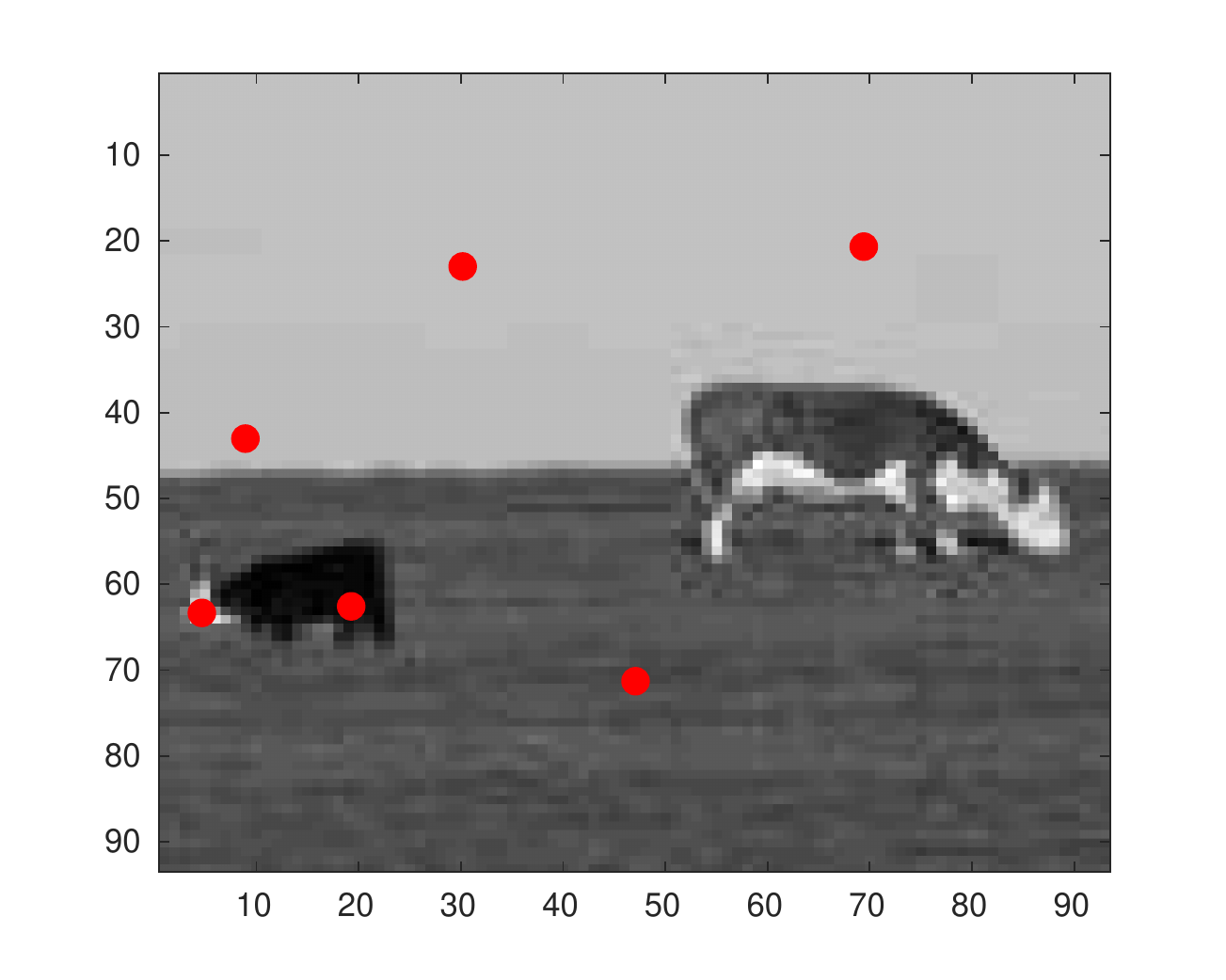}
		\caption{Clusters:\\ $\epsilon_1=0.2$, $\epsilon_2=0.05$}
	\end{subfigure}
	~ 
	\begin{subfigure}[b]{0.31\textwidth}
		\includegraphics[width=\textwidth]{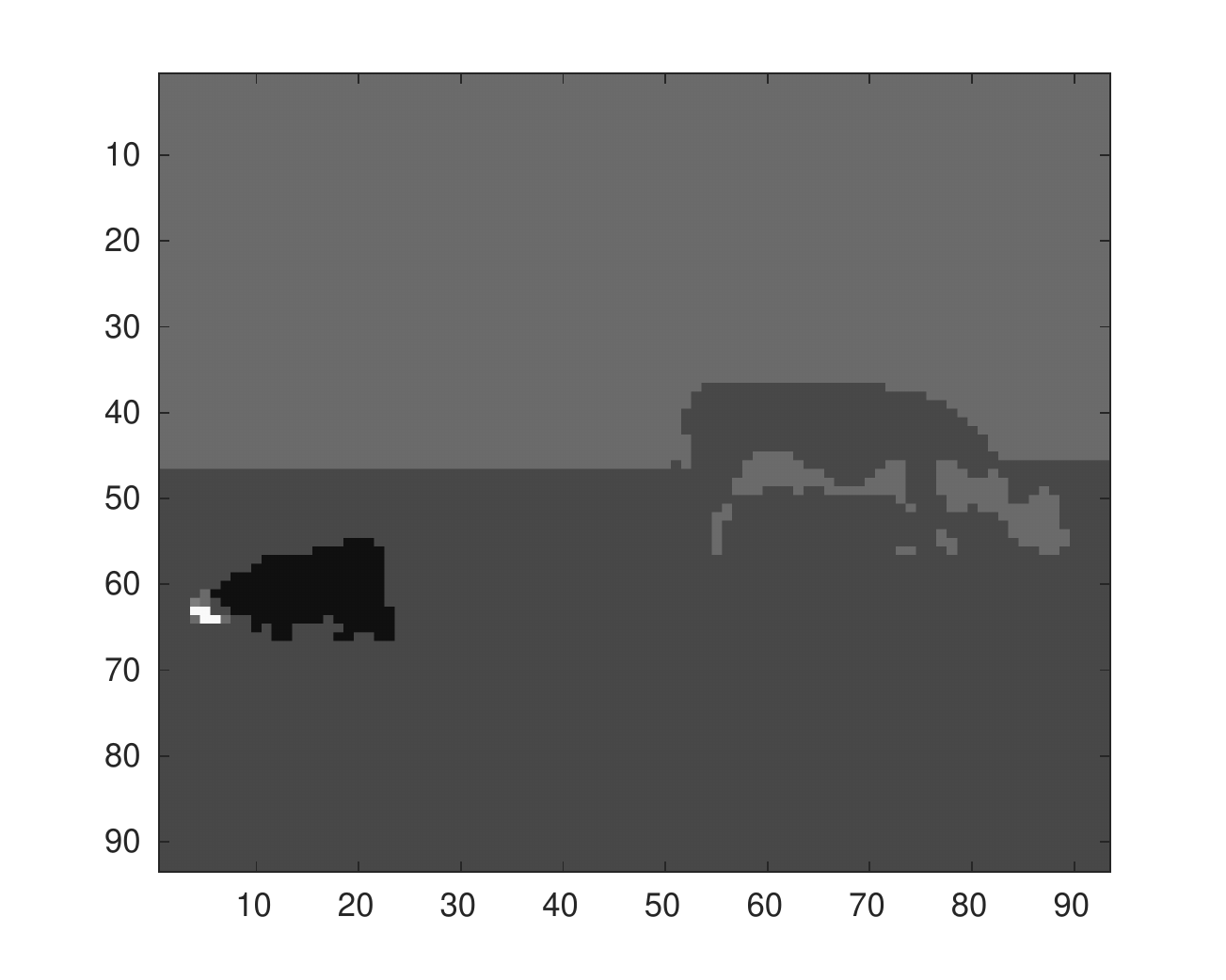}
		\caption{Segmentation:\\ $\epsilon_1=0.2$, $\epsilon_2=0.05$}
	\end{subfigure}
	\caption{Image segmentation of $93\times93$ gray scale image taken by the data--set~\cite{datasetICCV}.\label{fig:imSegCows}}
\end{figure}

The second image in Figure~\ref{fig:imSegCows} is also characterized by two backgrounds: the sky which is very homogeneous and the field which is less homogeneous. The goal of the segmentation process would be to identify the two cows, separating them from the background. We report the results obtained with $\epsilon_1=0.2$ and $\epsilon_2=0.05$, resulting in $6$ clusters. Observe that $\epsilon_2$ is taken small in order to avoid the formation of many clusters due to the low homogeneity of the field. The two cows are well identified by the method.

\begin{figure}[t!]
	\centering
	\begin{subfigure}[b]{0.31\textwidth}
		\includegraphics[width=\textwidth]{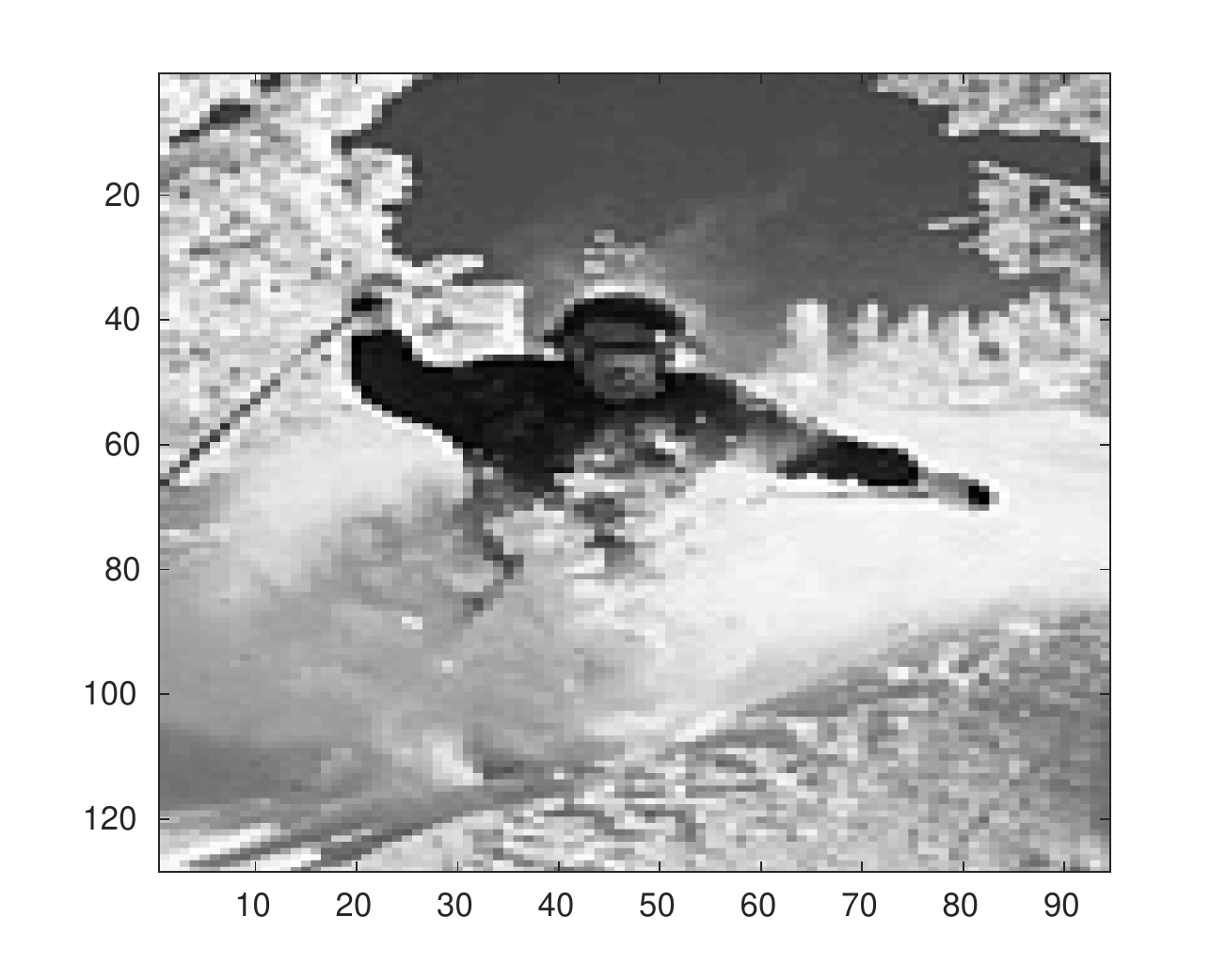}
		\caption{Initial image:\\$12032$ pixels}
	\end{subfigure}
	~ 
	\begin{subfigure}[b]{0.31\textwidth}
		\includegraphics[width=\textwidth]{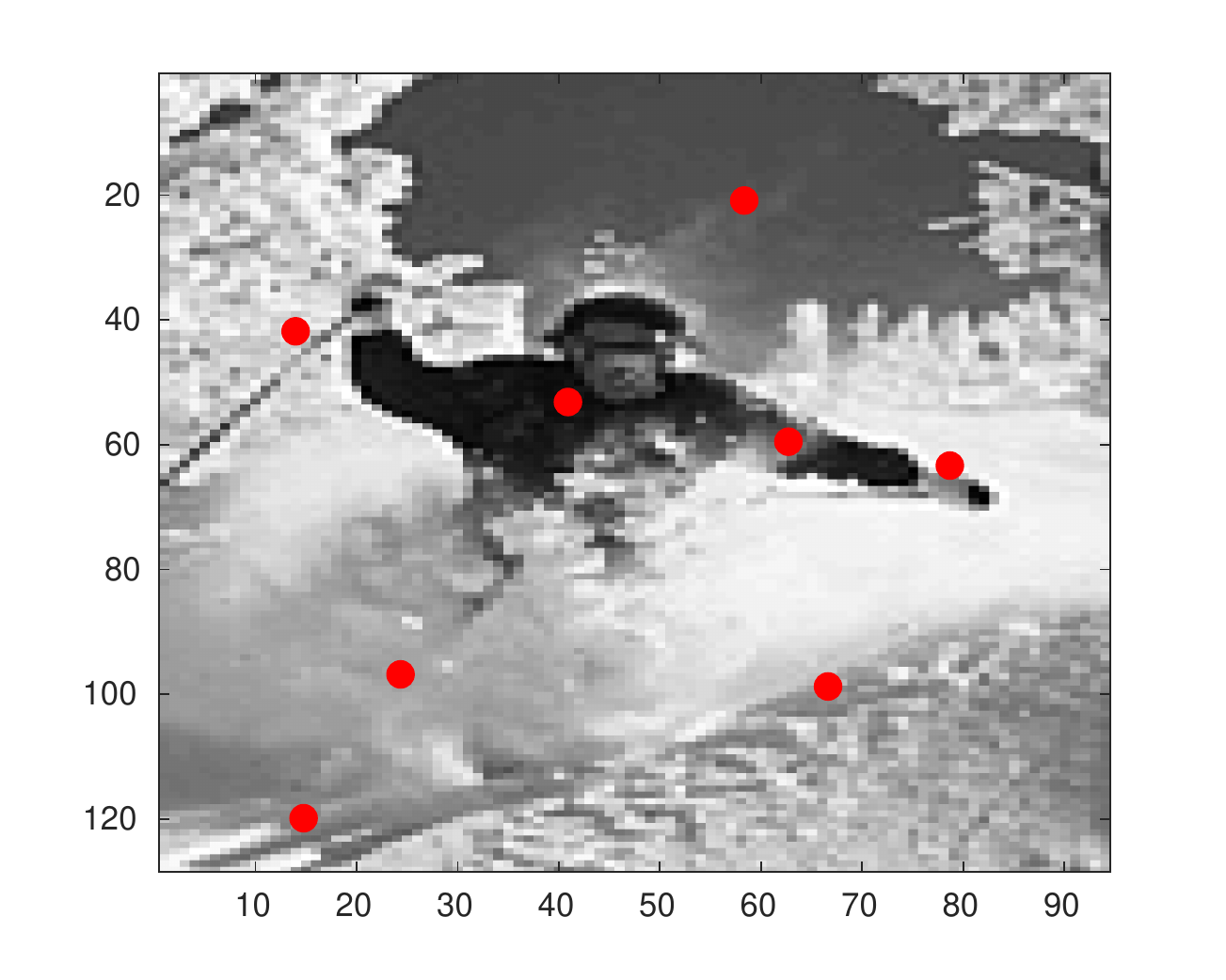}
		\caption{Clusters:\\ $\epsilon_1=0.15$, $\epsilon_2=0.15$}
	\end{subfigure}
	~ 
	\begin{subfigure}[b]{0.31\textwidth}
		\includegraphics[width=\textwidth]{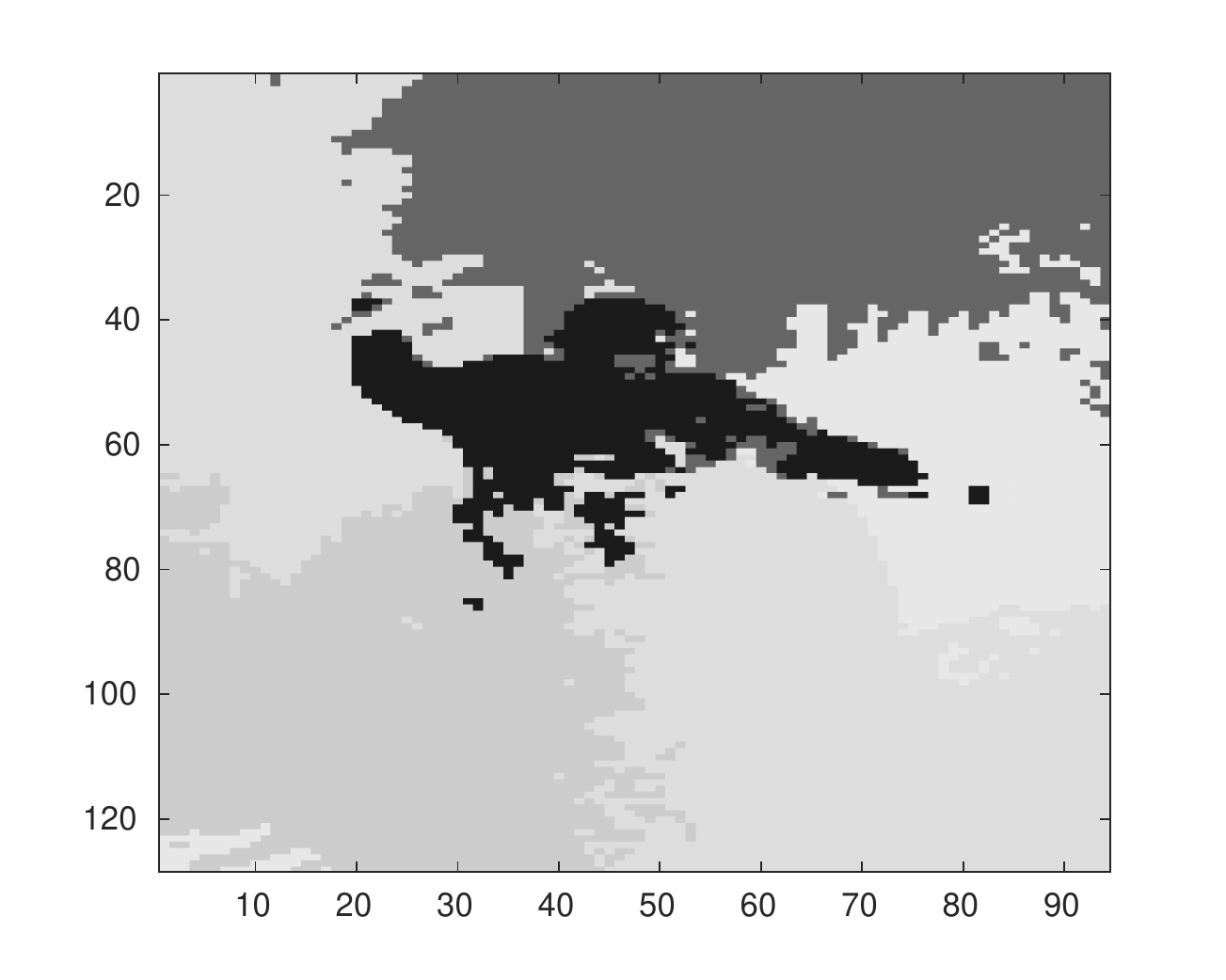}
		\caption{Segmentation:\\ $\epsilon_1=0.15$, $\epsilon_2=0.15$}
	\end{subfigure}
	\caption{Image segmentation of $128\times94$ gray scale image taken by the data--set~\cite{MartinFTM01}.\label{fig:imSegSnow}}
\end{figure}

In the case of Figure~\ref{fig:imSegSnow} the confidence level $\epsilon_2$ is taken larger in order to force clustering of regions with snow but different shadows. The clustering method identifies $8$ clusters and the skier is well separated by the background. 

\begin{figure}[t!]
	\centering
	\begin{subfigure}[b]{0.31\textwidth}
		\includegraphics[width=\textwidth]{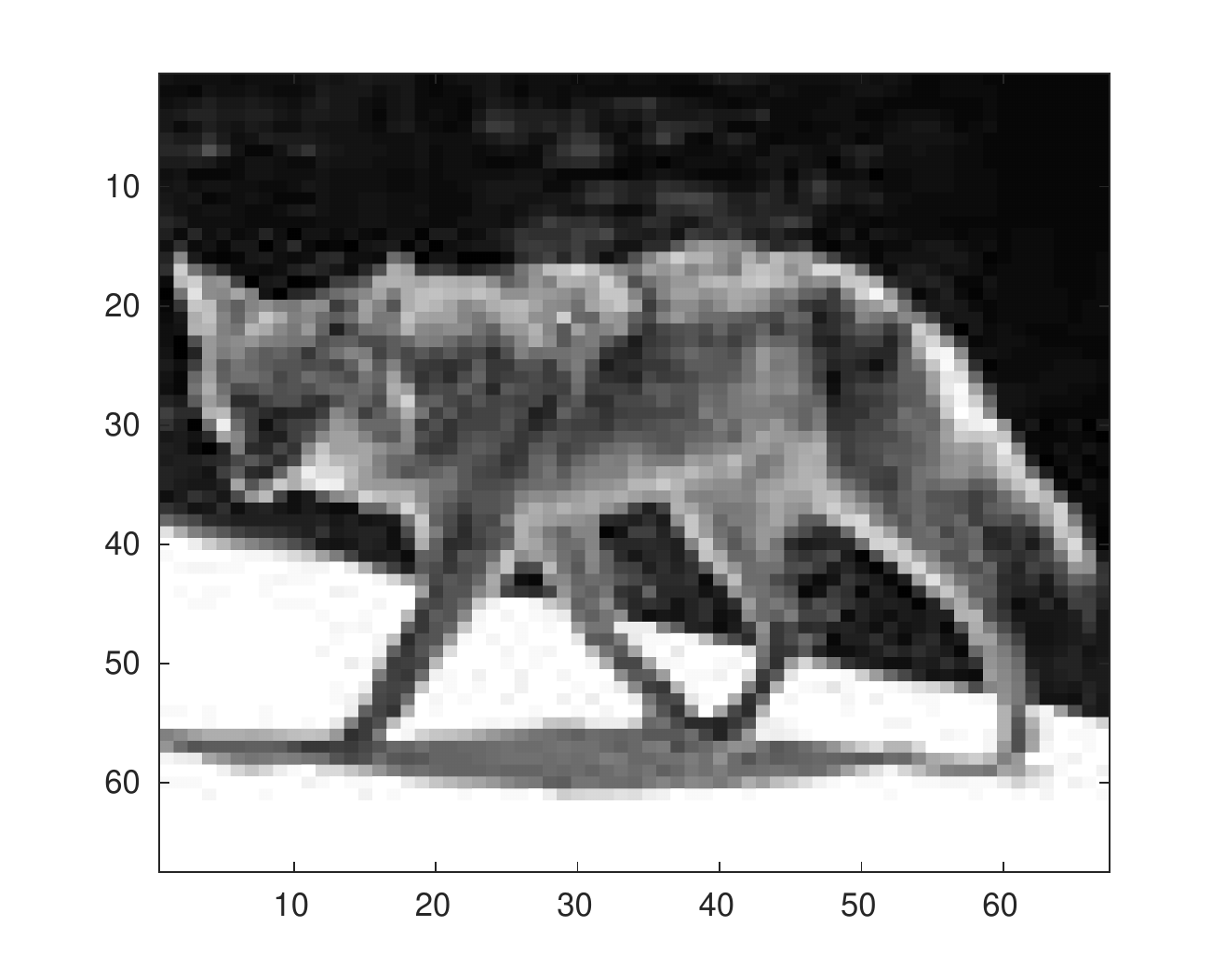}
		\caption{Initial image:\\$4489$ pixels}
	\end{subfigure}
	~ 
	\begin{subfigure}[b]{0.31\textwidth}
		\includegraphics[width=\textwidth]{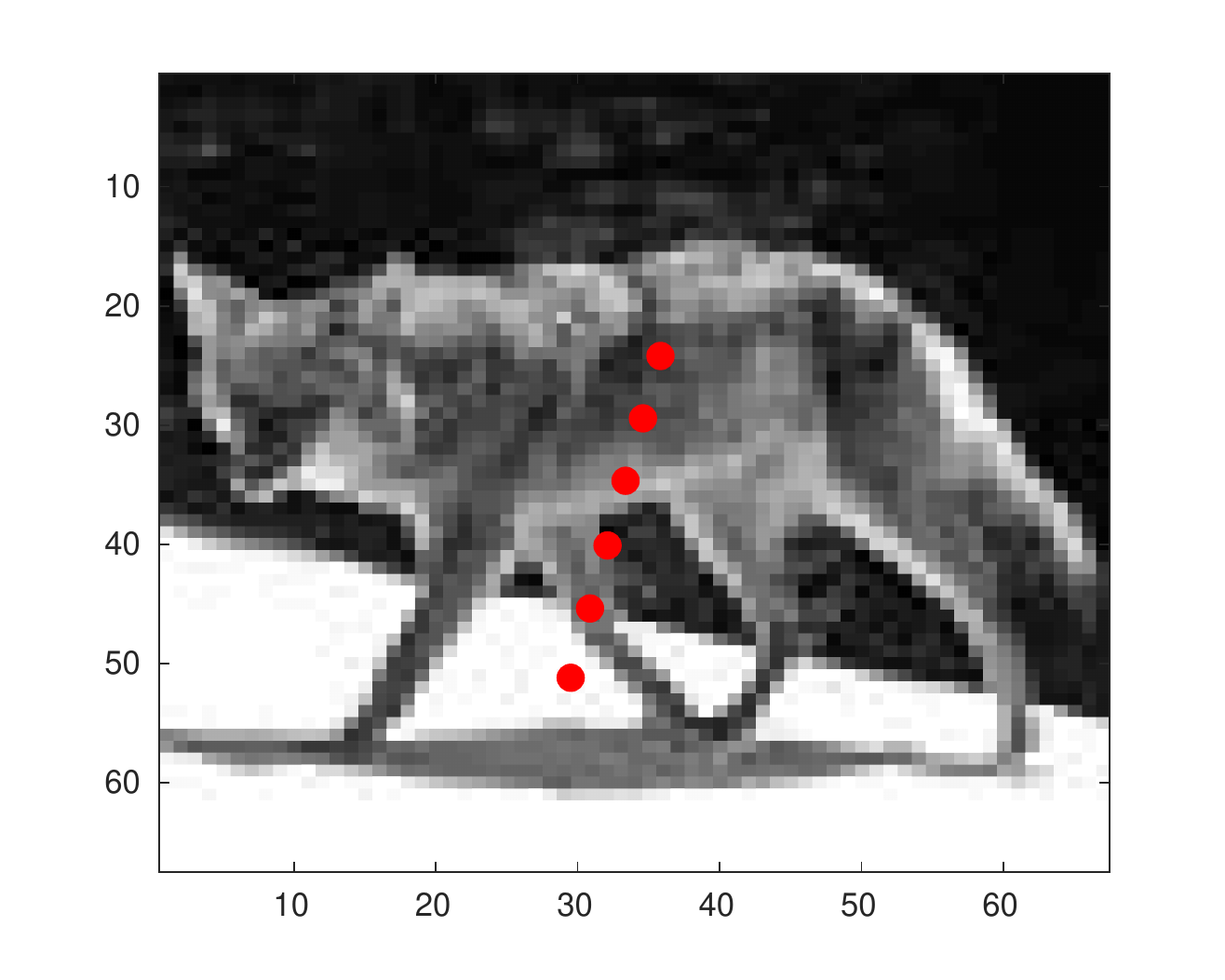}
		\caption{Clusters:\\ $\epsilon_1=0.3$, $\epsilon_2=0.1$}
	\end{subfigure}
	~ 
	\begin{subfigure}[b]{0.31\textwidth}
		\includegraphics[width=\textwidth]{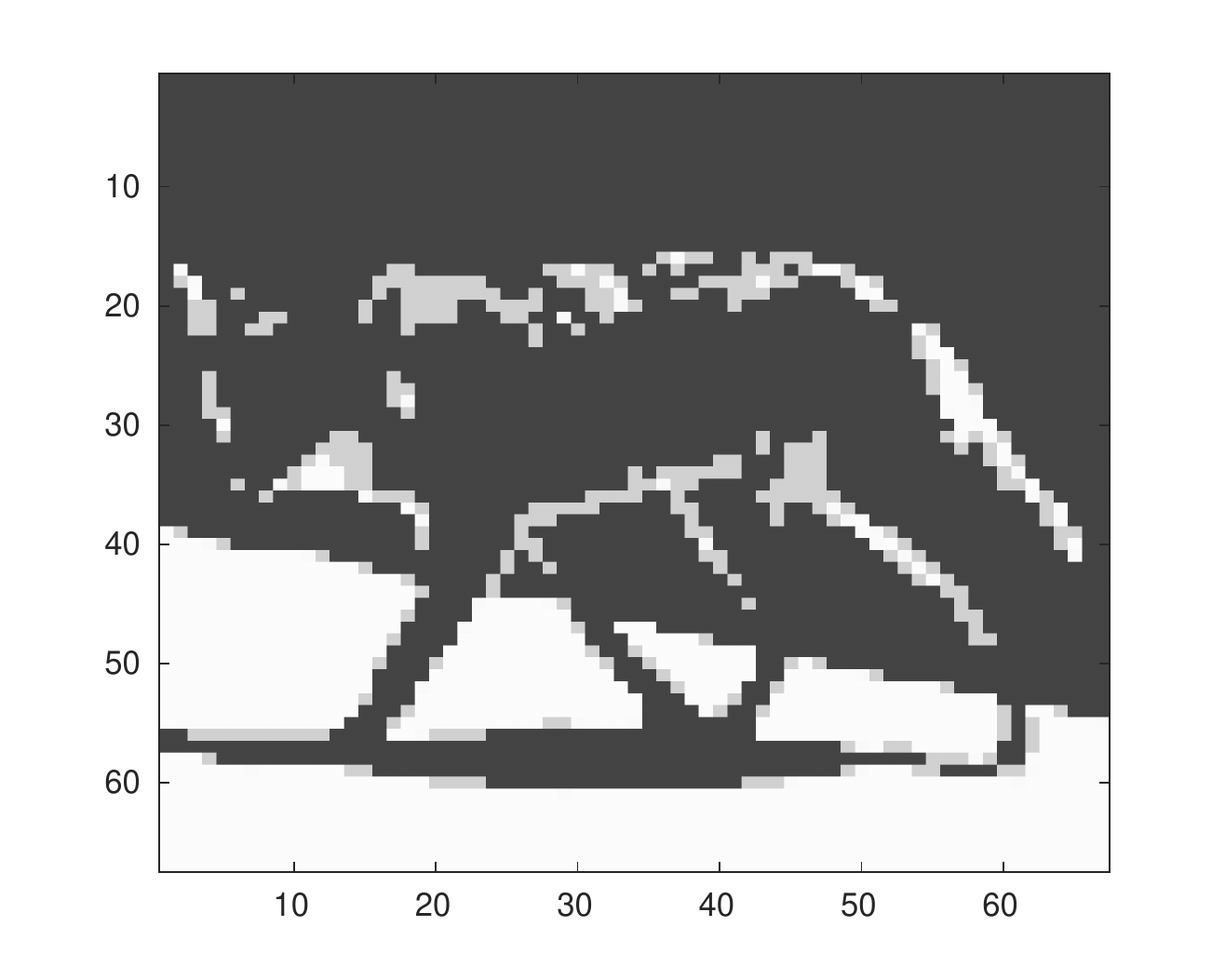}
		\caption{Segmentation:\\ $\epsilon_1=0.3$, $\epsilon_2=0.1$}
	\end{subfigure}
	~ 
	\begin{subfigure}[b]{0.31\textwidth}
		\includegraphics[width=\textwidth]{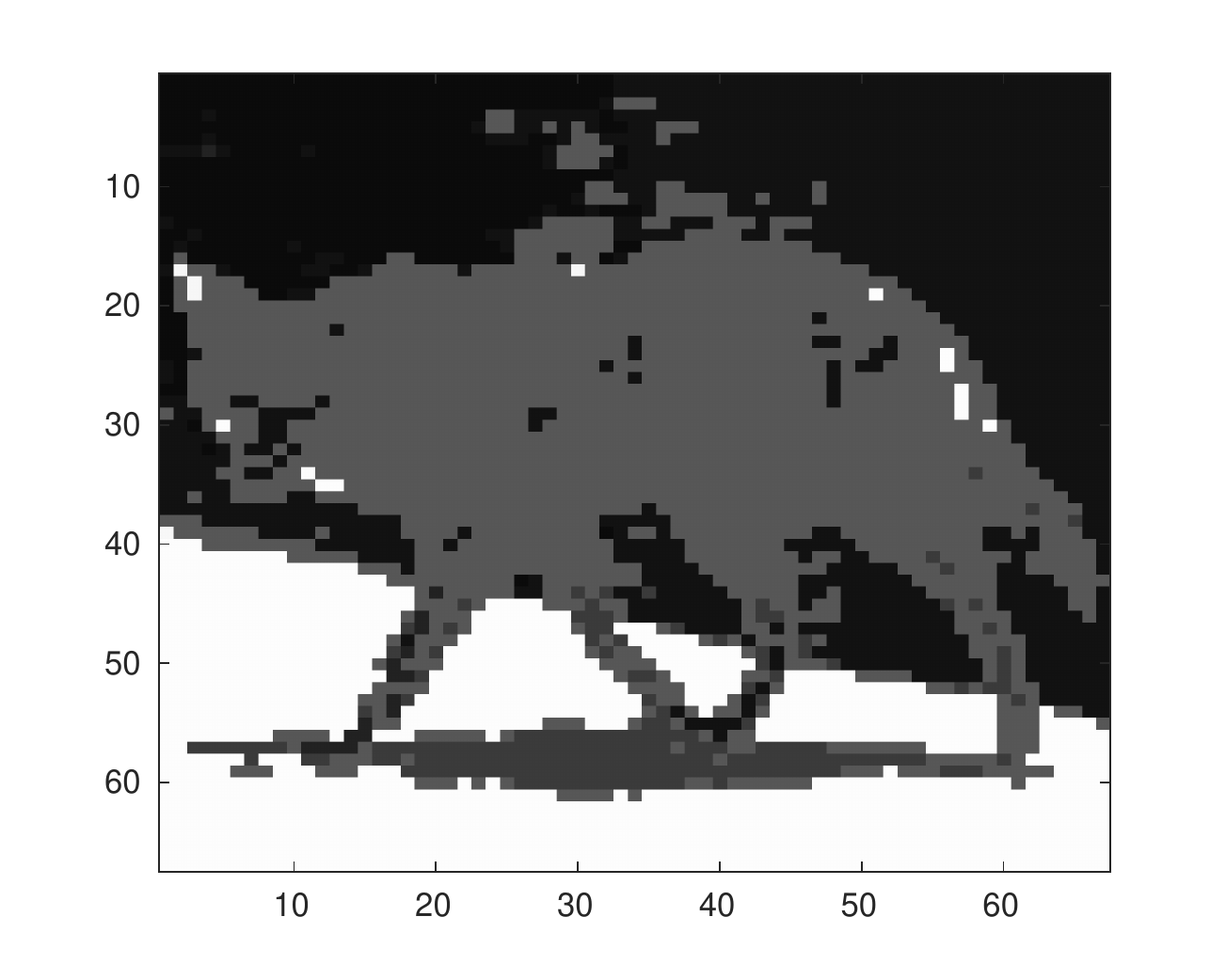}
		\caption{Segmentation:\\$\epsilon_1=0.2$, $\epsilon_2=0.1$}
	\end{subfigure}
	\begin{subfigure}[b]{0.31\textwidth}
		\includegraphics[width=\textwidth]{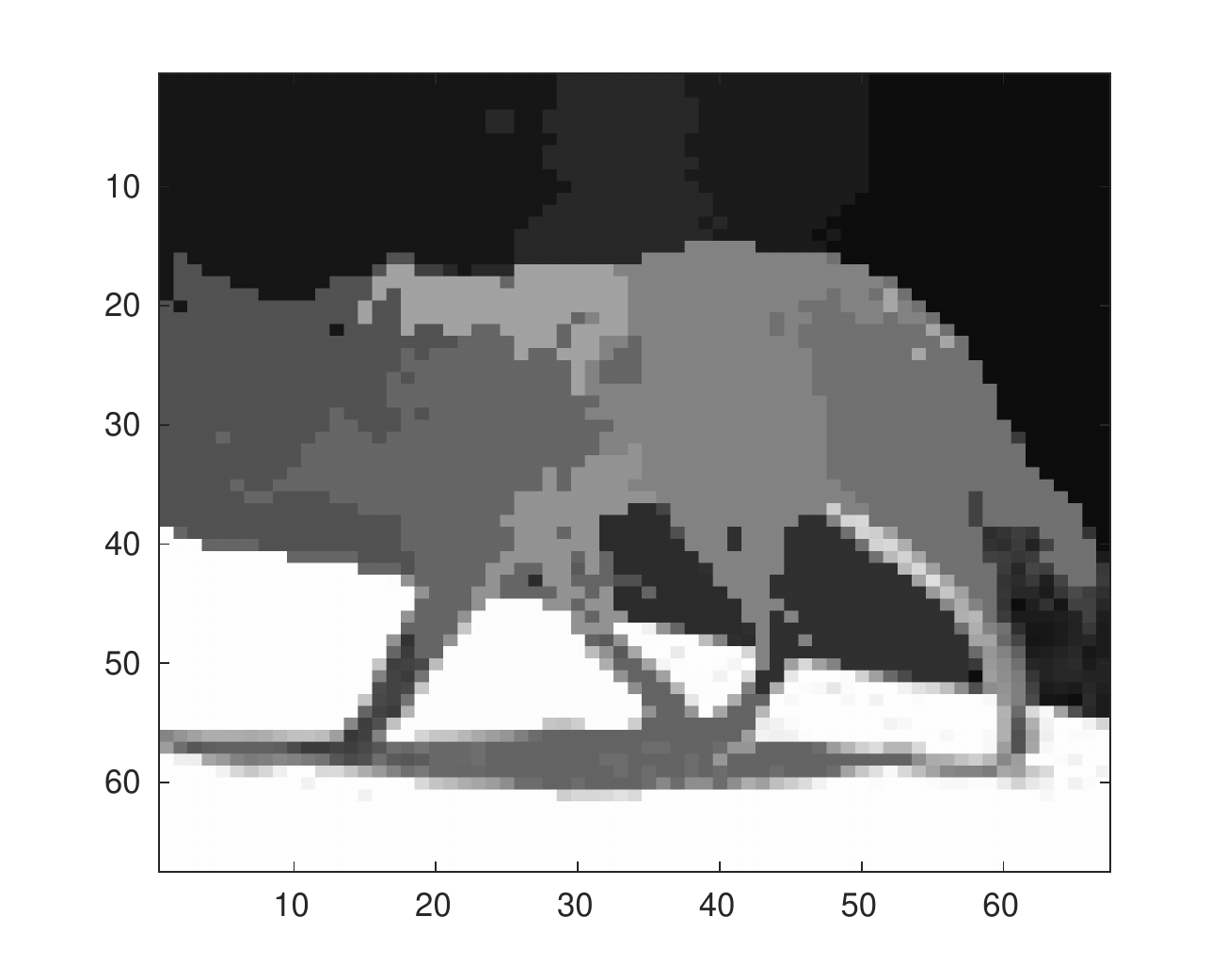}
		\caption{Segmentation:\\$\epsilon_1=0.1$, $\epsilon_2=0.1$}
	\end{subfigure}
	~ 
	\begin{subfigure}[b]{0.31\textwidth}
		\includegraphics[width=\textwidth]{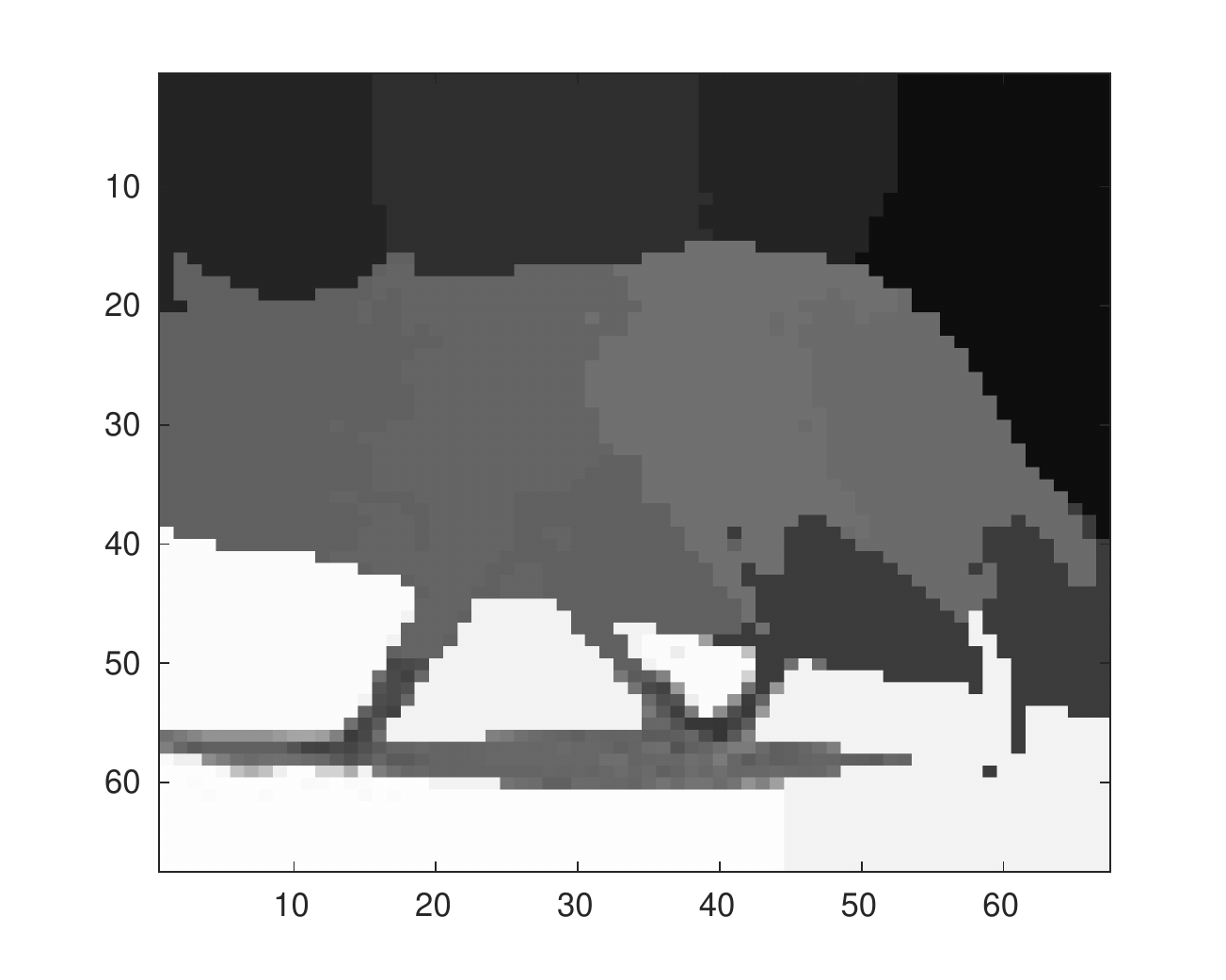}
		\caption{Segmentation:\\$\epsilon_1=0.1$, $\epsilon_2=0.2$}
	\end{subfigure}
	\caption{Image segmentation of $67\times67$ gray scale image taken by the data--set~\cite{MartinFTM01}.\label{fig:imSegWolf}}
\end{figure}

In Figure~\ref{fig:imSegWolf} we employ the kinetic model for clustering to segmentation of a real image taken by~\cite{MartinFTM01}. The true image is $67\times67$ pixels and several values of the bounded confidence levels $\epsilon_1$ and $\epsilon_2$ are applied in order to cluster with respect positions and intensity color of pixels. The best result is obtained with $\epsilon_1=0.3$ and $\epsilon_2=0.1$ which results in the formation of $8$ clusters. Lower values of $\epsilon_1$ allows to obtain more clusters and thus more gray intensity colors at equilibrium. Increasing $\epsilon_1$ requires to not take $\epsilon_2$ also large, otherwise many information are lost at equilibrium.

\begin{figure}[t!]
	\centering
	\begin{subfigure}[b]{0.31\textwidth}
		\includegraphics[width=\textwidth]{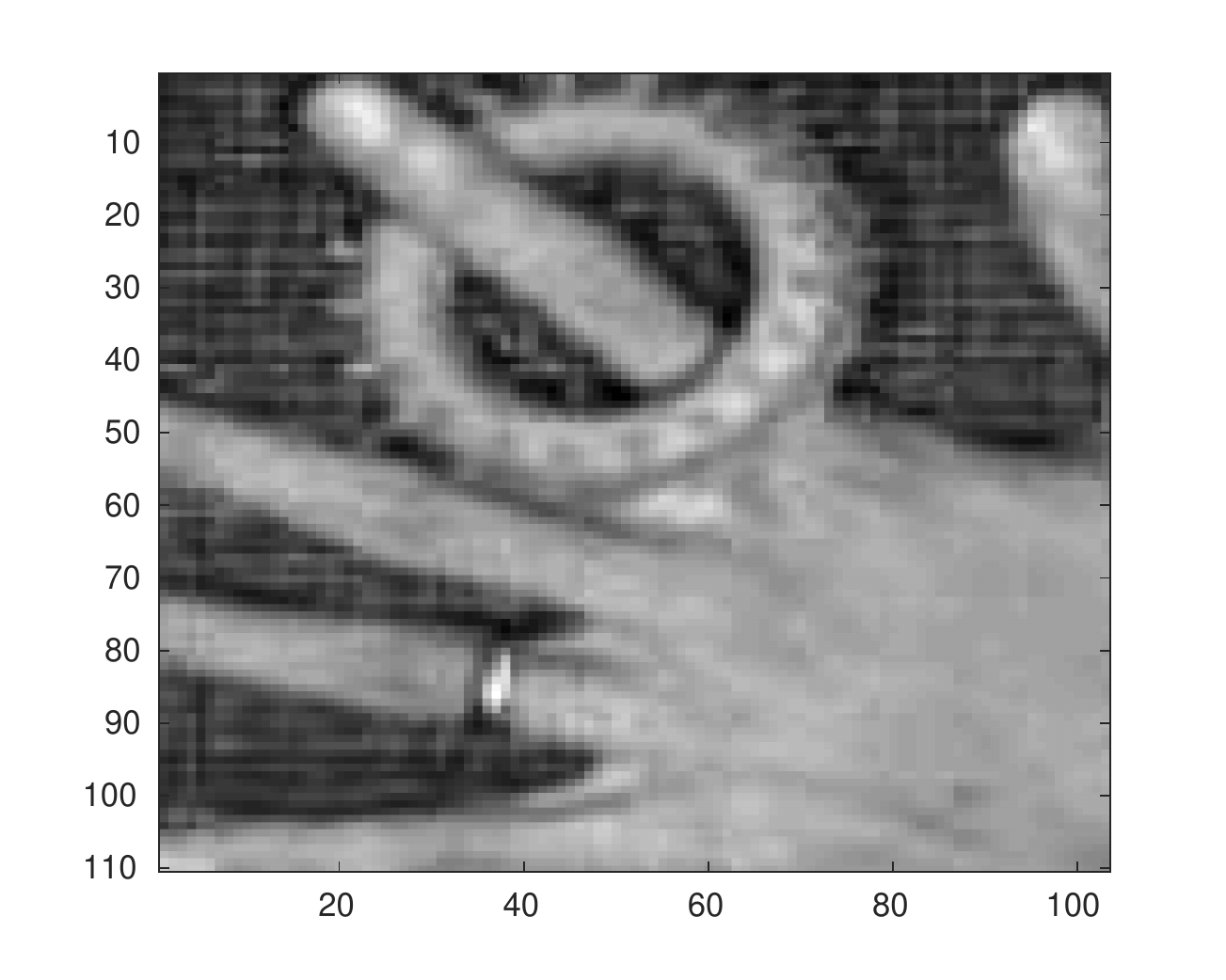}
		\caption{Initial image:\\$13992$ pixels}
	\end{subfigure}
	~ 
	\begin{subfigure}[b]{0.31\textwidth}
		\includegraphics[width=\textwidth]{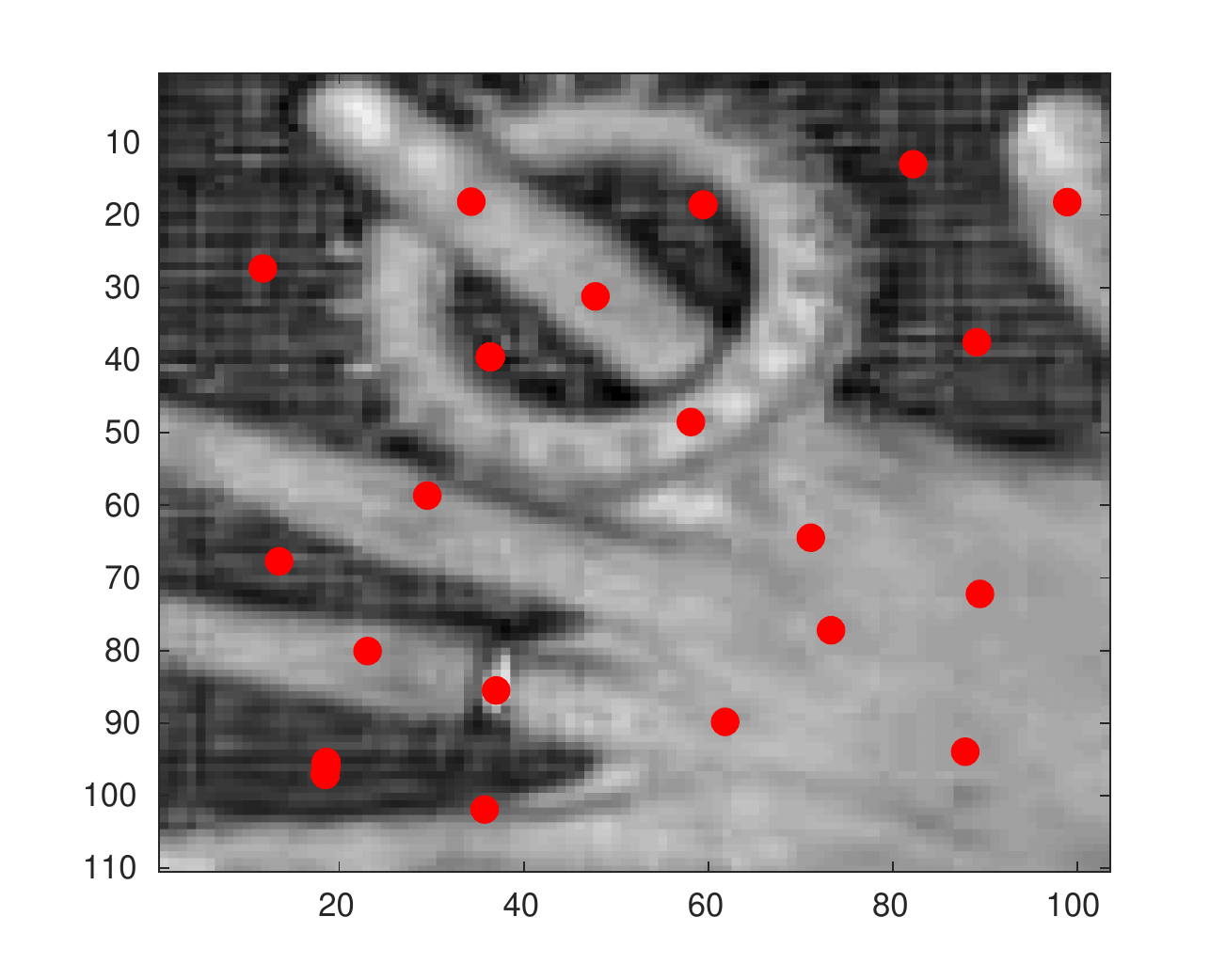}
		\caption{Clusters:\\ $\epsilon_1=0.05$, $\epsilon_2=0.10$}
	\end{subfigure}
	~ 
	\begin{subfigure}[b]{0.31\textwidth}
		\includegraphics[width=\textwidth]{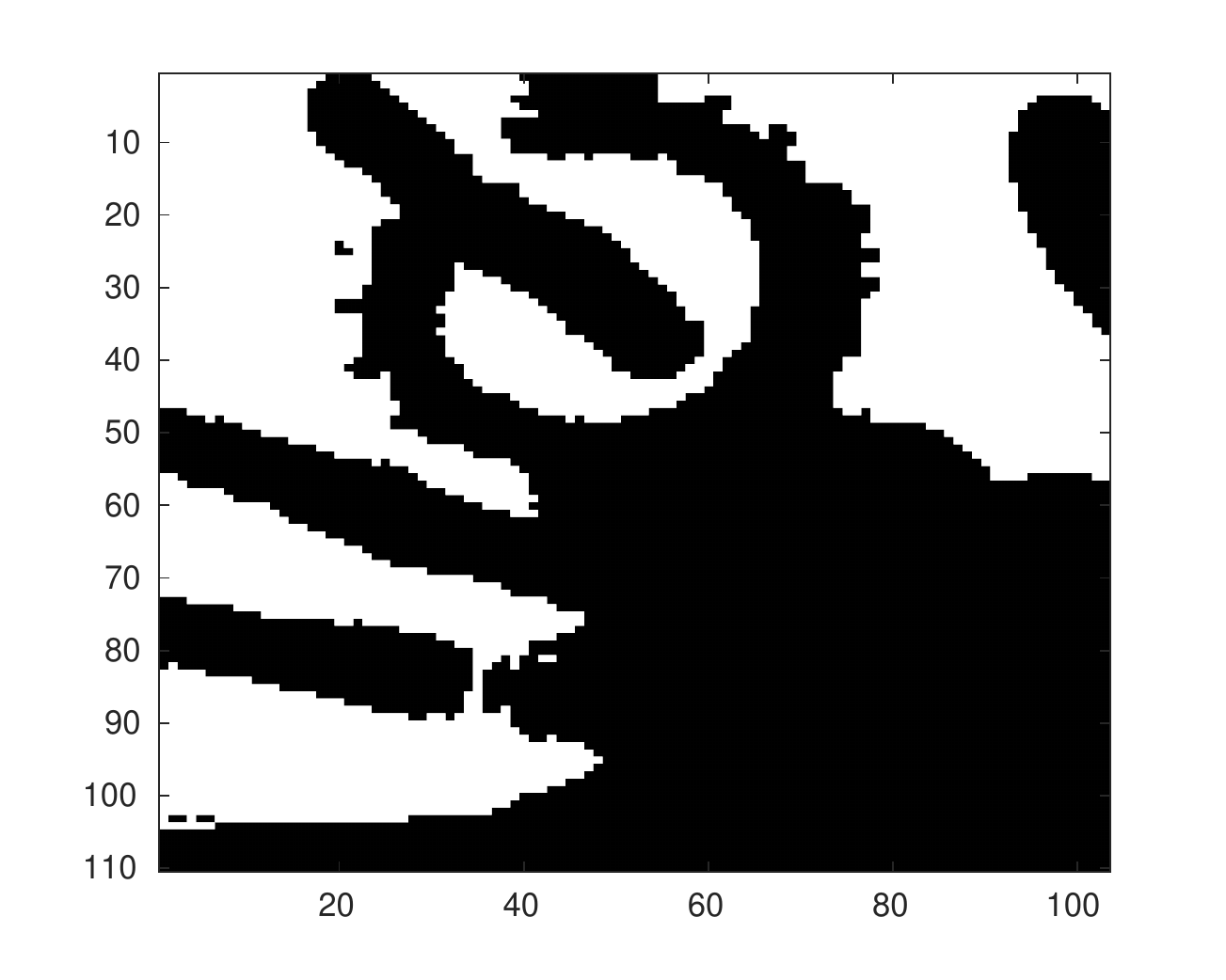}
		\caption{Segmentation:\\ $\epsilon_1=0.05$, $\epsilon_2=0.10$}
	\end{subfigure}
	\caption{Image segmentation of $132\times106$ gray scale image taken by the data--set~\cite{MartinFTM01}.\label{fig:imSegHand}}
\end{figure}

Finally, in Figure~\ref{fig:imSegHand} we present an additional experiment aimed to provide a comparison between our segmentation process and the one published in~\cite{2016Liuetal} based on the Kuramoto model. The image is selected from the data--set~\cite{MartinFTM01}. Right panel of Figure~\ref{fig:imSegHand} has to be compared with Figure (7c) in~\cite{2016Liuetal} and, in particular, we observe that our result is able to detect better some structures of the image that are lost by using the Kuramoto model.


\section{Conclusions} \label{sec:conclusion}

In this paper we have proposed a new method for clustering of large data which is based on a suitable generalization of the Hegselmann-Krause opinion dynamics model. The extension accounts for clustering with respect to two different sets of features of each single datum: one set describes time dependent characteristics, the other one represents static characteristics.

The mean--field limit of the particle model has been formally derived and the resulting kinetic equation allows for the study of mathematical properties. Time evolution of moments and asymptotic behavior of kinetic equation have been analytically characterized. Moreover, the derivation of the mean--field model allows for reduction of computational complexity thanks to a suitable random subset Monte Carlo algorithm.

Several applications to digital imaging have been proposed. In particular, we have focused on application to shape detection and image segmentation, showing the efficiency of the model to provide satisfying results. Finally, we emphasize that the present model has to be intended as a starting point towards more realistic applications, for example based on the use of non static features combined with suitable machine learning approaches. 

\section*{Acknowledgments}
The authors would like to thank the German Research Foundation (DFG) for the kind support within the Cluster of Excellence ``Internet of Production'' (IoP ID390621612).
Lorenzo Pareschi and Giuseppe Visconti acknowledge the support of the ``National Group for Scientific Computation (GNCS-INDAM)'', project ''Numerical approximation of hyperbolic problems and applications''.

\bibliographystyle{plain}
\bibliography{referencesBD}

\end{document}